\documentclass[10pt,reqno]{amsart}

\usepackage{graphicx}
\usepackage{times}
\usepackage{xcolor}
\usepackage[colorlinks=true,linkcolor=blue,citecolor=blue]{hyperref}%

\newtheorem{theorem}{Theorem}[section]
\newtheorem{lemma}[theorem]{Lemma}

\newtheorem{prop}[theorem]{Proposition}

\theoremstyle{definition}
\newtheorem{definition}[theorem]{Definition}
\theoremstyle{remark}
\newtheorem{remark}[theorem]{Remark}
\numberwithin{equation}{section}

\usepackage[top=1.25in, bottom=1.25in, left=1.25in, right=1.25in]{geometry}

\usepackage[scr]{rsfso}
\usepackage[english]{babel}
\usepackage{amsmath}
\usepackage{amsthm}
\usepackage{amssymb}
\usepackage{eucal}
\usepackage{comment}
\usepackage[integrals]{wasysym}

\newcommand\nc{\newcommand}
\nc{\E}{\mathbf{E}}
\nc{\R}{\mathbb R}
\nc{\C}{\mathbb C}
\nc{\Q}{\mathbb Q}
\nc{\Z}{\mathbb Z}
\nc{\wt}{\widetilde}
\nc{\rnc}{\renewcommand}
\nc{\e}{\varepsilon}
\nc{\grad}{\nabla}

\nc{\sck}{\mathrm{k}}
\nc{\bUz}{\mathbf{U}
}
\nc{\abbr}[1]{{\sc\lowercase{#1}}}

\rnc{\leq}{\leqslant}
\rnc{\geq}{\geqslant}
\rnc{\d}{\mathrm{d}}
\newenvironment{nouppercase}{%
  \renewcommand{\uppercasenonmath}[1]{}}{}
\title{\Large A flow-type scaling limit for random growth with memory}

\author{Amir Dembo and Kevin Yang}

\usepackage{setspace}
\begin{document}
\setstretch{0.99}
\begin{nouppercase}
\maketitle
\end{nouppercase}
\begin{abstract}
We study a stochastic Laplacian growth model, where a set $\mathbf{U}\subseteq\R^{\d}$ grows according to a reflecting Brownian motion in $\mathbf{U}$ stopped at level sets of its boundary local time. We derive a scaling limit for the leading-order behavior of the growing boundary (i.e. ``interface"). It is given by a geometric flow-type \abbr{PDE}. It is obtained by an averaging principle for the reflecting Brownian motion. We also show that this geometric flow-type \abbr{PDE} is locally well-posed, and its blow-up times correspond to changes in the diffeomorphism class of the growth model. Our results extend those of \cite{DGHS}, which restricts to star-shaped growth domains and radially outwards growth, so that in polar coordinates, the geometric flow transforms into a simple \abbr{ODE} with infinite lifetime. Also, we remove the ``separation of scales" assumption that was taken in \cite{DGHS}; this forces us to understand the local geometry of the growing interface.
\end{abstract}

\section{Introduction}
Random growth processes driven by diffusive particles are ubiquitous models for deposition, network dynamics, and various processes in biology, such as diffusion limited aggregation (\abbr{DLA}) \cite{WS}, the dielectric breakdown model \cite{NPW}, and the Hastings-Levitov model \cite{HL}. In a nutshell, the growth evolves according to the harmonic measure on its boundary (the ``interface") with respect to a fixed source. For example, in \abbr{DLA}, one grows a subset of $\Z^{\d}$ by sequentially sampling independent random walks starting at $\infty$ and adding whatever boundary point of the growing set that the random walks hit first.

Generally speaking, stochastic Laplacian growth models are difficult to analyze mathematically in part due to the complicated geometry of the interface. One exception is internal diffusion limited aggregation (\abbr{IDLA}), introduced in \cite{MD} to describe chemical deposition and in \cite{DF} from an entirely different algebraic perspective. This model is \abbr{DLA}, but the source is a fixed point in the interior of the growing set. This makes the interface much smoother than in \abbr{DLA}. Indeed, in \cite{LBG}, it was shown that the leading order behavior of the interface is spherical. Similar results were shown for variations of \abbr{IDLA} in \cite{BDCKL,GQ,LL,LP}. (Strictly speaking, in \cite{LP}, models beyond \abbr{IDLA} are also considered, and the scaling limit is determined by a \abbr{PDE} free boundary problem. Similarly, a multi-particle \abbr{IDLA} is related to a Stefan free boundary problem in \cite{GQ}.) However, there are many other types of stochastic Laplacian growth models of high interest. These include excited and reinforced random walks; see \cite{BW} for more background. Roughly speaking, in these models, every time the diffusive particle hits the interface, rather than resampling a new particle from a fixed source, one resamples the particle from a source determined by the previous hit location. For example, one can consider a random walk on a growing domain that reflects towards the origin whenever it hits the boundary; this would give the origin-excited random walk (\abbr{OERW}). In this case, the random walk can explore the local geometry of the boundary (near whatever boundary point it is currently located at), and the interface can also develop more interesting geometries than just the sphere. It is still expected that models like \abbr{OERW} have a limiting shape in the long-time limit, but there is no clear path for trying to tackle this problem \cite{K12}.

A key difficulty in analyzing models of this type is the non-Markovian nature of both the marginal interface process, and the marginal particle process. Nevertheless, \cite{DGHS} successfully studies a continuum version, where one takes a smooth subset in $\R^{\d}$, runs a Brownian motion inside the submanifold until it hits the boundary, adds a ``small" bump (say of volume $\e$) to the submanifold at that point, pushes the Brownian motion back towards the origin by a ``large" distance of order $1$, and repeats. (Actually, \cite{DGHS} studies a more general setup, where Brownian motion can be replaced by any Markov process satisfying ergodic and regularity properties for its transition kernel, and push-back towards the origin can be replaced by any map satisfying a ``separation of scales" condition that we explain shortly.) Precisely, in \cite{DGHS}, the authors compute the $\e\to0$ scaling limit for the random growth model; this limit is described by an explicit and simple \abbr{ODE} (we explain this shortly). Outside of \cite{DGHS}, let us also mention \cite{VST,NT,NST0,NST}. These study non-isotropic and off-lattice models. However, they restrict to growth in dimension $2$ (e.g. because of the use of conformal maps).

Our main interest is to push \cite{DGHS} further, specifically the following two issues left open in that paper.
\begin{enumerate}
\item The growth models in \cite{DGHS} are defined to grow radially outward and be smooth, star-shaped domains in $\R^{\d}$ (i.e. their boundaries are images of the $(\d-1)$-dimensional sphere under a real-valued map). This rules out seeing any interesting geometry, e.g. changes in diffeomorphism class {{}or Riemannian structure}. (We also mention other growth models of high interest, such as the Eden model \cite{Eden} and \abbr{DLA} \cite{WS}, for which a main difficulty is the complicated geometry that can develop at the microscopic level.) It also gives a canonical choice of coordinates that one can use globally to analyze the growth and particle processes. For instance, the geometry is captured by a real-valued function. Therefore, it is essentially one-dimensional (hence the simple differential equation derived in \cite{DGHS}).
\item In \cite{DGHS}, the push-back map acting on the particle when it hits the boundary is macroscopic; it does not depend on the scaling parameter $\e$. In particular, \cite{DGHS} needs a priori separation of scales between the size of the added bumps and this map. This assumption allows one to avoid dealing with a singular harmonic measure for the next step in the process, which would require geometric inputs. The ``push-back" map essentially regularizes this harmonic measure automatically, allowing \cite{DGHS} to proceed with rather abstract (analytic) methods and without the need for any geometry.
\end{enumerate}
We now briefly explain the content of this paper more precisely. The growth process that we consider can be described as follows. Take a compact, connected subset $\bUz\subseteq\R^{\d}$ with smooth boundary $\partial\bUz$. (Here, $\d\geq1$ is a fixed dimension.) Next, let $t\mapsto\mathbf{w}(t)\in\bUz$ denote a standard Brownian motion with unit normal reflection on $\partial\bUz$. (The exact reflection rule is not so important; the boundary vector field determining the reflection direction just has to be smooth and never parallel to $\partial\bUz$.) Let $\e>0$ be a ``microscopic scale" that we take to zero to get a scaling limit. The growth model is given by the following (see Figure \ref{figure:1}).
\begin{enumerate}
\item Suppose $\mathbf{w}(0)\in\partial\bUz$. According to a Poisson clock of speed $\e^{-1}$, sample a random point $\mathbf{z}\in\partial\bUz$ that is distributed as $\mathbf{w}(\tau)$, where $\tau$ is the first time that the boundary local time of $\mathbf{w}(\cdot)$ equals $\delta=\delta(\e)>0$ (some parameter to be determined). Add to $\bUz$ a smooth ``bump" (or subset) of volume $\e$ centered at the point $\mathbf{z}$. (Since $\mathbf{z}\in\partial\bUz$, this represents a small ``growth" of $\bUz$.) 
\begin{itemize}
\item Using local time to sample the next point for the particle, {{}as opposed to the push-back procedure in \cite{DGHS}}, is just a convenient (and analytically nice) choice of regularization of harmonic measure. Here, $\delta$ is analogous to the length-scale of the ``push-back map" mentioned earlier. The point is that we allow it to depend on $\e$.
\end{itemize}
\item Repeat, but start $\mathbf{w}$ at the random point $\mathbf{z}$ from above, and replace $\bUz$ by its augmentation.
\end{enumerate}
\begin{figure}[h!]
\includegraphics[width=0.3\textwidth]{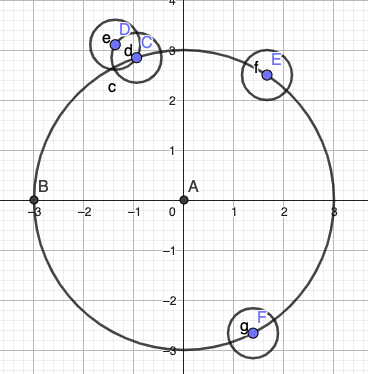}
\caption{The model of interest when the initial set $\bUz$ is a circle centered at point $\mathbf{A}$. Here, $\mathbf{B}$ (on the left) is $\mathbf{w}(0)$. The point $\mathbf{C}$ is sampled by starting reflecting Brownian motion $\mathbf{B}$ and waiting until its boundary local time equals $\delta(\e)$. At $\mathbf{C}$, we augment the circle by a small ball. The point $\mathbf{D}$ is sampled by starting reflecting Brownian motion on the augmented set at $\mathbf{C}$, and waiting for the boundary local time to hit $\delta(\e)$. Then, augment the set by adding a circle centered at $\mathbf{D}$. Iterate to sample $\mathbf{E}$ then $\mathbf{F}$ (with their respective circles). It is likely that if $\delta(\e)\ll1$, the $(k+1)$-th sampled point will be close to the $k$-th sampled point. But Brownian excursion theory says that this probability is not too close to $1$; this is why $\mathbf{E}$ is far from $\mathbf{D}$, for example.}
\label{figure:1}
\end{figure}
This model is similar to that in \cite{BBCH}. Both are reflecting Brownian motions with respect to a metric that depends on its boundary local time. In \cite{BBCH}, the goal is computing the invariant measure {{}of the Brownian particle}. The goal of this work, which we explain shortly, is to compute the large-scale \emph{dynamical} behavior of the metric \emph{by means of invariant measures of the Brownian particle}, more or less.

The above description consists of a Brownian ``particle" which drives the growth of an evolving subset. One can also rewrite this process as the evolution of the boundary (i.e. the interface) rather than the whole set itself. In particular, adding a small aggregate to the set is equivalent to letting the interface move ``outward" by small amount. Now, assuming that the diffeomorphism class of the set and its boundary does not change after the interface moves outward, one can express the new interface as the graph of a function on the original interface. Similar to how mean curvature flow for an evolving hypersurface is conveniently described by a \abbr{PDE} for a function on the original (i.e. ``initial") hypersurface, the above dynamics can be thought of as a particle process plus a time-dependent (random) function on the original interface $\partial\bUz$. This is the perspective we will take, as it is more convenient to work with, analytically. 

Our main result (Theorem \ref{theorem:thm2}) states that as we take $\e\to0$, the evolution of this random growth model converges to the solution of a deterministic ``flow-type" equation \eqref{eq:limit}, whose local-in-time well-posedness is shown in Theorem \ref{theorem:thm1}. This limit can be but is \emph{not} globally-in-time well-posed in general {{}(as evidenced by the explicit solutions given in Section \ref{section:examplesection}), and in particular, we show in Theorem \ref{theorem:thm1b} that blow-up of \eqref{eq:limit} occurs exactly when its solution fails to be an immersion. (Moreover, a canonical extension of this flow gives rise to an interface process whose blow-up occurs exactly when it fails to be a smooth codimension $1$ sub-manifold in $\R^{\d}$; we explain this after Theorem \ref{theorem:thm1b}.)} This, along with going well beyond radial growth, addresses the first limitation of \cite{DGHS}. As for the second limitation of \cite{DGHS}, \emph{we only need the local time parameter $\delta(\e)$ to satisfy $\delta(\e)\gg\e$, i.e. we allow any mesoscopic scale regularization} (with the same limit across all such scales). So, the speed of the Brownian particle can be anything much faster than the speed of the growth. In words, we show an \emph{averaging} or \emph{homogenization} principle; as $\e\to0$, the randomness of the Brownian particle averages out, {{}and the interface flows outwards at a speed depending only on its Euclidean surface measure (normalized to be a probability measure). In particular, the evolution of the interface becomes autonomous as $\e\to0$.} To accomplish this removal of separation of scales, we use methods of stochastic differential geometry, which is in stark contrast to the non-geometric methods in \cite{DGHS}.

{{}This concludes the heuristic overview of this paper. In Section \ref{section:mainresults}, we state and describe the main results and model precisely, followed by illustrating examples in Section \ref{section:examplesection}. The structure of the rest of the paper is described in Section \ref{subsection:organization}.}
\subsection{Acknowledgments}
We thank Otis Chodosh and Mykhaylo Shkolnikov
for helpful discussions. We thank Yuval Peres, Amanda Turner, and James Norris for helpful comments and for explaining their work. We thank James Norris for bringing up questions about global lifetime for special initial data to \eqref{eq:limit}. {{}We would also like to thank an anonymous referee for their very helpful comments, which improved the manuscript significantly.} Research supported in part by \abbr{NSF} grant \abbr{DMS}-1954337 (A.D.), and by a fellowship from the \abbr{ARCS} foundation and the \abbr{NSF} Mathematical Sciences Postdoctoral Fellowship program under Grant. No. \abbr{DMS}-2203075 (K.Y.). Both authors were sponsored by the \abbr{NSF} under Grant. No. \abbr{DMS}-1928930 at a program hosted at \abbr{MSRI} in Berkeley, CA during Fall 2021.
\subsection{Notation}
Given any set $I$, whenever we write $a \lesssim_{I} b$, we mean $|a|\leq C|b|$, where $C\geq0$ depends only on $I$. Similarly, when we write $a\gtrsim_{I} b$, we mean $b\lesssim_{I}a$. If there is no subscript, then $I$ is empty.
%
%
%
\section{Main results}\label{section:mainresults}
\subsection{The macroscopic flow equation}
In a nutshell, the macroscopic limit is a geometric flow for a subset of $\R^{\d}$ that ``inflates" in the outward normal direction with speed determined by its surface measure (normalized to be a probability measure). With notation explained afterwards, it is given by the following integral equation (in which $t\geq0$ and $x\in\partial\bUz$): 
\begin{align}
\Phi^{\mathrm{hom}}(t,x) \ = \ \Phi^{\mathrm{hom}}(0,x) + \int_{0}^{t}\d s\int_{\partial \bUz}\mathbf{T}^{\Phi^{\mathrm{hom}}(s,\cdot)}(\d y)\mathscr{K}(x,y)\mathsf{n}^{\Phi^{\mathrm{hom}}(s,\cdot)}(y).\label{eq:limit}
\end{align}
Above, $\mathscr{K}\in\mathscr{C}^{\infty}(\partial\bUz\times\partial\bUz,\R_{\geq0})$ is a fixed smooth kernel; it determines the shape of the outwards growth. To define the other terms in \eqref{eq:limit}, first fix a map $\Phi:\partial\bUz\to\R^{\d}$ that is diffeomorphic onto its image. We let $\mathbf{T}^{\Phi}$ be the pullback under $\Phi:\partial\bUz\to\Phi(\partial\bUz)$ of the Euclidean surface measure on $\Phi(\partial\bUz)$ (normalized to be a probability measure). We let $\mathsf{n}^{\Phi}(y)$ be the unit outward normal to $\Phi(\partial\bUz)$ at $\Phi(y)$. (It points away from the ``interior" of $\Phi(\partial\bUz)$, i.e. away from the unique compact, connected component with boundary $\Phi(\partial\bUz)$.)

To analyze \eqref{eq:limit}, it is convenient to compute the $\mathbf{T}$-measure on the \abbr{RHS} more explicitly. For this purpose, we introduce notation that will be important throughout this paper. For the rest of this paper, we fix a compact, connected set $\bUz\subseteq\R^{\d}$ with smooth boundary $\partial\bUz$.
\begin{definition}\label{definition:limit}
Let $\mathscr{C}^{\infty}_{\simeq}(\partial\bUz,\R^{\d})$ denote the space of all smooth maps on $\partial\bUz$ that are diffeomorphisms onto their image. For $\sck\geq0$, we let 
$\mathscr{C}^{\sck}_{\simeq}(\partial\bUz,\R^{\d})$ be {{}the set of injective maps in the completion of $\mathscr{C}^{\infty}_{\simeq}(\partial\bUz,\R^{\d})$ under the following metric:}
\begin{align}
\langle\phi,\psi\rangle_{\mathscr{C}^{\sck}} \ := \ \|\phi-\psi\|_{\mathscr{C}^{\sck}}+\sup_{i,j=1,\ldots,\d-1}\|(\mathrm{Jac}\phi)^{-1}_{ij}-(\mathrm{Jac}\psi)^{-1}_{ij}\|_{\mathscr{C}^{\sck-1}}.
\end{align}
Above, $\|\cdot\|_{\mathscr{C}^{\sck}}$ is the standard $\mathscr{C}^{\sck}(\partial\bUz,\R^{\d})$-norm, and $\mathrm{Jac}\phi$ is the Jacobian (or derivative matrix) of $\phi$ with respect to a fixed orthonormal frame $\mathbf{e}_{1},\ldots,\mathbf{e}_{\d-1}$ for the $(\d-1)$-dimensional manifold $\partial\bUz$. (To be precise, $\mathrm{Jac}\phi(y)$ is a linear map from the tangent space of $\partial\bUz$ at $y\in\partial\bUz$ to that of $\phi(\partial\bUz)$ at $\phi(y)$. The first tangent space has orthonormal basis $\mathbf{e}_{i}(y)$. The latter tangent space is given an orthonormal basis obtained by taking the image of $\mathbf{e}_{1}(y),\ldots,\mathbf{e}_{\d-1}(y)$ of the Jacobian of $\phi:\partial\bUz\to\R^{\d}$ and doing Gram-Schmidt for the resulting $(\d-1)$-many vectors; note these are linearly independent if $\phi$ is diffeomorphic onto its image. In particular, $\mathrm{Jac}\phi(y)$ has entries given by polynomials in the derivatives of $\phi$ along $\mathbf{e}_{1},\ldots,\mathbf{e}_{\d-1}$ at $y$. Also, $(\mathrm{Jac}\phi)^{-1}_{ij}$ is the $(i,j)$-entry of the inverse of $\mathrm{Jac}\phi$, which is well-defined since $\phi$ is a diffeomorphism onto its image.) 

Next, for any smooth hypersurface $\mathbb{H}\subseteq\R^{\d+1}$, we let $\d\Sigma_{\mathbb{H}}$ be the Euclidean surface metric on $\mathbb{H}$.
\end{definition}
{{}Injectivity is important to be able to make sense of the normal vector in \eqref{eq:limit}. We will only consider small perturbations of a diffeomorphisms anyway which preserve this injectivity property.} 

By change-of-variables, we can now compute
\begin{align}
\mathbf{T}^{\Phi}(\d y) \ := \ \frac{|\det\mathrm{Jac}\Phi(y)|}{\int_{\partial\bUz}|\det\mathrm{Jac}\Phi(w)|\d\Sigma_{\partial\bUz}(\d w)}\d\Sigma_{\partial\bUz}(\d y).\label{eq:tequation}
\end{align}
Our first result, Theorem \ref{theorem:thm1}, shows that \eqref{eq:limit} is locally-in-time well-posed. Before we state this precisely, for any $\tau \in (0,\infty]$, we let $\mathscr{C}([0,\tau),\mathscr{C}^{\mathrm{k}}_{\simeq}(\partial\bUz,\R^{\d}))$ be the space of continuous $\mathscr{C}^{\mathrm{k}}_{\simeq}(\partial\bUz,\R^{\d}))$-valued functions on $[0,\tau)$. We give it the topology of convergence in $\mathscr{C}^{\mathrm{k}}_{\simeq}(\partial\bUz,\R^{\d}))$ locally uniformly on $[0,\tau)$.
\begin{theorem}\label{theorem:thm1}
{{}Suppose that $\Phi^{\mathrm{hom}}(0,\cdot)\in\mathscr{C}^{10}_{\simeq}(\partial \bUz,\R^{\d})$. Let $\tau_{\mathrm{sol}}\geq0$ be the largest time such that \eqref{eq:limit} has a unique solution $\Phi^{\mathrm{hom}}(\cdot,\cdot)$ in $\mathscr{C}([0,\tau_{\mathrm{sol}}),\mathscr{C}^{10}_{\simeq}(\partial \bUz,\R^{\d}))$ with initial data $\Phi^{\mathrm{hom}}(0,\cdot)$. Then $\tau_{\mathrm{sol}}\in(0,\infty]$.}
\end{theorem}
The use of $\mathscr{C}^{10}_{\simeq}$ in Theorem \ref{theorem:thm1} is not important; any $\sck\geq1$ (instead of $\sck=10$) would likely be enough. (Because we assumed that $\mathscr{K}$ is smooth, our analysis does not depend very much on the number of derivatives $\sck$, except for the fact that the measure $\mathbf{T}^{\Phi}$ in \eqref{eq:limit} already requires a derivative of $\Phi$. So, in principle, we are only forced to take $\sck\geq1$. We chose $\sck=10$ in order to avoid having to track the number of derivatives of certain functions of $\Phi$ that we are allowed to take, which may exceed one merely for technical reasons.)

We emphasize that \eqref{eq:limit} is non-random, so $\tau_{\mathrm{sol}}$ is also non-random. {{}{{}It turns out that $\tau_{\mathrm{sol}}$ has a \emph{geometric} meaning, as indicated by the following result. (In Theorem \ref{theorem:thm1b} below, the function space is $\mathscr{C}^{10}(\partial\bUz,\R^{\d})$, not $\mathscr{C}^{10}_{\simeq}(\partial\bUz,\R^{\d})$. We also recall that an \emph{immersion} is a $\mathscr{C}^{1}$ injective map whose Jacobian is nowhere singular.)
\begin{theorem}\label{theorem:thm1b}
Recall the notation from Theorem \ref{theorem:thm1}, and suppose that $\tau_{\mathrm{sol}}<\infty$. There exists $\Phi^{\mathrm{hom}}(\tau_{\mathrm{sol}},\cdot)\in\mathscr{C}^{10}(\partial\bUz,\R^{\d})$ such that $\Phi^{\mathrm{hom}}(t,\cdot)\to\Phi^{\mathrm{hom}}(\tau_{\mathrm{sol}},\cdot)$ in $\mathscr{C}^{10}(\partial\bUz,\R^{\d})$ as $t\to\tau_{\mathrm{sol}}^{-}$. Moreover, $\Phi^{\mathrm{hom}}(\tau_{\mathrm{sol}},\cdot)$ is not an immersion.
\end{theorem}
Theorem \ref{theorem:thm1b} first shows that $\Phi^{\mathrm{hom}}(t,\partial\bUz)$ converges as $t\to\tau_{\mathrm{sol}}$. It also indicates a ``geometric" singularity at $\tau_{\mathrm{sol}}$. This does \emph{not} rule out being able to use the inclusion $\Phi^{\mathrm{hom}}(\tau_{\mathrm{sol}},\partial\bUz)\subseteq\R^{\d}$ to give $\Phi^{\mathrm{hom}}(\tau_{\mathrm{sol}},\partial\bUz)$ a smooth (or Riemannian) structure that may be diffeomorphic to $\partial\bUz$ through another map. (It seems to be difficult to describe how $\Phi^{\mathrm{hom}}(t,\partial\bUz)$ changes \emph{topologically} as $t\geq0$ varies, e.g. how does its fundamental group change? We construct examples in Section \ref{section:examplesection} that show how $\Phi^{\mathrm{hom}}(t,\partial\bUz)$ can become not diffeomorphic or not even homotopy equivalent to $\partial\bUz$ as $t\geq0$ grows.) But if $\Phi^{\mathrm{hom}}(\tau_{\mathrm{sol}},\partial\bUz)$ is a smooth codimension $1$ sub-manifold in $\R^{\d}$, we can solve the flow \eqref{eq:limit} with $\Phi^{\mathrm{hom}}(\tau_{\mathrm{sol}},\partial\bUz)$ in place of $\partial\bUz$ and with initial data given by the identity map. Thus, we get an evolving interface even after $\tau_{\mathrm{sol}}$, unless the interface fails to be a smooth codimension $1$ sub-manifold at time $\tau_{\mathrm{sol}}$ when viewed as a subset of $\R^{\d}$. In particular, the maximal time for this new evolving interface, obtained by gluing solutions to \eqref{eq:limit} together on different time-intervals, is the first time where it fails to be a smooth codimension $1$ sub-manifold in $\R^{\d}$.}}
\subsection{The microscopic random growth model}
We now construct the growth model as a Markov process. It will use more decorated notation. For this reason, we clarify Definition \ref{definition:interface} immediately afterwards.

Before we give the construction, we make an additional assumption (that will be true for the setting of our main result, Theorem \ref{theorem:thm2}). \emph{Assume that $\bUz$ has a Euclidean collar of length 1}, which we denote by $\mathbf{C}_{1}$. That is, there is a set $\mathbf{C}_{1}\subseteq\bUz$ such that each $x\in\mathbf{C}_{1}$ can be written uniquely in terms of its distance $\mathrm{dist}(x,\partial\bUz)$ from the boundary and the ``horizontal" point $x^{\mathrm{hor}}\in\partial\bUz$ that satisfies $\mathrm{dist}(x,\partial\bUz)=|x-x^{\mathrm{hor}}|$. So, there exists an isomorphism $\mathbf{C}_{1}\simeq\partial\bUz\times[0,1]$ given by $x\mapsto(x^{\mathrm{hor}},\mathrm{dist}(x,\partial\bUz))$. 
\begin{definition}\label{definition:interface}
Define the state space $\Omega:=\partial\bUz\times\mathscr{C}^{10}_{\simeq}(\partial\bUz,\R^{\d})$. 

Take the microscopic scaling parameter $\e>0$. Define $t\mapsto(\mathbf{z}^{\e}(t),\Phi^{\e}(t))\in\Omega$ as the Markov process with the following generator (for test functions $F:\Omega\to\R$), which uses notation to be explained afterwards:
\begin{align}
\mathscr{L}^{\e}F(z,\Phi) \ := \ \e^{-1}\int_{\partial\bUz}\mathbf{T}^{\Phi,\delta}_{z}(\d y)\left[F(y,\mathscr{T}^{\e,y}\Phi)-F(z,\Phi)\right], {{}\quad (z,\Phi)\in\Omega}. 
\end{align}
First, given any $y\in\partial\bUz$, we let $\mathscr{T}^{\e,y}\Phi(x)=\Phi(x)+\e\mathscr{K}(x,y)\mathsf{n}^{\Phi}(y)$ be the addition of a volume $\e$ bump to the $\Phi$-interface (as in the limiting dynamics \eqref{eq:limit}, but without integration over $y$). (Recall $\mathscr{K}$ from \eqref{eq:limit}.) It remains to specify $\mathbf{T}_{z}^{\Phi,\delta}(\d y)$. Let us choose it be the probability measure on $\partial\bUz$ {{}corresponding to the law of $\mathbf{x}^{\Phi,z}(\tau^{\Phi,z,\delta})$, which we define below precisely. (In words, $\mathbf{x}^{\Phi,z}$ denotes a reflecting Brownian motion on $\bUz$ that is defined with respect to a metric determined by $\Phi$, that is reflected on $\partial\bUz$, and that starts at $z\in\partial\bUz$; see \cite{DH} for the construction of a reflecting Brownian motion. We stop it at the first time $\tau^{\Phi,z,\delta}$ that the boundary local time of $\mathbf{x}^{\Phi,z}$ is equal to a fixed parameter $\delta=\delta(\e)$.)}
\begin{enumerate}
\item Given any $\Phi\in\mathscr{C}^{10}_{\simeq}(\partial\bUz,\R^{\d})$, we define a metric $\mathsf{g}^{\Phi}$ on $\bUz$ as follows. (See Figure \ref{figure:3}.)

First, let $\mathsf{g}^{\Phi}|_{\partial\bUz}$ be the pullback of the Euclidean surface metric on $\Phi(\partial\bUz)$ under 
\begin{align}
\Phi:\partial\bUz\simeq\Phi(\partial\bUz).
\end{align}
This is a metric on $\partial\bUz$, and it defines the restriction of our metric $\mathsf{g}^{\Phi}$ to $\partial\bUz$. On the interior $\bUz\setminus\mathbf{C}_{1}$, we let $\mathsf{g}^{\Phi}$ be standard Euclidean metric. Finally, on the collar $\mathbf{C}_{1}\simeq\partial\bUz\times[0,1]$, we let it interpolate between our boundary metric $\mathsf{g}^{\Phi}|_{\partial\bUz}$ and the Euclidean metric. We define it precisely as follows.

First, consistently with the previous paragraph for $\Phi$ the identity map, let $\mathsf{g}^{\mathrm{Id}}|_{\partial\bUz}$ be the Euclidean surface metric on $\partial\bUz$. Now, recall the identification $\mathbf{C}_{1}\simeq\partial\bUz\times[0,1]$, of any $x\in\mathbf{C}_{1}$ 
with $(x^{\mathrm{hor}},\mathrm{dist}(x,\partial\bUz))$, yielding
the \emph{interpolating metric} on $\mathbf{C}_{1}$:
\begin{align}
\mathsf{g}^{\mathrm{inter},\Phi}(x) \ &:= \ \mathsf{g}^{\mathrm{Id}}|_{\partial\bUz}(x^{\mathrm{hor}})\chi(\mathrm{dist}(x,\partial\bUz)) 
+ \ \mathsf{g}^{\Phi}|_{\partial\bUz}(x^{\mathrm{hor}})\left\{1-\chi(\mathrm{dist}(x,\partial\bUz))\right\}.
\end{align}
In the above, $\chi\in\mathscr{C}^{\infty}(\R)$ is $[0,1]$-valued such that $\chi(r)=0$ for $r\leq\frac12$ and $\chi(r)=1$ for $r\geq1$. In words, as we vary $\mathrm{dist}(x,\partial\bUz)$ from $0$ to $1$, we smoothly transition from $\mathsf{g}^{\Phi}|_{\partial\bUz}$ to $\mathsf{g}^{\mathrm{Id}}|_{\partial\bUz}$. (We actually hit $\mathsf{g}^{\mathrm{Id}}|_{\partial\bUz}$ when $\mathrm{dist}(x,\partial\bUz)=\frac12$ with our choice of $\chi$.) 

So far, if we think of $\mathbf{C}_{1}\simeq\partial\bUz\times[0,1]$ as a foliation of $\partial\bUz\times\{r\}$ for $r\in[0,1]$, we have defined a smooth (in $r$) family of metrics on each leaf $\partial\bUz\times\{r\}$. To get a metric on the entire foliation $\mathbf{C}_{1}\simeq\partial\bUz\times[0,1]$, we are left to choose any metric on $[0,1]$. In particular, we now define
\begin{align}
\mathsf{g}^{\Phi}|_{\mathbf{C}_{1}}:=\mathsf{g}^{\mathrm{inter},\Phi}(x)\wedge\d\mu^{\mathrm{Euc}},
\end{align}
which is a product metric on $\mathbf{C}_{1}\simeq\partial\bUz\times[0,1]$ (here, $\d\mu^{\mathrm{Euc}}$ is Euclidean metric on $[0,1]$).
\item The process $\mathbf{x}^{\Phi,z}$ is a reflecting Brownian motion on $\bUz$ with unit inward normal reflection and initial condition $z\in\partial\bUz$. It is defined with respect to the metric $\mathsf{g}^{\Phi}$. (Again, see \cite{DH}.) Also, given $\delta=\delta(\e)$, define the following stopping time, in which $\mathbf{L}(t)$ is the boundary local time of $\mathbf{x}^{\Phi,z}$:
\begin{align}
\tau^{\Phi,z,\delta} \ := \ \inf\{t\geq0:\mathbf{L}(t)=\delta\}.
\end{align}
(This stopping time is almost surely finite; see Lemma \ref{lemma:sde}.)
\end{enumerate}
\end{definition}
\begin{figure}[h!]
\includegraphics[width=0.3\textwidth]{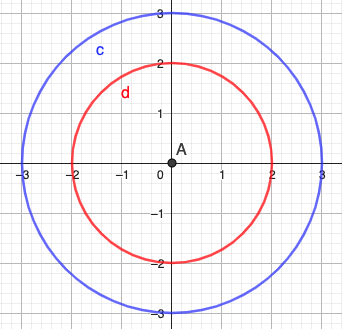}
\caption{Here, $\bUz$ is the disk whose boundary is the blue circle. The region between the blue and red circles is the collar $\mathbf{C}_{1}$ for $\bUz$. The metric $\mathsf{g}^{\Phi}$ is defined to be Euclidean inside the red circle. On the blue boundary, it is the pullback of Euclidean surface metric on $\Phi(\partial\bUz)$. Interpolate between blue and red circles to finish the definition of $\mathsf{g}^{\Phi}$.}
\label{figure:3}
\end{figure}
\begin{remark}
{{}An \abbr{SDE} representation of $\mathbf{x}^{\Phi,z}$ is given at the beginning of Section \ref{section:couplingproof}.}
\end{remark}
The process $t\mapsto(\mathbf{z}^{\e}(t),\Phi^{\e}(t))$ has two components. The second is the interface (or more precisely, the function determining the interface). The dynamics are described by placing a bump $x\mapsto\e\mathscr{K}(x,y)\mathsf{n}^{\Phi}(y)$ to the function $x\mapsto\Phi(x)$, where $y$ is the location of the next jump of $\mathbf{z}^{\e}(t)$. (Integrating over $y\in\partial\bUz$ in the generator comes from averaging over all possible jump locations according to the probability of said jump.)

What warrants more clarification is $\mathbf{z}^{\e}(t)$. As noted in the introduction, the particle we want to consider is Brownian motion with respect to the time-independent standard Euclidean metric and normal reflection on a \emph{time-dependent} set. If the time-dependent domain does not change diffeomorphism class, we can diffeomorphically map it to the original set $\bUz$. More precisely:
\begin{itemize}
\item Suppose $\bUz^{+}$ is $\bUz$ plus a small bump near a point on $\partial\bUz$. Take any diffeomorphism $\partial\bUz\simeq\partial\bUz^{+}$, and recall the foliation $\mathbf{C}_{1}\simeq\partial\bUz\times[0,1]$. We build a diffeomorphism $\bUz\simeq\bUz^{+}$ as follows. To each leaf in the foliation $\mathbf{C}_{1}\simeq\partial\bUz\times[0,1]=\{\partial\bUz\times\{r\}\}_{r\in[0,1]}$, apply a diffeomorphism that interpolates smoothly between $\partial\bUz\simeq\partial\bUz^{+}$ at $r=0$ and the identity at $r=1$. Then, let its restriction to $\bUz\setminus\mathbf{C}_{1}$ be the identity. (Intuitively, ``stretch" the ``edge" of the collar $\mathbf{C}_{1}$ while fixing $\bUz\setminus\mathbf{C}_{1}$.) The metric obtained by pulling back Euclidean metric on $\bUz^{+}$ under this diffeomorphism has the form $\mathsf{g}^{\Phi}$ from Definition \ref{definition:interface}. {{}(Indeed, we change the metric on the boundary $\partial\bUz$ and smoothly interpolate it with the Euclidean metric. Our choice of interpolation makes it so that how close to we are to Euclidean metric on $\partial\bUz$ at a particular point $x\in\bUz$ is a smooth function of only $\mathrm{dist}(x,\partial\bUz)$.)}
\end{itemize}
In particular, Definition \ref{definition:interface} is not just an isolated construction. It rewrites the growing domain as a fixed set with an evolving metric, which is, again, much more analytically convenient to use. For example, instead of a reflecting Brownian motion on a time-varying domain, we only need to study a reflecting Brownian motion on a fixed domain with a time-dependent metric. We also clarify that the exact constraints for the interpolator $\chi$ are not too important. Indeed, we can describe the same growing domain with a different choice of $\chi$ if we slightly adjust our diffeomorphism above. In fact, to capture the growing domain picture above, we should let $\chi$ evolve in an adapted fashion as well. This is not an issue, however; we can still choose $\chi$ to satisfy the constraints spelled out in Definition \ref{definition:interface}, which are there only for later analytic convenience. Moreover, the scaling limit we obtain is independent of $\chi(\cdot)$. One disadvantage with this description of the interface model is that it is a specific parameterization of the growth model which only holds at best until its diffeomorphism class changes. For further discussion on this, see Section \ref{subsection:extensionglobal}.

In this representation, the process $t\mapsto\mathbf{z}^{\e}(t)$ is a Poisson jump process of speed $\e^{-1}$, where the law of a jump is determined by stopping the reflecting Brownian motion $\mathbf{x}^{\Phi,z}$ when its boundary local time reaches $\delta=\delta(\e)$. The underlying process is often called the \emph{boundary trace} \cite{Hsu0} of the reflecting Brownian motion $\mathbf{x}^{\Phi,z}$ at time $\delta=\delta(\e)$,  {{}and it enjoys analytically nice properties. For instance, the generator of the boundary trace of Brownian motion on a fixed Riemannian manifold is the well-studied Dirichlet-to-Neumann operator \cite{ARP,Hsu0,EO}, and its invariant measure is the Riemannian measure on its the boundary \cite{Hsu0,EO}.} This is useful for proving homogenization in $\mathbf{z}^{\e}$, which is key to deriving a macroscopic limit for $\Phi^{\e}$.
\subsection{Convergence $\Phi^{\e}\to\Phi^{\mathrm{hom}}$}
We first introduce the relevant topological spaces, taken from \cite{Bil},
in which convergence will hold. {{}Let $\mathscr{D}([0,\tau),\mathscr{C}^{10}_{\simeq}(\partial\bUz,\R^{\d}))$ be the space of cadlag functions $[0,\tau)\to\mathscr{C}^{10}_{\simeq}(\partial\bUz,\R^{\d})$ with the following Skorokhod topology. We say $\{\Gamma^{(\e)}\}_{\e>0}$ converges to $\Gamma$ as $\e\to0$ in $\mathscr{D}([0,\tau),\mathscr{C}^{10}_{\simeq}(\partial\bUz,\R^{\d}))$ if for any $\rho>0$, we have $\Gamma^{(\e)}\to\Gamma$ in the Skorokhod space $\mathscr{D}([0,\tau-\rho],\mathscr{C}^{10}_{\simeq}(\partial\bUz,\R^{\d}))$, whose definition and properties can be found in Sections 12 and 13 of \cite{Bil}. In words, Skorokhod convergence of paths defined for times in the open interval $[0,\tau)$ is equivalent to convergence when restricted to the ``compact approximations" $[0,\tau-\rho]$ of $[0,\tau]$ for $\rho>0$ small.} 

Let us also clarify that by a collar of length $\ell$, we (again) mean a set $\mathbf{C}_{\ell}\subseteq\bUz$ such that every element in $x\in\mathbf{C}_{\ell}$ can be written uniquely as its distance, which is $\leq\ell$, from the boundary $\partial\bUz$ and a point {{}$x^{\mathrm{hor}}\in\bUz$} that satisfies $\mathrm{dist}(x,\partial\bUz)=|x-x^{\mathrm{hor}}|$. 

Finally, let $\tau_{\mathrm{sol}}^{\e}$ denote the supremum of all $t\geq0$ for which $\Phi^{\e}(t,\cdot)\in\mathscr{C}^{10}_{\simeq}(\partial\mathbf{U},\R^{\d})$. Recall $\tau_{\mathrm{sol}}$ from Theorem \ref{theorem:thm1}.
\begin{theorem}\label{theorem:thm2}
First, assume $\delta=\delta(\e)>0$ in \emph{Definition \ref{definition:interface}} satisfies $\delta\times\e^{-1}\to\infty$ and $\delta\to0$ as $\e\to0$. Next, assume $\Phi^{\e}(0,\cdot)=\Phi^{\mathrm{init}}(\cdot)\in\mathscr{C}^{10}_{\simeq}(\partial\bUz,\R^{\d})$ is deterministic and independent of $\e$. Assume $\bUz\subseteq\R^{\d}$ is a compact, connected subset with smooth boundary, and assume it has a collar of length 3.

Let $\Phi^{\mathrm{hom}}$ solve \eqref{eq:limit} with $\Phi^{\mathrm{hom}}(0,\cdot)=\Phi^{\mathrm{init}}(\cdot)$. We have $\tau_{\mathrm{sol}}^{\e}\wedge\tau_{\mathrm{sol}}\to\tau_{\mathrm{sol}}$ and $\Phi^{\e}-\Phi^{\mathrm{hom}}\to0$ in the space $\mathscr{D}([0,\tau_{\mathrm{sol}}),\mathscr{C}^{10}_{\simeq}(\partial\bUz,\R^{\d}))$, with both of these limits being in probability as 
$\e \to 0$.
\end{theorem}
The condition $\delta(\e)\gg\e$ is essentially the assumption of working in the ``homogenization" or ``averaging" setting; we explain this point more near the end of the proof sketch below.

Let us clarify the claim in Theorem \ref{theorem:thm2} about $\tau_{\mathrm{sol}}^{\e},\tau_{\mathrm{sol}}$. In principle, the blow-up time for $t\mapsto\Phi^{\e}(t,\cdot)$ may be different than $\tau_{\mathrm{sol}}$ (and even vanishing as $\e\to0$). Theorem \ref{theorem:thm2} says this does not happen (with high probability). Also, the assumption of a collar of length 3 is not important. We only need a collar with some positive length that is independent of $\e$. (The choice of 3 is just convenient.)

Let us compare Theorem \ref{theorem:thm2} to the results of \cite{NST0,NST}. These papers, which restrict to growth in $\R^{2}$, derive inflating discs as scaling limits for the models therein, though variants of the model in \cite{NST0} can have more complicated shape theorems which may exhibit similar blow-up. Moreover, in \cite{NST0}, there is a regularization-scale parameter analogous to $\delta(\e)$.
\subsubsection{Proof sketch for Theorem \ref{theorem:thm2}}
By the Ito formula for $F(z,\Phi)=\Phi$ (before time $\tau_{\mathrm{sol}}^{\e}$), we have
\begin{align}
\Phi^{\e}(t,x) \ &= \ \Phi^{\e}(0,x) + \e^{-1}\int_{0}^{t}\d s\int_{\partial \bUz}\mathbf{T}_{\mathbf{z}^{\e}(s)}^{\Phi^{\e}(s,\cdot),\delta}(\d y)\left(\mathscr{T}^{\e,y}\Phi^{\e}(s,x)-\Phi^{\e}(s,x)\right) + M^{\e}(t,x) \nonumber \\
&= \ \Phi^{\e}(0,x) + \int_{0}^{t}\d s\int_{\partial \bUz}\mathbf{T}_{\mathbf{z}^{\e}(s)}^{\Phi^{\e}(s,\cdot),\delta}(\d y)\mathscr{K}(x,y)\mathsf{n}^{\Phi^{\e}(s,\cdot)}(y) + M^{\e}(t,x). \label{eq:itointro}
\end{align}
The function $M^{\e}$ is a martingale with respect to the canonical filtration of $t\mapsto(\mathbf{z}^{\e}(t),\Phi^{\e}(t,\cdot))$; explicitly, it equals the following compensated Poisson process (which records jumps of $t\mapsto\Phi^{\e}(t,\cdot)$):
\begin{align}
M^{\e}(t,x)=\sum_{\tau_{i}\leq t}\e\mathscr{K}(x,\mathbf{z}^{\e}(\tau_{i}))\mathsf{n}^{\Phi^{\e}(\tau_{i}^{-},\cdot)}(\mathbf{z}^{\e}(\tau_{i})) - \int_{0}^{t}\d s\int_{\partial \bUz}\mathbf{T}_{\mathbf{z}^{\e}(s)}^{\Phi^{\e}(s,\cdot),\delta}(\d y)\mathscr{K}(x,y)\mathsf{n}^{\Phi^{\e}(s,\cdot)}(y). \label{eq:phimg}
\end{align}
Above, the sum is over all jump times $\tau_{i}$ of the system until time $t$. The assumption that $\delta\times\e^{-1}\to\infty$ says that the speed of $\mathbf{z}^{\e}$ is much faster than the speed of the interface growth. Thus, by the time that the growth has evolved by $\mathrm{o}(1)$, the law of the particle $\mathbf{z}^{\e}$ is close to the {{}normalized surface measure on the interface, i.e. the measure induced by the metric $\mathsf{g}^{\Phi^{\e}(s,\cdot)}|_{\partial\bUz}$}. Standard ergodic theory then lets us replace $\mathbf{T}^{\Phi^{\e}(s,\cdot),\delta}_{\mathbf{z}^{\e}(s)}(\d y)$ in \eqref{eq:itointro} with {{}said measure, which is just $\mathbf{T}^{\Phi^{\e}(s,\cdot)}(\d y)$ from \eqref{eq:limit}}. So, $\Phi^{\e}$ and $\Phi^{\mathrm{hom}}$ solve the same \abbr{PDE} (modulo $\mathrm{o}(1)$ terms) with the same initial data. It remains to show stability of \eqref{eq:limit} under perturbations; this more or less follows by well-posedness of \eqref{eq:limit} (see Theorem \ref{theorem:thm1}).

The assumption $\delta\times\e^{-1}\to\infty$ appears optimal; we cannot generally allow $\delta$ to vanish any faster. (Indeed, if $\delta\times\e^{-1}=\mathrm{O}(1)$, the particle and growth move at the same speed, so a homogenization principle seems out of reach and possibly even  false.) Finally, the assumption that $\delta\to0$ is purely for technical convenience. It is not necessary, but if $\delta\not\to0$, the necessary analysis is essentially already addressed in \cite{DGHS}.
\subsection{Extension to global-in-time results and more about blow-up}\label{subsection:extensionglobal}
Theorems \ref{theorem:thm1} and \ref{theorem:thm2} only hold locally in time, because the geometry of the interface can certainly change in finite time; {{}see Theorem \ref{theorem:thm1b}}. However, the growth set itself is still well-defined as a union of compact subsets. Thus, its global-in-time evolution can, in principle, be studied as well. Put more carefully, the issue is not that $\Phi^{\e}$ or $\Phi^{\mathrm{hom}}$ blow up in $\mathscr{C}^{\sck}(\partial\bUz,\R^{\d})$, but that their inverses do; {{}again, see Theorem \ref{theorem:thm1b}}. {{}If the image of $\Phi^{\mathrm{hom}}$ at the blow-up time $\tau_{\mathrm{sol}}$ is a smooth codimension $1$ sub-manifold in $\R^{\d}$, then similar to the discussion after Theorem \ref{theorem:thm1b}, we can prove Theorem \ref{theorem:thm2} but with $\Phi^{\mathrm{hom}}(\tau_{\mathrm{sol}},\partial\bUz)$ as the new initial subset instead of $\partial\bUz$. This would give us a scaling limit result on the time-interval $\tau_{\mathrm{sol}}+[0,\tau)$ for some deterministic $\tau>0$, which we can then glue to the scaling limit on $[0,\tau_{\mathrm{sol}})$. Ultimately, this gives us a notion of scaling limit past time $\tau_{\mathrm{sol}}$. 

If $\Phi^{\mathrm{hom}}(\tau_{\mathrm{sol}},\partial\bUz)\subseteq\R^{\d}$ is \emph{not} a smooth codimension $1$ sub-manifold, then one can regularize it. Precisely, look at points on $\partial\bUz$ where $\Phi^{\mathrm{hom}}(\tau_{\mathrm{sol}},\cdot)$ is not an immersion, and perturb $\mathrm{Jac}\Phi^{\mathrm{hom}}(\tau_{\mathrm{sol}},\cdot)$ around said points to make it globally invertible. (For example, if $\Phi^{\mathrm{hom}}(\tau_{\mathrm{sol}},\partial\bUz)\subseteq\R^{\d}$ has a self-intersection, we can ``fatten" it in a smooth fashion so that this modification of $\Phi^{\mathrm{hom}}(\tau_{\mathrm{sol}},\partial\bUz)$ is a smooth codimension $1$ sub-manifold.) See also \cite{DNS} for an example where the authors regularize blow-ups in the Stefan heat flow problem by means of a stochastic particle system.}

Although general initial data to \eqref{eq:limit} can lead to finite blow-up time, we {{}will see} in Section \ref{section:examplesection} that convex initial data and special kernels $\mathscr{K}$ lead to solutions with infinite lifetime. It may be the case that if the initial set $\bUz$ is convex, then the solution to \eqref{eq:limit} has infinite lifetime for more kernels satisfying milder assumptions. Infinite lifetime is probably still false for general kernels, however.
\subsection{Organization of the paper}\label{subsection:organization}
{{}Theorems \ref{theorem:thm1} and \ref{theorem:thm1b} are proved in Section \ref{section:flowproofs}. Theorem \ref{theorem:thm2} is proved in Section \ref{section:scalinglimit}. These proofs rely on an important homogenization estimate (Proposition \ref{prop:hom}), which is stated in Section \ref{section:hom} and proved in Section \ref{section:homproof}. For ease of reading, we defer certain points regarding reflecting \abbr{SDE}s to Appendix \ref{section:appendix}. We also defer the proof of a technically involved result (Proposition \ref{prop:coupling}) to Section \ref{section:couplingproof}.}
%
%
%
\section{Some examples}\label{section:examplesection}
We now give examples to illustrate Theorem \ref{theorem:thm1b} and what the flow \eqref{eq:limit} looks like. We give two examples below for which we can compute solutions explicitly via symmetry. We also give a less concrete but intuitive example which shows that generically, the solution to \eqref{eq:limit} can be quite complicated, and that simple solutions are the exception, not the rule.
\subsection{A concrete example}
Suppose that $\bUz$ is the unit ball centered at the origin in $\R^{\d}$, and $\Phi^{\mathrm{hom}}(0,\cdot)$ is the identity (so the initial data for the growth model is the unit ball). Also, suppose $\mathscr{K}$ is the heat kernel on the sphere $\partial\bUz$ (at any positive time). In this case, we claim that $\Phi^{\mathrm{hom}}(t,x)=\Phi^{\mathrm{hom}}(0,x)+t\mathbf{n}(x)$ solves \eqref{eq:limit}, where $\mathbf{n}(x)$ is the unit outward normal at $x\in\partial\bUz$. In words, if we plug the unit ball into the growth equation \eqref{eq:limit}, then its evolution is given by space-time homogeneous ``inflation". To see this claim, note that this choice of $\Phi^{\mathrm{hom}}(t,\cdot)$ has constant Jacobian, so that $\mathbf{T}^{\Phi^{\mathrm{hom}}(t,\cdot)}(\d y)$ is the uniform measure on the sphere {{}$\partial\mathbf{U}$}. Also, for this choice of $\Phi^{\mathrm{hom}}(t,\cdot)$, the normal vectors in \eqref{eq:limit} are just the unit outward radial vectors on $\partial\bUz$ (since for all $t$, the graph of $\Phi^{\mathrm{hom}}(t,\cdot)$ is still a sphere). The fact that $\Phi^{\mathrm{hom}}(t,x)-\Phi^{\mathrm{hom}}(0,x)$ is parallel to $\mathbf{n}(x)$ now follows since we are averaging unit outward normals at $y\in\partial\bUz$ with respect to a probability density on $\partial\bUz$ that is invariant under reflection about any plane connecting $x$ and its antipode. The fact that the projection of $\Phi^{\mathrm{hom}}(t,x)-\Phi^{\mathrm{hom}}(0,x)$ onto $\mathbf{n}(x)$ has coefficient $t$ follows because the total volume of the double integral on the \abbr{RHS} of \eqref{eq:limit} is $t$ (since $\mathscr{K}$ is a heat kernel on the sphere $\partial\bUz$).

More generally, one can take the initial data $\bUz$ for the growth model to be a convex, compact subset of $\R^{\d}$ that is closed under $x\mapsto-x$, i.e. the sub-level set of a norm on $\R^{\d}$. If $\mathscr{K}$ is an appropriate kernel that is invariant under transformations that preserve this norm, then the growth is given by sub-level sets of the same norm (but with a growing level set parameter).

Let us give a probabilistic interpretation. The evolution of $\Phi^{\e}$ is described by adding small bumps to the boundary $\partial\bUz$ at places where the boundary local time of a Brownian motion equals a fixed increment $\delta>0$. Assuming that the fluctuations in this random interface process average out as $\e\to0$, so that the small-$\e$ limit of $\Phi^{\e}$ is deterministic, it now becomes at least intuitively clear that if $\partial\bUz$ is the unit sphere in $\R^{\d}$, then the interface at time $t$ is still a sphere in $\R^{\d}$, just by the rotational symmetry of Brownian motion. 

\begin{figure}[h!]
\includegraphics[width=0.3\textwidth]{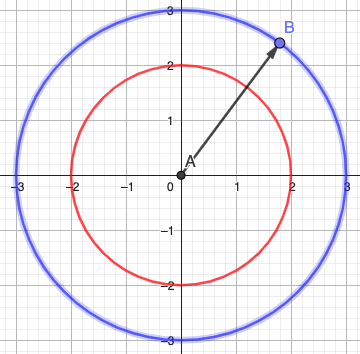}
\caption{The first concrete example of solution $\Phi^{\mathrm{hom}}$. The red circle inside is the initial data. The blue circle is the solution at time $1$, after which the circle inflates from radius 2 to radius 3.}
\label{figure:2}
\end{figure}
\subsection{Another concrete example}
{{}Now, let $\bUz$ be an annulus centered at the origin, so that $\partial\bUz$ is the disjoint union of two spheres $\mathbb{S}_{1}$ and $\mathbb{S}_{2}$ of radius $R,r$, respectively, where $R>r>0$. Let $\mathscr{K}(x,y)$ the heat kernel on the sphere $\mathbb{S}_{j}$ for $x,y$ in the same sphere $\mathbb{S}_{j}$ and $j=1,2$, and set $\mathscr{K}(x,y)=0$ if $x\in\mathbb{S}_{j}$ and $y\in\mathbb{S}_{\ell}$ for $j\neq\ell$. By construction, the two spheres inflate independently from one another since $\mathscr{K}(x,y)=0$ if $x\in\mathbb{S}_{j}$ and $y\in\mathbb{S}_{\ell}$ for $j\neq\ell$. In particular, we have $\Phi^{\mathrm{hom}}(t,x)=x+c_{j}t\mathbf{n}(x)$ for $x\in\mathbb{S}_{j}$ and $j=1,2$, where $c_{1},c_{2}>0$ are some speeds determined by the ratio of the surface area of the outer sphere and the surface of the inner sphere. (These speeds are relevant because the flow \eqref{eq:limit} modulates the speed of growth by the total surface area, since $\mathbf{T}^{\Phi^{\mathrm{hom}}(s,\cdot)}$ therein is normalized to be a probability measure.) In words, the bigger sphere $\mathbb{S}_{1}$ inflates outwards as in the previous example. On the other hand, the inner sphere $\mathbb{S}_{2}$ shrinks and collapses to the origin. Indeed, the unit outward normal vector field on the inner sphere $\mathbb{S}_{2}$ points towards the origin. Thus, the solution $\Phi^{\mathrm{hom}}(t,x)=x+c_{2}t\mathbf{n}(x)$, if we restrict to $x\in\mathbb{S}_{2}$, gives us $\Phi(t,x)=0$ for all $x\in\mathbb{S}_{2}$ when $t=rc_{2}^{-1}$. At this blow-up time, the growth model is a ball, which is not homotopic to the annulus.}
\subsection{The less concrete but more complicated example}
{{}Take two disjoint balls $\mathbb{B}_{1},\mathbb{B}_{2}$ separated by a small distance $\vartheta>0$. For $x,y$ in the same sphere $\partial\mathbb{B}_{j}$ and $j=1,2$, we define $\mathscr{K}(x,y)$ to be the heat kernel on $\partial\mathbb{B}_{j}$ like in the previous two examples. For $x,y$ in different spheres, set $\mathscr{K}(x,y)=0$. As in the previous example, both spheres inflate independently, and they will touch at some time that is short if $\vartheta$ is small. Now, connect $\mathbb{B}_{1},\mathbb{B}_{2}$ by a thin tube to get a simply connected horseshoe-type shape. Assuming that the distance $\vartheta$ is small enough, the two balls will again meet. Thus, the growth goes from a simply connected horseshoe-type shape to a set that is not simply connected; see Figure \ref{figure:4}. (This example is similar to the model in \cite{LL}.)}
\begin{figure}[h!]
\includegraphics[width=0.4\textwidth]{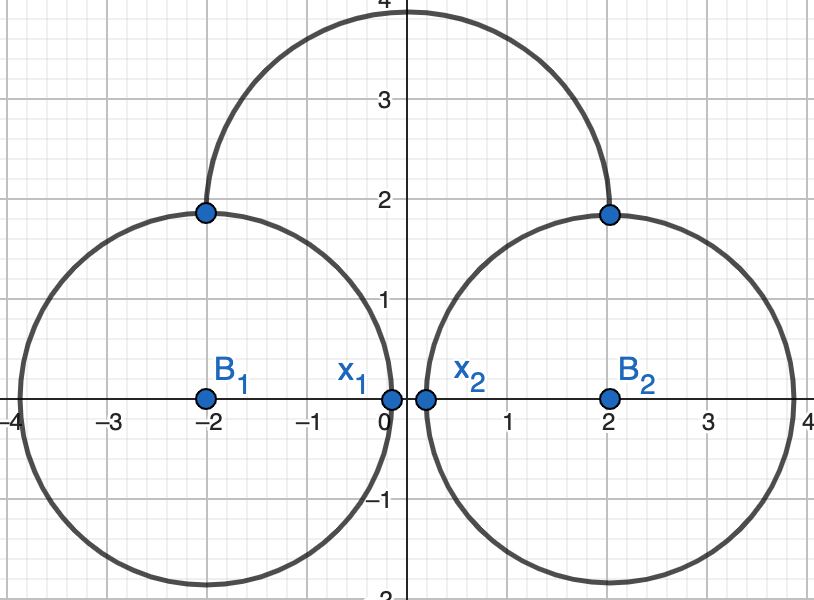}
\caption{A horseshoe-type shape (where the width of the connecting tube is thought as positive but small). This grows according to \eqref{eq:limit} by inflating $\mathbb{B}_{1},\mathbb{B}_{2}$ until the points $x_{1},x_{2}$ meet, so the set goes from being simply connected to having nontrivial fundamental group.}
\label{figure:4}
\end{figure}
%
%
%
\section{Homogenization}\label{section:hom}
The purpose of this section is to state the necessary homogenization ingredient that we briefly described after Theorem \ref{theorem:thm2}. It is the replacement of ``fluxes" $\mathsf{F}^{\e}\to\mathsf{F}^{\e,\mathrm{eff}}$, where
\begin{align}
\mathsf{F}^{\e}(t,x) \ &:= \ \int_{0}^{t}\d s\int_{\partial \bUz}\mathbf{T}_{\mathbf{z}^{\e}(s)}^{\Phi^{\e}(s,\cdot),\delta}(\d y)\mathscr{K}(x,y)\mathsf{n}^{\Phi^{\e}(s,\cdot)}(y) \\
\mathsf{F}^{\e,\mathrm{eff}}(t,x) \ &:= \ \int_{0}^{t}\d s \int_{\partial \bUz}\mathbf{T}^{\Phi^{\e}(s,\cdot)}(\d y)\mathscr{K}(x,y)\mathsf{n}^{\Phi^{\e}(s,\cdot)}(y).
\end{align}
We first establish some useful notation to be used throughout the paper.
\begin{enumerate}
\item For any $k=(k_{1},\ldots,k_{\d-1})$, set $\partial^{k}=\partial_{\mathbf{e}_{1}}^{k_{1}}\ldots\partial_{\mathbf{e}_{\d-1}}^{k_{\d-1}}$, where $\mathbf{e}_{1},\ldots,\mathbf{e}_{\d-1}$ is our fixed orthonormal frame for $\partial\bUz$, and $\partial^{k_{i}}_{\mathbf{e}_{i}}$ means the $k_{i}$-th derivative in the $\mathbf{e}_{i}$ direction.
\item For any $\tau\geq0$ and $\Phi\in\mathscr{C}([0,\tau],\mathscr{C}^{N}_{\simeq}(\partial\bUz,\R^{\d}))$, define the following where the second supremum is over multi-indices $k=(k_{1},\ldots,k_{\d-1})$, and where $|k|=|k_{1}|+\ldots+|k_{\d-1}|$:
\begin{align}
\|\Phi\|_{\tau,N} \ &:= \ \sup_{0\leq t\leq\tau}\sup_{0\leq|k|\leq N}\sup_{x\in\partial\bUz}|\partial^{k}\Phi(t,x)| \nonumber\\
\langle\Phi\rangle_{\tau,N} \ &:= \ \|\Phi\|_{\tau,N}+\sup_{i,j=1,\ldots,\d-1}\|(\mathrm{Jac}\Phi)_{ij}^{-1}\|_{\tau,N-1}. \nonumber
\end{align}
For convenience, we will write $\|\Phi\|_{\tau}:=\|\Phi\|_{\tau,10}$ and $\langle\Phi\rangle_{\tau}:=\langle\Phi\rangle_{\tau,10}$. It will also be convenient to define time-independent norms. For any $\Phi:\partial\bUz\to\R^{\d}$, set $\Phi(0,\cdot)=\Phi(\cdot)$; with this and the previous notation, we set $\|\Phi\|_{\mathscr{C}^{N}}:=\|\Phi\|_{0,N}$ and $\langle\Phi\rangle_{\mathscr{C}^{N}}:=\langle\Phi\rangle_{0,N}$.
\end{enumerate}
\begin{prop}\label{prop:hom}
Fix any stopping time $0\leq\tau\lesssim1$. There exist $b(\e),p(\e)\to0$ as $\e\to0$ so that with probability $\geq1-p(\e)$, we have $\|\mathsf{F}^{\e}-\mathsf{F}^{\e,\mathrm{eff}}\|_{\tau}\lesssim b(\e)i(\langle\Phi^{\e}\rangle_{\tau})$, where $i(\cdot)\geq1$ is an increasing, continuous, and deterministic function.
\end{prop}
As we alluded to in the introduction, the proof of Proposition \ref{prop:hom} is the key technical heart to this paper. We defer it to after proofs of Theorems \ref{theorem:thm1} and \ref{theorem:thm2}. In particular, we assume it for now, use it, and then prove it.
%
%
%
\section{Proofs of Theorem \ref{theorem:thm1} {{}and Theorem \ref{theorem:thm1b}}}\label{section:flowproofs}
{{}Before we proceed, we clarify that several of the results and proofs in the rest of this paper use continuous and increasing functions $i(\cdot)\geq1$ (and similar objects). These functions will reflect the dependence of our estimates on the geometry of the image of $\partial\bUz$ under some $\Phi\in\mathscr{C}^{10}_{\simeq}(\partial\bUz,\R^{\d})$. The exact choices of function $i(\cdot)$ will not be important, and it will be allowed to vary a finite number of times.}
\subsection{Preliminary regularity estimates}
We will start with continuity of $\mathsf{n}^{\Phi}$ as a function of $\Phi$. Beyond facilitating a fixed-point argument to solve \eqref{eq:limit}, it will also be important in controlling the microscopic process and thus proving Theorem \ref{theorem:thm2} as well.
\begin{lemma}\label{lemma:thm11}
Fix $\Phi_{1},\Phi_{2}\in\mathscr{C}^{10}_{\simeq}(\partial \bUz,\R^{\d})$. For some $i(\cdot)\geq1$ increasing and continuous, we have
\begin{align}
\|\mathsf{n}^{\Phi_{1}}-\mathsf{n}^{\Phi_{2}}\|_{\mathscr{C}^{0}} \ \leq \ i(\langle\Phi_{1}\rangle_{\mathscr{C}^{1}})\|\Phi_{1}-\Phi_{2}\|_{\mathscr{C}^{1}} \quad\mathrm{and}\quad \|\mathsf{n}^{\Phi_{1}}\|_{\mathscr{C}^{1}} \ \lesssim \ i(\langle\Phi_{1}\rangle_{\mathscr{C}^{10}}).
\end{align}
\end{lemma}
\begin{proof}
Recall $\mathrm{Jac}\Phi$ is the Jacobian matrix of $\Phi$. By definition, for $y\in\partial \bUz$ and $k=1,2$ and $1\leq\ell\leq\d-1$,
\begin{align}
\mathsf{n}^{\Phi_{k}}(y)\cdot[\mathrm{Jac}\Phi_{k}](y)\mathbf{e}_{\ell}(y) \ = \ 0, \label{eq:thm11normalequation}
\end{align}
where we recall $\mathbf{e}_{1},\ldots,\mathbf{e}_{\d-1}$ is the fixed orthonormal tangent frame for $\partial \bUz$. From this, we compute
\begin{align}
\left\{\mathsf{n}^{\Phi_{1}}(y)-\mathsf{n}^{\Phi_{2}}(y)\right\}\cdot[\mathrm{Jac}\Phi_{1}](y)\mathbf{e}_{\ell}(y) \ = \ -\mathsf{n}^{\Phi_{2}}(y)\cdot\left\{[\mathrm{Jac}\Phi_{1}](y)-[\mathrm{Jac}\Phi_{2}](y)\right\}\mathbf{e}_{\ell}(y). \label{eq:thm111}
\end{align}
The \abbr{RHS} of \eqref{eq:thm111} is $\mathrm{O}(\|\Phi_{1}-\Phi_{2}\|_{\mathscr{C}^{1}})$ since $\mathsf{n}^{\Phi_{2}}$ and $\mathbf{e}_{\ell}$ are unit length. Note this upper bound is independent of $y\in\partial \bUz$ and $\ell$. Also note the tangent space of $\Phi_{1}(\partial\bUz)$ at $\Phi(y)$ has $\{[\mathrm{Jac}\Phi_{1}](y)\mathbf{e}_{\ell}(y)\}_{\ell}$ as a basis. Thus, the projection of $\mathsf{n}^{\Phi_{1}}(y)-\mathsf{n}^{\Phi_{2}}(y)$ onto the tangent space of $\Phi_{1}(\partial \bUz)$ at $\Phi(y)$ has Euclidean length $i(\langle\Phi_{1}\rangle_{\mathscr{C}^{1}})\mathrm{O}(\|\Phi_{1}-\Phi_{2}\|_{\mathscr{C}^{1}})$ (the extra factor $i(\|\Phi_{1}^{-1}\|_{\mathscr{C}^{1}})$ comes from the fact that $[\mathrm{Jac}\Phi_{1}](y)\mathbf{e}_{\ell}(y)$ is not an orthonormal basis of the aforementioned tangent space; we must perform the Gram-Schmidt procedure and re-scale by $\langle\Phi_{1}\rangle_{\mathscr{C}^{1}}$-dependent factors). In words, we just said that $\mathsf{n}^{\Phi_{i}}(y)$ is normal to the tangent space of $\Phi_{i}(\partial\bUz)$ at $\Phi_{i}(y)$ and perturbed in $\Phi_{i}$. We must now control the projection of $\mathsf{n}^{\Phi_{1}}(y)-\mathsf{n}^{\Phi_{2}}(y)$ on the span of $\mathsf{n}^{\Phi_{1}}(y)$.

Letting $\mathbf{f}^{\Phi_{1}}_{\ell}$ be an orthonormal tangent frame for $\Phi_{1}(\partial \bUz)$, we now write
\begin{align}
\mathsf{n}^{\Phi_{2}}(y) \ &= \ \alpha\mathsf{n}^{\Phi_{1}}(y) + {\sum}_{\ell}\alpha^{\Phi_{1}}_{\ell}(y)\mathbf{f}^{\Phi_{1}}_{\ell}(y) \nonumber \\
\mathrm{where}\qquad  &\alpha:=\mathsf{n}^{\Phi_{1}}(y)\cdot\mathsf{n}^{\Phi_{2}}(y), \qquad 
 \alpha_{\ell}^{\Phi_{1}}(y):=\mathbf{f}_{\ell}^{\Phi_{1}}(y)\cdot\mathsf{n}^{\Phi_{2}}(y).
\end{align}
By construction of $\mathbf{f}^{\Phi_{1}}$ vectors as being normal to $\mathsf{n}^{\Phi_{1}}$, the sum over $\ell$ on the \abbr{RHS} equals the projection of $\mathsf{n}^{\Phi_{2}}(y)-\mathsf{n}^{\Phi_{1}}(y)$ onto the tangent space of $\Phi_{1}(\partial\bUz)$ at $\Phi(y)$. Thus, we have already proven that its length is $\leq i(\langle\Phi_{1}\rangle_{\mathscr{C}^{1}})\mathrm{O}(\|\Phi_{1}-\Phi_{2}\|_{\mathscr{C}^{1}})$. Let us now take the dot product of both sides with $\mathsf{n}^{\Phi_{2}}(y)$. Because $\mathsf{n}^{\Phi_{2}}(y)$ is unit length, the \abbr{LHS} gives 1, while the sum on the \abbr{RHS} gives $i(\langle\Phi_{1}\rangle_{\mathscr{C}^{1}})\mathrm{O}(\|\Phi_{1}-\Phi_{2}\|_{\mathscr{C}^{1}})$. The first term on the \abbr{RHS} gives $\alpha^{2}$. Thus, we have $|\alpha^{2}-1|\lesssim i(\langle\Phi_{1}\rangle_{\mathscr{C}^{1}})\mathrm{O}(\|\Phi_{1}-\Phi_{2}\|_{\mathscr{C}^{1}})$, and it now suffices to show that $\alpha$ is uniformly bounded away from $-1$ (and thus $\mathsf{n}^{\Phi_{1}}(y)$ and $\mathsf{n}^{\Phi_{2}}(y)$ have the same orientation). To this end, we note that by the inverse function theorem, because $\Phi\in\mathscr{C}^{10}_{\simeq}(\partial\bUz,\R^{\d})$, there exists $\upsilon>0$ depending only on $\langle\Phi_{1}\rangle_{\mathscr{C}^{1}}$ and $\partial\bUz$ so that if $\|\Phi_{1}-\Phi_{2}\|_{\mathscr{C}^{1}}\leq\upsilon$, then $\mathsf{n}^{\Phi_{1}}(y)\cdot\mathsf{n}^{\Phi_{2}}(y)\geq0$. (In words, the direction of the unit outward normal to $\Phi_{1}(\partial\bUz)$ at any point cannot ``drastically flip direction" after a small perturbation of $\Phi_{1}$.) Now, if $\|\Phi_{1}-\Phi_{2}\|_{\mathscr{C}^{1}}\leq\upsilon$ for this choice of $\upsilon$, then $\alpha=\mathsf{n}^{\Phi_{1}}(y)\cdot\mathsf{n}^{\Phi_{2}}(y)\geq0$ is clearly bounded uniformly away from $-1$, and thus the first estimate follows. If $\|\Phi_{1}-\Phi_{2}\|_{\mathscr{C}^{1}}\geq\upsilon$, then the first estimate follows immediately anyway by possibly modifying $i(\cdot)$. Indeed, in this case,
\begin{align}
\|\mathsf{n}^{\Phi_{1}}-\mathsf{n}^{\Phi_{2}}\|_{\mathscr{C}^{0}} \ \lesssim \ 1 \ = \ \upsilon\times\upsilon^{-1} \ \leq \ \upsilon^{-1}\|\Phi_{1}-\Phi_{2}\|_{\mathscr{C}^{1}},
\end{align}
and $\upsilon^{-1}$ depends only on $\langle\Phi_{1}\rangle_{\mathscr{C}^{10}}$. 

We now prove the second bound claimed in the lemma. Recall \eqref{eq:thm11normalequation}. As $[\mathrm{Jac}\Phi_{1}](y)$ is invertible, the set $\{[\mathrm{Jac}\Phi_{1}](y)\mathbf{e}_{\ell}(y)\}_{\ell}$ has $(\d-1)$-dimensional span. Therefore, \eqref{eq:thm11normalequation} has a unique solution $\mathsf{n}^{\Phi_{1}}$ that is unit-length and outwards pointing. Now, note
\begin{align}
0 \ = \ \partial\{\mathsf{n}^{\Phi_{1}}(y)\cdot[\mathrm{Jac}\Phi_{1}](y)\mathbf{e}_{\ell}(y)\} \ = \ \partial\mathsf{n}^{\Phi_{1}}(y) \cdot [\mathrm{Jac}\Phi_{1}](y)\mathbf{e}_{\ell}(y) + \mathsf{n}^{\Phi_{1}}(y)\cdot\partial\{[\mathrm{Jac}\Phi_{1}](y)\mathbf{e}_{\ell}(y)\},
\end{align}
where $\partial$ is with respect to the frame $\{\mathbf{e}_{\ell}\}_{\ell}$. This controls the projection of $\partial\mathsf{n}^{\Phi_{1}}$ onto $[\mathrm{Jac}\Phi_{1}](y)\mathbf{e}_{\ell}(y)$ by $\mathrm{O}(\|\Phi_{1}\|_{\mathscr{C}^{2}})$ (and therefore projection of $\partial\mathsf{n}^{\Phi_{1}}$ onto the tangent space spanned by $[\mathrm{Jac}\Phi_{1}](y)\mathbf{e}_{\ell}(y)$ by $i(\langle\Phi_{1}\rangle_{\mathscr{C}^{10}})$). It now suffices to note that the projection of $\partial\mathsf{n}^{\Phi_{1}}$ onto $\mathsf{n}^{\Phi_{1}}$ is zero, namely that $\partial\mathsf{n}^{\Phi_{1}}\cdot\mathsf{n}^{\Phi_{1}}=0$ (this can be seen by applying $\partial$ to $\mathsf{n}^{\Phi_{1}}\cdot\mathsf{n}^{\Phi_{1}}=1$). 
\end{proof}
We will now do (almost) the same thing but for {{}normalized surface} measures $\mathbf{T}^{\Phi}$ instead of normal vectors. (The proof of the following is elementary, modulo a couple of details, so we only explain said details.)
\begin{lemma}\label{lemma:thm12}
Take $\Phi_{1},\Phi_{2}\in\mathscr{C}^{10}_{\simeq}(\partial \bUz,\R^{\d})$. We have the following (with $\mathbf{T}^{\Phi}$ and $\langle\cdot,\cdot\rangle_{\mathscr{C}^{10}}$ from \eqref{eq:tequation} and \emph{Definition \ref{definition:limit}}, respectively):
\begin{align}
\int_{\partial \bUz}|\mathbf{T}^{\Phi_{1}}(\d y)-\mathbf{T}^{\Phi_{2}}(\d y)| \ &\lesssim_{} i(\langle\Phi_{1}\rangle_{\mathscr{C}^{10}},\langle\Phi_{2}\rangle_{\mathscr{C}^{10}})\langle\Phi_{1},\Phi_{2}\rangle_{\mathscr{C}^{10}}.
\end{align}
Above, the function $i(\cdot,\cdot)\geq1$ is jointly continuous and increasing in both inputs.
\end{lemma}
\begin{proof}
By an elementary calculation, the \abbr{LHS} of the proposed estimate is
\begin{align*}
&\lesssim\int_{\partial\bUz}\frac{|\det\mathrm{Jac}\Phi_{1}(y)-\det\mathrm{Jac}\Phi_{2}(y)|}{\int_{\partial\bUz}|\det\mathrm{Jac}\Phi_{1}(x)|\d\Sigma_{\partial\bUz}(x)}\d\Sigma_{\partial\bUz}(y)\\
&+\int_{\partial\bUz}|\det\mathrm{Jac}\Phi_{2}(y)|\left|\frac{1}{\int_{\partial\bUz}|\det\mathrm{Jac}\Phi_{1}(x)|\d\Sigma_{\partial\bUz}(x)}-\frac{1}{\int_{\partial\bUz}|\det\mathrm{Jac}\Phi_{2}(x)|\d\Sigma_{\partial\bUz}(x)}\right|\d\Sigma_{\partial\bUz}(y).
\end{align*}
Since $\det\mathrm{Jac}\Phi$ is polynomial of first derivatives of $\Phi$, the result would follow immediately if we knew that $|\det\mathrm{Jac}\Phi_{i}(x)|\gtrsim_{\langle\Phi_{i}\rangle_{\mathscr{C}^{10}}}1$, namely that it is uniformly bounded below depending only on $\langle\Phi_{i}\rangle_{\mathscr{C}^{10}}$. Since $\langle\Phi_{i}\rangle_{\mathscr{C}^{10}}$ controls the inverse of $\mathrm{Jac}\Phi_{i}$, this is clear.
\end{proof}
\subsection{Proof of Theorem \ref{theorem:thm1}}
Our construction of solutions to \eqref{eq:limit} will be done by contraction mapping. (In a nutshell, the main point is that differentiating the \abbr{RHS} of the equation \eqref{eq:limit} only hits the smooth kernel $\mathscr{K}$, not the solution itself. Hence, there is no ``loss of regularity". {{}Moreover, the initial data $\Phi^{\mathrm{hom}}(0,\cdot)$ is assumed to be injective, so the space we will build a contraction mapping on, which is a small ball in $\mathscr{C}^{10}_{\simeq}(\partial\bUz,\R^{\d})$-norm around $\Phi^{\mathrm{hom}}(0,\cdot)$, will be complete with respect to the $\mathscr{C}^{10}_{\simeq}(\partial\bUz,\R^{\d})$-norm and consist only of maps that are injective.}) For $\tau\geq0$, let us define the following map for any $\Phi\in\mathscr{C}([0,\tau],\mathscr{C}^{10}_{\simeq}(\partial \bUz,\R^{\d}))$:
\begin{align}
\mathfrak{S}\Phi(t,x)  \ = \ \Phi(0,x) + \int_{0}^{t}\d s\int_{\partial \bUz}\mathbf{T}^{\Phi(s,\cdot)}(\d y)\mathscr{K}(x,y)\mathsf{n}^{\Phi(s,\cdot)}(y).
\end{align}
We now define the space on which we will build a contraction mapping. {{}Let $\mathscr{B}(\Phi^{\mathrm{hom}}(0,\cdot),r)$ be the closed ball of radius $r>0$ in $\mathscr{C}^{10}_{\simeq}(\partial \bUz,\R^{\d})$-norm centered at $\Phi^{\mathrm{hom}}(0,\cdot)$, where $r>0$ is small enough so that every element in this ball is diffeomorphic onto its image. (In particular, $\mathscr{B}(\Phi^{\mathrm{hom}}(0,\cdot),r)$ is complete with respect to the $\mathscr{C}^{10}_{\simeq}(\partial \bUz,\R^{\d})$-norm.) Also, for any $\tau\geq0$, set $\mathscr{X}(\tau)=\{\Phi\in\mathscr{C}([0,\tau],\mathscr{B}(\Phi^{\mathrm{hom}}(0,\cdot),r)):\Phi(0,\cdot)=\Phi^{\mathrm{hom}}(0,\cdot)\}$.} We first show that $\mathfrak{S}$ defines a map $\mathscr{X}(\tau)\to\mathscr{X}(\tau)$ if $\tau>0$ is small enough (but depending only on $\Phi^{\mathrm{hom}}$ and $\bUz$). Via the triangle inequality, for $\Phi\in\mathscr{X}(\tau)$ (with $\tau>0$ to be determined shortly), we have the following (with $\|\|_{s}$-norms defined before Proposition \ref{prop:hom}):
\begin{align}
\|\mathfrak{S}\Phi-\Phi^{\mathrm{hom}}(0,\cdot)\|_{\tau} \ &\leq \ \int_{0}^{\tau}\d s \|\int_{\partial \bUz}\mathbf{T}^{\Phi(s,\cdot)}(\d y)\mathscr{K}(\cdot,y)\mathsf{n}^{\Phi(s,\cdot)}(y)\|_{s} \nonumber \\
&\leq \ \int_{0}^{\tau}\d s \int_{\partial \bUz}\mathbf{T}^{\Phi(s,\cdot)}(\d y)\|\mathscr{K}(\cdot,y)\mathsf{n}^{\Phi(s,\cdot)}(y)\|_{s} \ \lesssim \ \tau,
\end{align}
where the last estimate follows by regularity of the kernel $\mathscr{K}(\cdot,\cdot)$. Fix $\upsilon>0$ sufficiently small but depending only on $\Phi^{\mathrm{hom}}(0,\cdot)$ and $\partial \bUz$ such that if $\Psi\in\mathscr{C}^{10}(\partial \bUz,\R^{\d})$ and $\|\Psi-\Phi^{\mathrm{hom}}(0,\cdot)\|_{\mathscr{C}^{10}}\leq\upsilon$, then $\Psi$ is invertible and $\langle\Psi,\Phi^{\mathrm{hom}}(0,\cdot)\rangle_{\mathscr{C}^{10}}\leq r$. (The existence of such $\upsilon>0$ follows from the inverse function theorem.) Take $\tau>0$ depending only on $\Phi^{\mathrm{hom}}(0,\cdot)$ and $\partial \bUz$ so that $\|\mathfrak{S}\Phi-\Phi^{\mathrm{hom}}(0,\cdot)\|_{\tau'}\leq\upsilon$ for all $\tau'\leq\tau$, which we can do because of the previous display. For any such $\tau'$, we deduce $\mathfrak{S}\Phi\in\mathscr{X}(\tau')$ and $\mathfrak{S}$ defines a continuous map $\mathscr{X}(\tau')\to\mathscr{X}(\tau')$.

We now show that for a possibly smaller $\tau>0$ (but still depending only on $\Phi^{\mathrm{hom}}(0,\cdot)$ and $\partial \bUz$), the map $\mathfrak{S}:\mathscr{X}(\tau)\to\mathscr{X}(\tau)$ is a strict contraction. Take $\Phi_{1},\Phi_{2}\in\mathscr{X}(\tau)$; recall this implies $\Phi_{1}(0,\cdot)=\Phi_{2}(0,\cdot)=\Phi^{\mathrm{hom}}(0,\cdot)$. By definition of $\mathfrak{S}$,
\begin{align}
\mathfrak{S}\Phi_{1}(t,x) - \mathfrak{S}\Phi_{2}(t,x) \ = \ &\int_{0}^{t}\d s \int_{\partial \bUz}\{\mathbf{T}^{\Phi_{1}(s,\cdot)}(\d y)-\mathbf{T}^{\Phi_{2}(s,\cdot)}(\d y)\}\mathscr{K}(x,y)\mathsf{n}^{\Phi_{1}(s,\cdot)}(y) \nonumber  \\
+ \ &\int_{0}^{t}\d s \int_{\partial \bUz}\mathbf{T}^{\Phi_{2}(s,\cdot)}(\d y)\mathscr{K}(x,y)\{\mathsf{n}^{\Phi_{1}(s,\cdot)}(y)-\mathsf{n}^{\Phi_{2}(s,\cdot)}(y)\}.
\end{align}
Thus, by the triangle inequality, regularity of the kernel $\mathscr{K}(\cdot,\cdot)$, Lemma \ref{lemma:thm11}, and Lemma \ref{lemma:thm12}, we have
\begin{align}
\|\mathfrak{S}\Phi_{1}-\mathfrak{S}\Phi_{2}\|_{\tau} \ \leq \ &\int_{0}^{\tau}\d s \int_{\partial \bUz}|\mathbf{T}^{\Phi_{1}(s,\cdot)}(\d y)-\mathbf{T}^{\Phi_{2}(s,\cdot)}(\d y)|\|\mathscr{K}(\cdot,y)\mathsf{n}^{\Phi_{1}(s,\cdot)}(y)\|_{\mathscr{C}^{10}} \nonumber \\
+ \ &\int_{0}^{\tau}\d s \int_{\partial \bUz}\mathbf{T}^{\Phi_{2}(s,\cdot)}(\d y)\|\mathscr{K}(\cdot,y)\|_{\mathscr{C}^{10}}|\mathsf{n}^{\Phi_{1}(s,\cdot)}(y)-\mathsf{n}^{\Phi_{2}(s,\cdot)}(y)|\nonumber  \\
\lesssim \ &i(\langle\Phi_{1}\rangle_{\tau},\langle\Phi_{2}\rangle_{\tau})\langle\Phi_{1},\Phi_{2}\rangle_{\tau}\times\tau.
\end{align}
Recall that if $\tau>0$ is sufficiently small depending only on $\Phi^{\mathrm{hom}}(0,\cdot)$ and $\partial \bUz$, then $\langle\Phi_{1}\rangle_{\tau}+\langle\Phi_{2}\rangle_{\tau}\lesssim_{\tau}1$. Thus, for such $\tau>0$, we can forget $i(\langle\Phi_{1}\rangle_{\tau},\langle\Phi_{2}\rangle_{\tau})$ by changing the implied constant. In particular, given $\varrho>0$, the previous display implies existence of $\tau>0$ such that $\|\mathfrak{S}\Phi_{1}-\mathfrak{S}\Phi_{2}\|_{\tau}\leq\varrho\langle\Phi_{1},\Phi_{2}\rangle_{\tau}$. By the inverse function theorem, this also implies $\langle\mathfrak{S}\Phi_{1},\mathfrak{S}\Phi_{2}\rangle_{\tau}\lesssim\wt{i}(\langle\Phi_{1}\rangle_{\tau},\langle\Phi_{2}\rangle_{\tau})\times\varrho\langle\Phi_{1},\Phi_{2}\rangle_{\tau}$, where $\wt{i}(\cdot,\cdot)\geq1$ is jointly continuous and increasing. Again, because $\langle\Phi_{1}\rangle_{\tau}+\langle\Phi_{2}\rangle_{\tau}\lesssim_{\tau}1$, we deduce $\langle\mathfrak{S}\Phi_{1},\mathfrak{S}\Phi_{2}\rangle_{\tau}\lesssim\varrho\langle\Phi_{1},\Phi_{2}\rangle_{\tau}$. We now choose $\varrho>0$ sufficiently small (and $\tau>0$ accordingly) so that $\langle\mathfrak{S}\Phi_{1},\mathfrak{S}\Phi_{2}\rangle_{\tau}<\frac12\langle\Phi_{1},\Phi_{2}\rangle_{\tau}$. So, $\mathfrak{S}:\mathscr{X}(\tau)\to\mathscr{X}(\tau)$ is a proper contraction for $\tau>0$ depending only on $\Phi^{\mathrm{hom}}(0,\cdot)$ and $\bUz$. Existence and uniqueness until such time $\tau$ is now a consequence of Banach's fixed point theorem. This shows the existence of a lifetime (for well-posedness) that is strictly positive. \qed
\subsection{Proof of Theorem \ref{theorem:thm1b}}
{{}We first show that $\Phi^{\mathrm{hom}}(t,\cdot)$ converges in $\mathscr{C}^{10}(\partial\bUz,\R^{\d})$ as $t\to\tau_{\mathrm{sol}}^{-}$. Fix $\rho>0$ small. For any $t_{1},t_{2}\in[\tau_{\mathrm{sol}}-\rho,\tau_{\mathrm{sol}})$ such that $t_{1}\leq t_{2}$, we have the following by \eqref{eq:limit}:
\begin{align*}
\|\Phi^{\mathrm{hom}}(t_{1},\cdot)-\Phi^{\mathrm{hom}}(t_{2},\cdot)\|_{\mathscr{C}^{10}(\partial\bUz,\R^{\d})}&\lesssim\|\int_{t_{1}}^{t_{2}}\int_{\partial\bUz}\mathbf{T}^{\Phi^{\mathrm{hom}}(s,\cdot)}(\d y)\mathscr{K}(\cdot,y)\mathsf{n}^{\Phi^{\mathrm{hom}}(s,\cdot)}(y)\|_{\mathscr{C}^{10}(\partial\bUz,\R^{\d})}\\
&\lesssim\int_{t_{1}}^{t_{2}}\int_{\partial\bUz}\mathbf{T}^{\Phi^{\mathrm{hom}}(s,\cdot)}(\d y)\|\mathscr{K}(\cdot,y)\mathsf{n}^{\Phi^{\mathrm{hom}}(s,\cdot)}(y)\|_{\mathscr{C}^{10}(\partial\bUz,\R^{\d})},
\end{align*}
where the second line follows by the triangle inequality and that $\mathbf{T}^{\Phi^{\mathrm{hom}}(s,\cdot)}(\d y)$ is a positive measure independent of the variable with respect to which we take $\mathscr{C}^{10}(\partial\bUz,\R^{\d})$-norms. Since $\mathscr{K}(\cdot,y)$ is smooth uniformly in $y$, and since all normal vectors here are length-one, we have
\begin{align*}
\int_{t_{1}}^{t_{2}}\int_{\partial\bUz}\mathbf{T}^{\Phi^{\mathrm{hom}}(s,\cdot)}(\d y)\|\mathscr{K}(\cdot,y)\mathsf{n}^{\Phi^{\mathrm{hom}}(s,\cdot)}(y)\|_{\mathscr{C}^{10}(\partial\bUz,\R^{\d})}\lesssim\int_{t_{1}}^{t_{2}}\int_{\partial\bUz}\mathbf{T}^{\Phi^{\mathrm{hom}}(s,\cdot)}(\d y)=t_{2}-t_{1}\leq\rho,
\end{align*}
where the last two statements follow by the fact that $\mathbf{T}^{\Phi^{\mathrm{hom}}(s,\cdot)}(\d y)$ is a probability measure on $\partial\bUz$ and that $\tau_{\mathrm{sol}}\leq t_{1}\leq t_{2}\leq\tau_{\mathrm{sol}}-\rho$. Thus, for any such $t_{1},t_{2}$, we deduce $\|\Phi^{\mathrm{hom}}(t_{1},\cdot)-\Phi^{\mathrm{hom}}(t_{2},\cdot)\|_{\mathscr{C}^{10}(\partial\bUz,\R^{\d})}\lesssim\rho$. Because $\rho>0$ is arbitrary, this shows that the sequence $\{\Phi^{\mathrm{hom}}(t,\cdot)\}_{t}$ is Cauchy as $t\to\tau_{\mathrm{sol}}^{-}$ in $\mathscr{C}^{10}(\partial\bUz,\R^{\d})$. By completeness of $\mathscr{C}^{10}(\partial\bUz,\R^{\d})$, the first claim about convergence in Theorem \ref{theorem:thm1b} follows.

We now show that $\Phi^{\mathrm{hom}}(\tau_{\mathrm{sol}},\cdot)$ is not an immersion. We assume that $\Phi^{\mathrm{hom}}(\tau_{\mathrm{sol}},\cdot)$  is injective; otherwise there is nothing to prove. Thus, we are left to show that $\det\mathrm{Jac}\Phi^{\mathrm{hom}}(\tau_{\mathrm{sol}},y)=0$ for some $y\in\partial\bUz$. For any $t\in[0,\tau_{\mathrm{sol}})$, the adjugate formula for the inverse of a matrix gives us
\begin{align*}
[\mathrm{Jac}\Phi^{\mathrm{hom}}(t,x)]^{-1}_{ij}=[\det\mathrm{Jac}\Phi^{\mathrm{hom}}(t,x)]^{-1}\mathfrak{P}_{ij}(\mathrm{Jac}\Phi^{\mathrm{hom}}(t,x)), \quad i,j=1,\ldots,d-1,
\end{align*}
where $\mathfrak{P}_{ij}(\mathrm{Jac}\Phi^{\mathrm{hom}}(t,x))$ is some degree $d-1$ polynomial of the entries of $\mathrm{Jac}\Phi^{\mathrm{hom}}(t,x)$ that depends on the indices $i,j$. By the product rule, the $k$-th order derivatives of $[\mathrm{Jac}\Phi^{\mathrm{hom}}(t,x)]^{-1}_{ij}$ have the form
\begin{align}
[\det\mathrm{Jac}\Phi^{\mathrm{hom}}(t,x)]^{-k-1}\mathfrak{P}_{ij;k}(\Phi^{\mathrm{hom}}(t,x)),\label{eq:thm1bkth}
\end{align}
where $\mathfrak{P}_{ij;k}(\Phi^{\mathrm{hom}}(t,x))$ is a polynomial that depends on the $\ell$-th derivatives of $\Phi^{\mathrm{hom}}(t,x)$ for $\ell\in\{1,\ldots,k+1\}$. Now, because $\tau_{\mathrm{sol}}<\infty$, and because $\Phi^{\mathrm{hom}}(t,\cdot)$ is a smooth injection on $\partial\bUz$ for all $t\in[0,\tau_{\mathrm{sol}}]$, the only way $\tau_{\mathrm{sol}}$ can be a blow-up time is if \eqref{eq:thm1bkth} diverges in $\mathscr{C}^{0}(\partial\bUz,\R)$-norm as $t\to\tau_{\mathrm{sol}}^{-}$ for some $k\in\{0,\ldots,9\}$. (Indeed, the $k$-th order derivatives of $\Phi^{\mathrm{hom}}(t,x)$ are uniformly bounded on $(t,x)\in[0,\tau_{\mathrm{sol}}]\times\partial\bUz$ by the first part of Theorem \ref{theorem:thm1b}. Thus, if the $\mathscr{C}^{10}_{\simeq}(\partial\bUz,\R^{\d})$-norm of $\Phi^{\mathrm{hom}}(t,\cdot)$ blows up at $\tau_{\mathrm{sol}}^{-}$, then at least one of the entries of $[\mathrm{Jac}\Phi^{\mathrm{hom}}(t,x)]^{-1}$ must have $\mathscr{C}^{9}(\partial\bUz,\R)$-norm diverging as $t\to\tau_{\mathrm{sol}}^{-}$; see Definition \ref{definition:limit}.) However, by the first part of Theorem \ref{theorem:thm1b} and the fact that $\mathfrak{P}_{ij;k}(\Phi^{\mathrm{hom}}(t,x))$ depends only on the $\ell$-th derivatives of $\Phi^{\mathrm{hom}}(t,x)$ for $\ell\in\{1,\ldots,k+1\}$, we have
\begin{align*}
\sup_{t\in[0,\tau_{\mathrm{sol}}]}\|\mathfrak{P}_{ij;k}(\Phi^{\mathrm{hom}}(t,\cdot))\|_{\mathscr{C}^{0}(\partial\bUz,\R)}<\infty, \quad k\in\{0,\ldots,9\}.
\end{align*}
Thus, in order for \eqref{eq:thm1bkth} to diverge in $\mathscr{C}^{0}(\partial\bUz,\R)$-norm for some $k\in\{0,\ldots,9\}$ as $t\to\tau_{\mathrm{sol}}^{-}$, we must have that $\det\mathrm{Jac}\Phi^{\mathrm{hom}}(t,x_{t})\to0$ along a sequence $x_{t}$ as $t\to\tau_{\mathrm{sol}}^{-}$. Since $\partial\bUz$ is compact, we can find a convergent subsequence in $\{x_{t}\}_{t}$ whose limit is the point $y$ we are looking for, so the proof is complete. \qed}
%
%
%
\section{Proof of Theorem \ref{theorem:thm2}}\label{section:scalinglimit}
\subsection{Preliminary martingale estimate}
We start by controlling the martingale term in \eqref{eq:itointro}. This requires a combination of standard martingale bounds with a Sobolev inequality.
\begin{lemma}\label{lemma:mg}
Fix any stopping time $0\leq\tau\lesssim1$. We have $\|M^{\e}\|_{\tau}\lesssim\e^{1/3}$ with probability at least $1-\mathrm{O}(\e^{1/6})$.
\end{lemma}
\begin{proof}
We apply $\partial^{k}$ to \eqref{eq:phimg} (for $k\geq0$ fixed). This gives the following, where $\partial^{k}$ is always in $x$:
\begin{align}
\partial^{k}M^{\e}(t,x) \ &= \ \sum_{0\leq\tau_{i}\leq t}\e\partial^{k}\mathscr{K}(x,\mathbf{z}^{\e}(\tau_{i}))\mathsf{n}^{\Phi^{\e}(\tau_{i}^{-},\cdot)}(\mathbf{z}^{\e}(\tau_{i}))- \int_{0}^{t}\d s\int_{\partial \bUz}\mathbf{T}_{\mathbf{z}^{\e}(s)}^{\Phi^{\e}(s,\cdot),\delta}(\d y)\partial^{k}\mathscr{K}(x,y)\mathsf{n}^{\Phi^{\e}(s,\cdot)}(y). \nonumber
\end{align}
Given any $x\in\partial \bUz$, note the process $t\mapsto\partial^{k}M^{\e}(t,x)$ is a martingale with respect to the canonical filtration. This can be seen from the above display (or the observation that $\partial^{k}$ is a linear operator, and martingales are closed under linear combinations). Thus, we can estimate its time-maximal process with the Burkholder-Davis-Gundy inequality. Precisely, defining $\mathfrak{m}_{\e}^{\partial,k}(\tau,x)=\sup_{0\leq t\leq\tau}|\partial^{k}M^{\e}(t,x)|$, for any $p\geq1$, we have the following (pointwise-in-$x$) moment estimate:
\begin{align}
\E|\mathfrak{m}_{\e}^{\partial,k}(\tau,x)|^{2p} \ &\lesssim_{p} \ \E\left[{\sum}_{\tau_{i}}\e^{2}|\partial^{k}\mathscr{K}(x,\mathbf{z}^{\e}(\tau_{i}))\mathsf{n}^{\Phi^{\e}(\tau_{i}^{-},\cdot)}(\mathbf{z}^{\e}(\tau_{i}))|^{2}\right]^{p}.
\end{align}
Again, recall $|\partial^{k}\mathscr{K}(x,\mathbf{z}^{\e}(\tau_{i}))\mathsf{n}^{\Phi^{\e}(\tau_{i}^{-},\cdot)}(\mathbf{z}^{\e}(\tau_{i}))|\lesssim_{k}1$ by regularity of $\mathscr{K}$ and boundedness of the normal. Also, the number of jumps $\tau_{i}\leq\tau$ has Poisson distribution of speed $\mathrm{O}(\e^{-1})$. Since $\tau\lesssim1$, this gives $\E|\#\{\tau_{i}\leq\tau\}|^{p}\lesssim_{p}\e^{-p}$, and therefore we have
\begin{align}
\E\left[{\sum}_{\tau_{i}}\e^{2}|\partial^{k}\mathscr{K}(x,\mathbf{z}^{\e}(\tau_{i}))\mathsf{n}^{\Phi^{\e}(\tau_{i}^{-},\cdot)}(\mathbf{z}^{\e}(\tau_{i}))|^{2}\right]^{p} \ \lesssim_{k,p} \ \e^{2p}\E|\#\{\tau_{i}\leq\tau\}|^{p} \ \lesssim_{p} \ \e^{p}.
\end{align}
Combining the previous two displays yields $\E|\mathfrak{m}_{\e}^{\partial,k}(\tau,x)|^{2p}\lesssim_{k,p}\e^{p}$. To bound $\|M^{\e}\|_{\tau}$, we use Morrey's Sobolev inequality. Recalling the definition of $\|\|_{\tau}$, this gives the following estimate (for some positive integers $N=N(\d)$ and $p=p(\d)$)
\begin{align}
\E\|M^{\e}\|_{\tau}^{2p} \ &\lesssim \ \E\sup_{0\leq t\leq\tau}\sup_{0\leq|k|\leq N}\int_{\partial \bUz}\Sigma_{\partial \bUz}(\d x) |\partial^{k}M^{\e}(t,x)|^{2p} \nonumber \\
&\leq \ \sum_{0\leq|k|\leq N}\int_{\partial \bUz}\Sigma_{\partial \bUz}(\d x) \E|\mathfrak{m}_{\e}^{\partial,k}(\tau,x)|^{2p} \ \lesssim_{p} \ \e^{p}.
\end{align}
The second bound above follows from controlling the supremum over $0\leq|k|\leq N$ by a sum, and by the triangle inequality and Fubini theorem to move the supremum over $0\leq t\leq\tau$ inside the integral. It remains to use the Chebyshev inequality.
\end{proof}
\subsection{Proof of Theorem \ref{theorem:thm2}}
First, let us establish some notation. Fix any $\tau\lesssim1$ that satisfies $\tau<\tau_{\mathrm{sol}}$, and define $\wt{\tau}:=\tau\wedge\wt{\tau}_{\mathrm{sol}}^{\e}$, where $\wt{\tau}_{\mathrm{sol}}^{\e}$ is the following ``pseudo-explosion time" for $\Phi^{\e}$ (which uses $\langle\rangle_{t}$-type norms that were built before Proposition \ref{prop:hom}):
\begin{align}
\wt{\tau}_{\mathrm{sol}}^{\e} \ &:= \ \inf\left\{t\geq0: \langle\Phi^{\e}\rangle_{t}\geq \tfrac12+\langle\Phi^{\mathrm{hom}}\rangle_{t}\right\}.
\end{align}
We emphasize $\wt{\tau}_{\mathrm{sol}}^{\e}$ is a stopping time since $\Phi^{\mathrm{hom}}$ is deterministic (and therefore adapted for trivial reasons). Set $\mathcal{E}:=\mathcal{E}(1)\cap\mathcal{E}(2)$, where $\mathcal{E}(1):=\{\|\mathsf{F}^{\e}-\mathsf{F}^{\e,\mathrm{eff}}\|_{\wt{\tau}}\lesssim b(\e)i(\|\Phi^{\e}\|_{\wt{\tau}})\}$ and $\mathcal{E}(2):=\{\|M^{\e}\|_{\wt{\tau}}\lesssim\e^{1/3}\}$ with $i(\cdot)$ and $b(\e)$ from Proposition \ref{prop:hom}. By Proposition \ref{prop:hom} and Lemma \ref{lemma:mg}, we know $\mathbf{P}(\mathcal{E}^{C})\leq\mathbf{P}(\mathcal{E}(1)^{C})+\mathbf{P}(\mathcal{E}(2)^{C})\to0$ as $\e\to0$. (For our application of Proposition \ref{prop:hom}, we are using the fact that before $\wt{\tau}_{\mathrm{sol}}^{\e}$, we have control on $\langle\Phi^{\e}\rangle_{t}$ in terms of $\langle\Phi^{\mathrm{hom}}\rangle_{t}$, which is bounded in terms of $\tau$ since $\wt{\tau}_{\mathrm{sol}}^{\e}\leq\tau<\tau_{\mathrm{sol}}$.) Therefore, it suffices to condition on $\mathcal{E}$ for $0\leq\tau<\tau_{\mathrm{sol}}$ to be determined. By definition, we have 
\begin{align}
\Phi^{\e}(t,x)-\Phi^{\mathrm{hom}}(t,x) \ = \ &\int_{0}^{t}\d s\int_{\partial \bUz}\{\mathbf{T}^{\Phi^{\e}(s,\cdot)}(\d y)-\mathbf{T}^{\Phi^{\mathrm{hom}}(s,\cdot)}(\d y)\}\mathscr{K}(x,y)\mathsf{n}^{\Phi^{\e}(s,\cdot)}(y) \nonumber \\
+ \ &\int_{0}^{t}\d s\int_{\partial \bUz}\mathbf{T}^{\Phi^{\mathrm{hom}}(s,\cdot)}(\d y)\mathscr{K}(x,y)\{\mathsf{n}^{\Phi^{\e}(s,\cdot)}(y)-\mathsf{n}^{\Phi^{\mathrm{hom}}(s,\cdot)}(y)\} \nonumber \\
+ \ &\mathbf{F}^{\e}(t,x)-\mathbf{F}^{\e,\mathrm{eff}}(t,x)+M^{\e}(t,x).
\end{align}
We conditioned on $\mathcal{E}$, so the last line is small in $\|\|_{\wt{\tau}}$. By Lemma \ref{lemma:thm11}, Lemma \ref{lemma:thm12}, the triangle inequality, and regularity of $\mathscr{K}$,
\begin{align*}
&\|\Phi^{\e}-\Phi^{\mathrm{hom}}\|_{\wt{\tau}} \\
\lesssim \ &\int_{0}^{\wt{\tau}}\d s \int_{\partial \bUz}|\mathbf{T}^{\Phi^{\e}(s,\cdot)}(\d y)-\mathbf{T}^{\Phi^{\mathrm{hom}}(s,\cdot)}(\d y)|\|\mathscr{K}(\cdot,y)\mathsf{n}^{\Phi^{\e}(s,\cdot)}(y)\|_{\mathscr{C}^{10}} + b(\e)i(\|\Phi^{\e}\|_{\wt{\tau}}) + \e^{1/3}\\
+ \ &\int_{0}^{\wt{\tau}}\d s \int_{\partial \bUz}\mathbf{T}^{\Phi^{\mathrm{hom}}(s,\cdot)}(\d y)\|\mathscr{K}(\cdot,y)\|_{\mathscr{C}^{10}}|\mathsf{n}^{\Phi^{\e}(s,\cdot)}(y)-\mathsf{n}^{\Phi^{\mathrm{hom}}(s,\cdot)}(y)|\\
\lesssim \ &i(\langle\Phi_{1}\rangle_{\tau},\langle\Phi_{2}\rangle_{\tau})\times \int_{0}^{\wt{\tau}}\d s \int_{\partial \bUz}\mathbf{T}^{\Phi^{\mathrm{hom}}(s,\cdot)}(\d y)\|\Phi^{\e}-\Phi^{\mathrm{hom}}\|_{s}+ b(\e)i(\langle\Phi^{\e}\rangle_{\wt{\tau}}) + \e^{1/3} \\
\lesssim \ &\int_{0}^{\wt{\tau}}\d s \|\Phi^{\e}-\Phi^{\mathrm{hom}}\|_{s}+b(\e)+\e^{1/3},
\end{align*}
where the last bound follows from the fact that for $s\leq\wt{\tau}$, we have uniform control on $\Phi^{\e}$ and its inverse (because we have such control for $\Phi^{\mathrm{hom}}$ and its inverse). By combining the previous display with the Gronwall inequality, we obtain
\begin{align}
\|\Phi^{\e}-\Phi^{\mathrm{hom}}\|_{\wt{\tau}} \ \lesssim_{\Phi^{\mathrm{hom}}} \ b(\e)+\e^{1/3}, \label{eq:thm22}
\end{align}
where the implied constant depends continuously on $\wt{\tau}$ and $\Phi^{\mathrm{hom}}$. Because $\langle\Phi^{\mathrm{hom}}\rangle_{\tau}\lesssim1$, by the inverse function theorem, we can find $\upsilon>0$ sufficiently small such that $\|\Phi-\Phi^{\mathrm{hom}}\|_{\wt{\tau}}\leq\upsilon$ implies $\langle\Phi,\Phi^{\mathrm{hom}}\rangle_{\wt{\tau}}\leq1/99$. We now choose $\e$ sufficiently small depending only on $\Phi^{\mathrm{hom}}$ so that the \abbr{RHS} of \eqref{eq:thm22}, including the implied constant, is less than $\upsilon$. For any such $\e$, we have 
\begin{align}
\langle\Phi^{\e},\Phi^{\mathrm{hom}}\rangle_{\wt{\tau}} \ \leq \ \frac{1}{99}. \label{eq:thm23}
\end{align}
We now claim that \eqref{eq:thm23} implies $\wt{\tau}=\tau\wedge\wt{\tau}_{\mathrm{sol}}^{\e}=\tau$. Indeed, if not, then \eqref{eq:thm23} would hold upon replacing $\wt{\tau}$ by $\wt{\tau}_{\mathrm{sol}}^{\e}$. But this contradicts the definition of $\wt{\tau}_{\mathrm{sol}}^{\e}$. Thus, we deduce \eqref{eq:thm23} holds if we replace $\wt{\tau}$ by $\tau$. This implies $\langle\Phi^{\e}\rangle_{\tau}\leq\langle\Phi^{\mathrm{hom}}\rangle_{\tau}+1/99\lesssim1$. In particular, we get $\tau_{\mathrm{sol}}^{\e}\wedge\tau=\tau$, where $\tau_{\mathrm{sol}}^{\e}$ is the true $\mathscr{C}^{10}_{\simeq}$-explosion time for $\Phi^{\e}$, and this shows the first claim of Theorem \ref{theorem:thm2}. To show the second claim, since $\wt{\tau}=\tau$ (for $\e$ sufficiently small), we know that \eqref{eq:thm22} holds for $\wt{\tau}$ replaced by $\tau$, so $\|\Phi^{\e}-\Phi^{\mathrm{hom}}\|_{\tau}\to0$ as $\e\to0$ in probability. Since $\tau<\tau_{\mathrm{sol}}$, we know $\langle\Phi^{\mathrm{hom}}\rangle_{\tau}\lesssim_{\tau}1$. So we can again use the inverse function theorem to deduce that $\|\Phi^{\e}-\Phi^{\mathrm{hom}}\|_{\tau}\to0$ implies $\langle\Phi^{\e},\Phi^{\mathrm{hom}}\rangle_{\tau}\to0$ in probability as $\e\to0$. This finishes the proof, since $\tau<\tau_{\mathrm{sol}}$ was arbitrary. \qed
%
%
%
\section{Proof of Proposition \ref{prop:hom}}\label{section:homproof}
The remainder of this paper (before the appendix) is dedicated to proving Proposition \ref{prop:hom}. We will do so in two sections. The current section provides a list of the necessary ingredients with a proof of all but one of these ingredients. The only ingredient that we do not prove in this section (Proposition \ref{prop:coupling}) has a technically involved proof that, itself, requires several steps. This proof is the purpose of the next section. In any case, assuming Proposition \ref{prop:coupling}, we then prove Proposition \ref{prop:hom} at the end of this section.
\subsection{Time average}
The first step for showing Proposition \ref{prop:hom} amounts to replacing the integrand in $\mathsf{F}^{\e}$ (see before Proposition \ref{prop:hom}) by its time-average on a time-scale $\tau(\e)$ that satisfies $\tau(\e)\to0$ and $\e^{-1}\delta(\e)\tau(\e)\to\infty$, both as $\e\to0$. (Note that because $\delta(\e)\e^{-1}\to\infty$ by assumption, such $\tau(\e)$ exists. In fact, this is exactly the utility of the assumption $\delta(\e)\e^{-1}\to\infty$.) The point of $\tau(\e)\to0$ is to control the error in this replacement-by-time-average; indeed, on time-scales $\tau(\e)\ll1$, the interface is moving on scales of $\mathrm{o}(1)$. The point of $\e^{-1}\delta(\e)\tau(\e)\to\infty$ will be explained prior to the final step (Lemma \ref{lemma:ergodic}).
\begin{lemma}\label{lemma:timeaverage}
Suppose $\tau(\e)>0$ satisfies $\tau(\e)\to0$ and $\e^{-1}\delta(\e)\tau(\e)\to\infty$ as $\e\to0$. Now, define
\begin{align}
\mathsf{F}^{\e,\mathrm{av}}(t,x) \ := \ \mathbf{1}_{t\geq\tau(\e)}\int_{0}^{t-\tau(\e)}\d s \ \tau(\e)^{-1}\int_{0}^{\tau(\e)}\d r \int_{\partial\bUz}\mathbf{T}^{\Phi^{\e}(s+r,\cdot),\delta}_{\mathbf{z}^{\e}(s+r)}(\d y)\mathscr{K}(x,y)\mathsf{n}^{\Phi^{\e}(s+r,\cdot)}(y).
\end{align}
For any possibly random $\tau\geq0$, we have the deterministic bound $\|\mathsf{F}^{\e}-\mathsf{F}^{\e,\mathrm{av}}\|_{\tau}\lesssim\tau(\e)$. Similarly, we have the deterministic bound $\|\mathsf{F}^{\e,\mathrm{eff}}-\mathsf{F}^{\e,\mathrm{eff},\mathrm{av}}\|_{\tau}\lesssim\tau(\e)$, where $\mathsf{F}^{\e,\mathrm{eff}}$ is from before \emph{Proposition \ref{prop:hom}}, and
\begin{align}
\mathsf{F}^{\e,\mathrm{eff},\mathrm{av}}(t,x) \ := \ \mathbf{1}_{t\geq\tau(\e)}\int_{0}^{t-\tau(\e)}\d s \ \tau(\e)^{-1}\int_{0}^{\tau(\e)}\d r \int_{\partial\bUz}\mathbf{T}^{\Phi^{\e}(s+r,\cdot)}(\d y)\mathscr{K}(x,y)\mathsf{n}^{\Phi^{\e}(s+r,\cdot)}(y).
\end{align}
\end{lemma}
\begin{remark}
Lemma \ref{lemma:timeaverage} is true for any choice of $\tau(\e)\geq0$. We placed conditions on our choice of $\tau(\e)$ in Lemma \ref{lemma:timeaverage} because we will only ever be interested in this case.
\end{remark}
\begin{proof}
The proof for both bounds follows the same argument, so we give details for the first. Take $0\leq t\leq\tau(\e)$. By construction, we have $\mathsf{F}^{\e,\mathrm{av}}(t,\cdot)=0$. By regularity of $\mathscr{K}$, we also have
\begin{align}
\|\mathsf{F}^{\e}(t,\cdot)\|_{\mathscr{C}^{10}} \ \leq \ \int_{0}^{t}\d s\int_{\partial\bUz}\mathbf{T}^{\Phi^{\e}(s,\cdot),\delta}_{\mathbf{z}^{\e}(s)}(\d y)\|\mathscr{K}(\cdot,y)\|_{\mathscr{C}^{10}}|\mathsf{n}^{\Phi^{\e}(s,\cdot)}(y)| \ \lesssim \ t \ \leq \ \tau(\e). \label{eq:timeaverage1}
\end{align}
Therefore, for $0\leq t\leq\tau(\e)$, we have $\|\mathsf{F}^{\e}(t,\cdot)-\mathsf{F}^{\e,\mathrm{av}}(t,\cdot)\|_{\mathscr{C}^{10}}\lesssim\tau(\e)$. So, it suffices to take $\tau\geq\tau(\e)$ and $\tau(\e)\leq t\leq\tau$. In this case, we first rewrite $\mathsf{F}^{\e}(t,x)$ by decomposing the time-integration domain as follows (for any $x\in\partial\bUz$):
\begin{align}
\mathsf{F}^{\e}(t,x) \ &= \ \left(\int_{0}^{t-\tau(\e)}+\int_{t-\tau(\e)}^{t}\right)\d s \int_{\partial\bUz}\mathbf{T}^{\Phi^{\e}(s,\cdot),\delta}_{\mathbf{z}^{\e}(s)}(\d y)\mathscr{K}(x,y)\mathsf{n}^{\Phi^{\e}(s,\cdot)}(y) \ \nonumber \\
& =: \ \mathsf{F}^{\e,1}(t,x)+\mathsf{F}^{\e,2}(t,x). 
\end{align}
Because the length of $[t-\tau(\e),t]$ is equal to $\tau(\e)$, the derivation of \eqref{eq:timeaverage1} also shows $\|\mathsf{F}^{\e,2}(t,\cdot)\|_{\mathscr{C}^{10}}\lesssim\tau(\e)$. So, it suffices to show the deterministic estimate $\|\mathsf{F}^{\e,1}(t,\cdot)-\mathsf{F}^{\e,\mathrm{av}}(t,\cdot)\|_{\mathscr{C}^{10}}\lesssim\tau(\e)$ in order to complete this proof. To this end, we compute via Fubini and change-of-variables to deduce that $\mathsf{F}^{\e,1}(t,x)-\mathsf{F}^{\e,\mathrm{av}}(t,x)$ is equal to
\begin{align}
&\int_{0}^{t-\tau(\e)}\d s \ \tau(\e)^{-1}\int_{0}^{\tau(\e)}\d r \int_{\partial\bUz}\mathscr{K}(x,y)\left\{\mathbf{T}^{\Phi^{\e}(s,\cdot),\delta}_{\mathbf{z}^{\e}(s)}(\d y)\mathsf{n}^{\Phi^{\e}(s,\cdot)}(y)-\mathbf{T}^{\Phi^{\e}(s+r,\cdot),\delta}_{\mathbf{z}^{\e}(s+r)}(\d y)\mathsf{n}^{\Phi^{\e}(s+r,\cdot)}(y)\right\} \nonumber \\
= \ &\tau(\e)^{-1}\int_{0}^{\tau(\e)}\d r \int_{0}^{t-\tau(\e)}\d s \int_{\partial\bUz}\mathscr{K}(x,y)\left\{\mathbf{T}^{\Phi^{\e}(s,\cdot),\delta}_{\mathbf{z}^{\e}(s)}(\d y)\mathsf{n}^{\Phi^{\e}(s,\cdot)}(y)-\mathbf{T}^{\Phi^{\e}(s+r,\cdot),\delta}_{\mathbf{z}^{\e}(s+r)}(\d y)\mathsf{n}^{\Phi^{\e}(s+r,\cdot)}(y)\right\}
\nonumber  \\
= \ &\tau(\e)^{-1}\int_{0}^{\tau(\e)}\d r \left(\int_{0}^{t-\tau(\e)}-\int_{r}^{t-\tau(\e)+r}\right)\d s \int_{\partial\bUz}\mathbf{T}^{\Phi^{\e}(s,\cdot),\delta}_{\mathbf{z}^{\e}(s)}(\d y)\mathscr{K}(x,y)\mathsf{n}^{\Phi^{\e}(s,\cdot)}(y).
\end{align}
Given any $0\leq r\leq\tau(\e)$, we know that the symmetric difference between the domains $[0,t-\tau(\e)]$ and $[r,t-\tau(\e)+r]$ has length $\lesssim r\leq\tau(\e)$. Thus, the derivation of \eqref{eq:timeaverage1} also shows that the last line in the previous display is an average of terms whose $\|\|_{\mathscr{C}^{10}}$-norm is $\lesssim\tau(\e)$. Therefore, the desired estimate $\|\mathsf{F}^{\e,1}(t,\cdot)-\mathsf{F}^{\e,\mathrm{av}}(t,\cdot)\|_{\mathscr{C}^{10}}\lesssim\tau(\e)$ follows.
\end{proof}
\subsection{Freezing the metric}
Again, the proof of Proposition \ref{prop:hom} is based on showing relaxation of the particle process $t\mapsto\mathbf{z}^{\e}(t)$ to its invariant measure. A (technical) obstacle to proving something of this type, however, is that the metric which determines the invariant measure of $\mathbf{z}^{\e}$ is changing in time. We now start to resolve this difficulty. In particular, we bound the error in replacing $\Phi^{\e}(s+r,\cdot)\mapsto\Phi^{\e}(s,\cdot)$ in the definition of $\mathsf{F}^{\e,\mathrm{av}}$ in Lemma \ref{lemma:timeaverage}. (Thus, we bound the cost in freezing the metric that determines the next bump of the process $\Phi^{\e}$. We note $\mathbf{z}^{\e}(s+r)$ is still defined with respect to an evolving metric; replacing it by a particle with frozen metric is the goal of the next step Proposition \ref{prop:coupling}.)

We first provide a general estimate for $z\mapsto\mathbf{T}^{\Phi,\delta}_{z}$ for general $\Phi,z$. This will be useful beyond the current step. (We emphasize that the following is \emph{not} an honest continuity statement because of the $\alpha(\e)$ term in the upper bound. It only states that the next jump of $\mathbf{z}^{\e}$ should be close to its current position, because its next jump lands where the boundary local time of a Brownian motion has accumulated $\delta(\e)=\mathrm{o}(1)$. In fact, this is exactly the utility of the assumption $\delta(\e)=\mathrm{o}(1)$. Continuity in $z$ is probably true without this assumption, but its proof would be more complicated, and this case is already likely treatable via methods of \cite{DGHS}.)
\begin{lemma}\label{lemma:freezeprelim}
Fix $\Phi_{1},\Phi_{2}\in\mathscr{C}^{10}_{\simeq}(\partial\bUz,\R^{\d})$. There exists $\alpha(\e)\to0$ as $\e\to0$ so that uniformly over $z_{1},z_{2}\in\partial\bUz$, we have the following deterministic upper bound, in which $i(\cdot,\cdot)\geq1$ is jointly continuous and increasing in both entries:
\begin{align}
\|\int_{\partial\bUz}\left(\mathbf{T}^{\Phi_{1},\delta}_{z_{1}}(\d y)-\mathbf{T}^{\Phi_{2},\delta}_{z_{2}}(\d y)\right)\mathscr{K}(\cdot,y)\mathsf{n}^{\Phi_{1}}(y)\|_{\mathscr{C}^{10}} \ \lesssim \ \left[|z_{1}-z_{2}|+\alpha(\e)\right]\times i(\langle\Phi_{1}\rangle_{\mathscr{C}^{10}},\langle\Phi_{2}\rangle_{\mathscr{C}^{10}}). 
\end{align}
\end{lemma}
\begin{proof}
{{}Recall that $\mathbf{T}^{\Phi,\delta}_{z}(\d y)$ is a probability measure on $\partial\bUz$.} Thus, we can rewrite
\begin{align}
\int_{\partial\bUz}\left(\mathbf{T}^{\Phi_{1},\delta}_{z_{1}}(\d y)-\mathbf{T}^{\Phi_{2},\delta}_{z_{2}}(\d y)\right) & \mathscr{K}(\cdot,y)\mathsf{n}^{\Phi_{1}}(y) \  \  \nonumber \\
= & \int_{\partial\bUz}\mathbf{T}^{\Phi_{1},\delta}_{z_{1}}(\d y)\left\{\mathscr{K}(\cdot,y)\mathsf{n}^{\Phi_{1}}(y)-\mathscr{K}(\cdot,z_{1})\mathsf{n}^{\Phi_{1}}(z_{1})\right\} \label{eq:freezeprelim1a}\\
& - \ \int_{\partial\bUz}\mathbf{T}^{\Phi_{2},\delta}_{z_{2}}(\d y)\left\{\mathscr{K}(\cdot,y)\mathsf{n}^{\Phi_{1}}(y)-\mathscr{K}(\cdot,z_{2})\mathsf{n}^{\Phi_{1}}(z_{2})\right\}\label{eq:freezeprelim1b}\\
& + \ \mathscr{K}(\cdot,z_{1})\mathsf{n}^{\Phi_{1}}(z_{1})-\mathscr{K}(\cdot,z_{2})\mathsf{n}^{\Phi_{1}}(z_{2}).\label{eq:freezeprelim1c}
\end{align}
By Lemma \ref{lemma:thm11}, we deduce $\|\eqref{eq:freezeprelim1c}\|_{\mathscr{C}^{10}}\lesssim i(\langle\Phi_{1}\rangle_{\mathscr{C}^{10}})|z_{1}-z_{2}|$ for an increasing and continuous $i(\cdot)\geq1$. (We clarify that the $\mathscr{C}^{10}$-norm in this bound is with respect to the first input of $\mathscr{K}$, so we only need to control the first derivative of $\mathsf{n}^{\Phi_{1}}$, which Lemma \ref{lemma:thm11} does, and not its first 11 derivatives, for example. In particular, we only need smoothness of $\mathscr{K}$. Also, technically, the metric $|z_{1}-z_{2}|$ should be replaced by the geodesic distance coming from the Euclidean hyper-surface metric on $\partial\bUz$, which is related to the $\mathsf{g}^{\Phi_{1}}$-metric through the factor $i(\langle\Phi_{1}\rangle_{\mathscr{C}^{10}})$. But these metrics are comparable up to a factor depending only on $\partial\bUz$, since $\partial\bUz$ is a smooth hyper-surface by assumption.) By the same token, for \eqref{eq:freezeprelim1a}-\eqref{eq:freezeprelim1b}, we have (for $z=z_{1},z_{2}$)
\begin{align}
\|\int_{\partial\bUz}\mathbf{T}^{\Phi_{1},\delta}_{z}(\d y)\left\{\mathscr{K}(\cdot,y)\mathsf{n}^{\Phi_{1}}(y)-\mathscr{K}(\cdot,z)\mathsf{n}^{\Phi_{1}}(z)\right\}\|_{\mathscr{C}^{10}} \ \lesssim \ i(\langle\Phi_{1}\rangle_{\mathscr{C}^{10}})\int_{\partial\bUz}\mathbf{T}^{\Phi_{1},\delta}_{z}(\d y)|y-z|.
\end{align}
Indeed, the $\mathscr{C}^{10}$-norm on the \abbr{LHS} is with respect to the first input of $\mathscr{K}$ and is unrelated to the integration variable, so the previous display follows first by triangle inequality and then by the previous paragraph. Now, recall that $\mathbf{T}^{\Phi_{1},\delta}_{z}(\d y)$ is the transition kernel for $z\mapsto\mathbf{x}^{\Phi_{1},z}(\tau^{\Phi_{1},z,\delta})$ (see Definition \ref{definition:interface}). The \abbr{RHS} of the previous display is therefore the expectation of $|\mathbf{x}^{\Phi_{1},z}(\tau^{\Phi_{1},z,\delta})-z|$. We now note that the continuous-time process $t\mapsto\mathbf{x}^{\Phi_{1},z}(\tau^{\Phi_{1},z,t})$ is a strong Feller process whose generator is the Dirichlet-to-Neumann operator that depends only on the first two derivatives of $\Phi_{1}$ and $\Phi_{1}^{-1}$ (the latter of which exists because $\Phi_{1}$ is invertible by assumption); see \cite{ARP,EO,Hsu1} for these facts. Therefore, as $|z|+|\mathbf{x}^{\Phi_{1},z}|\lesssim1$ (they belong to a fixed compact set $\bUz$) and since $\delta\to0$ as $\e\to0$ by assumption, we deduce that the \abbr{RHS} of the previous display vanishes as $\e\to0$ with a rate depending only on $\langle\Phi_{1}\rangle_{\mathscr{C}^{10}}$. This finishes the proof.
\end{proof}
Equipped with Lemmas \ref{lemma:thm11} and \ref{lemma:freezeprelim}, we can now justify $\mathbf{T}^{\Phi^{\e}(s+r,\cdot)}_{\mathbf{z}^{\e}(s+r)}\mapsto\mathbf{T}^{\Phi^{\e}(s,\cdot)}_{\mathbf{z}^{\e}(s+r)}$ in the definition of $\mathsf{F}^{\e,\mathrm{av}}$ in Lemma \ref{lemma:timeaverage}.
\begin{lemma}\label{lemma:freeze}
Recall $\tau(\e)$ and $\mathsf{F}^{\e,\mathrm{av}}$ from \emph{Lemma \ref{lemma:timeaverage}}. Fix any stopping time $0\leq\tau\lesssim1$. Now, for $t\geq0$ and $x\in\partial\bUz$, define
\begin{align}
\wt{\mathsf{F}}^{\e,\mathrm{av}}(t,x) \ := \ \mathbf{1}_{t\geq\tau(\e)}\int_{0}^{t-\tau(\e)}\d s \ \tau(\e)^{-1}\int_{0}^{\tau(\e)}\d r \int_{\partial\bUz}\mathbf{T}^{\Phi^{\e}(s,\cdot),\delta}_{\mathbf{z}^{\e}(s+r)}(\d y)\mathscr{K}(x,y)\mathsf{n}^{\Phi^{\e}(s,\cdot)}(y).
\end{align}
There exists increasing and continuous $i(\cdot)\geq1$ and $\alpha(\e)\to0$ as $\e\to0$ such that 
\begin{align}
\E[\|\mathsf{F}^{\e,\mathrm{av}}-\wt{\mathsf{F}}^{\e,\mathrm{av}}\|_{\tau}i(\langle\Phi^{\e}\rangle_{\tau})^{-1}]\lesssim\tau(\e)+\alpha(\e).
\end{align}
\end{lemma}
\begin{proof}
By construction, we know $\mathsf{F}^{\e,\mathrm{av}}(t,\cdot)-\wt{\mathsf{F}}^{\e,\mathrm{av}}(t,\cdot)=0$ if $0\leq t\leq\tau(\e)$, so take $\tau(\e)\leq t\leq\tau$. We first write
\begin{align}
\mathsf{F}^{\e,\mathrm{av}}(t,x)-\wt{\mathsf{F}}^{\e,\mathrm{av}}(t,x) \ = \ \mathsf{E}^{\e,1}(t,x) + \mathsf{E}^{\e,2}(t,x),
\end{align}
where $\mathsf{E}^{\e,1}$ (resp. $\mathsf{E}^{\e,2}$) comes from replacing $\Phi^{\e}(s+r,\cdot)\mapsto\Phi^{\e}(s,\cdot)$ in the $\mathbf{T}$-term (resp. $\mathsf{n}$-term):
\begin{align*}
\mathsf{E}^{\e,1}(t,x) \ := \ &\int_{0}^{t-\tau(\e)}\d s \ \tau(\e)^{-1}\int_{0}^{\tau(\e)}\d r \int_{\partial\bUz}\left\{\mathbf{T}^{\Phi^{\e}(s+r,\cdot),\delta}_{\mathbf{z}^{\e}(s+r)}(\d y)-\mathbf{T}^{\Phi^{\e}(s,\cdot),\delta}_{\mathbf{z}^{\e}(s+r)}(\d y)\right\}\mathscr{K}(x,y)\mathsf{n}^{\Phi^{\e}(s+r)}(y) \\
\mathsf{E}^{\e,2}(t,x) \ := \ &\int_{0}^{t-\tau(\e)}\d s \ \tau(\e)^{-1}\int_{0}^{\tau(\e)}\d r \int_{\partial\bUz}\mathbf{T}^{\Phi^{\e}(s,\cdot),\delta}_{\mathbf{z}^{\e}(s+r)}(\d y)\mathscr{K}(x,y)\left\{\mathsf{n}^{\Phi^{\e}(s+r,\cdot)}(y)-\mathsf{n}^{\Phi^{\e}(s,\cdot)}(y)\right\}.
\end{align*}
To control $\mathsf{E}^{\e,1}$, we use regularity of $\mathscr{K}$ and Lemma \ref{lemma:freezeprelim}. This gives, for increasing and continuous $i(\cdot,\cdot),\wt{i}(\cdot)\geq1$, that
\begin{align*}
\|\mathsf{E}^{\e,1}(t,\cdot)\|_{\mathscr{C}^{10}} \ \lesssim \ \alpha(\e)\int_{0}^{t-\tau(\e)}
\!\!
\d s \ \tau(\e)^{-1}\int_{0}^{\tau(\e)} \!\! \d r \ i(\langle\Phi^{\e}(s+r,\cdot)\rangle_{\mathscr{C}^{10}},\langle\Phi^{\e}(s,\cdot)\rangle_{\mathscr{C}^{10}}) \ \lesssim_{} \ \alpha(\e)\times i(\langle\Phi^{\e}\rangle_{\tau}).
\end{align*}
We now control $\mathsf{E}^{\e,2}$, this time by using regularity of $\mathscr{K}$ and Lemma \ref{lemma:thm11}. We deduce
\begin{align}
&\|\mathsf{E}^{\e,2}(t,\cdot)\|_{\mathscr{C}^{10}} \nonumber\\
&\lesssim_{} \ \int_{0}^{t-\tau(\e)}\d s \ \tau(\e)^{-1}\int_{0}^{\tau(\e)}\d r \int_{\partial\bUz}\mathbf{T}^{\Phi^{\e}(s,\cdot),\delta}_{\mathbf{z}^{\e}(s+r)}(\d y)\|\mathscr{K}(\cdot,y)\|_{\mathscr{C}^{10}}|\mathsf{n}^{\Phi^{\e}(s+r,\cdot)}(y)-\mathsf{n}^{\Phi^{\e}(s,\cdot)}(y)| \nonumber \\
&\lesssim \ \int_{0}^{t-\tau(\e)}\d s \ \tau(\e)^{-1}\int_{0}^{\tau(\e)}\d r \left\{\int_{\partial\bUz}\mathbf{T}^{\Phi^{\e}(s,\cdot)}_{\mathbf{z}^{\e}(s+r)}(\d y)\right\}\times i(\langle\Phi^{\e}(s+r,\cdot)\rangle_{\mathscr{C}^{10}})\|\Phi^{\e}(s+r,\cdot)-\Phi^{\e}(s,\cdot)\|_{\mathscr{C}^{10}} \nonumber\\
&\lesssim \ i(\langle\Phi^{\e}\rangle_{\tau})\int_{0}^{t-\tau(\e)}\d s \ \tau(\e)^{-1}\int_{0}^{\tau(\e)}\d r \|\Phi^{\e}(s+r,\cdot)-\Phi^{\e}(s,\cdot)\|_{\mathscr{C}^{10}}. \nonumber
\end{align}
We extend $[0,t-\tau(\e)]$ to $[0,\mathfrak{t}-\tau(\e)]$ for some deterministic $\tau\leq\mathfrak{t}\lesssim1$ in the last line for an upper bound. This gives an estimate on $\|\mathsf{E}^{\e,2}(t,\cdot)\|_{\mathscr{C}^{10}}$ uniformly in $\tau(\e)\leq t\leq\tau$, thus on $\|\mathsf{E}^{\e,2}\|_{\tau}$. So by the Markov inequality, Fubini, and noting that $\Phi^{\e}$ is a jump process whose jumps are explicitly given by $\e\mathscr{K}$-sized bumps at jump locations of $\mathbf{z}^{\e}$, we deduce the following estimate in which $i(\cdot)\geq1$ is continuous and increasing, and where the sum in the second line below is over all jump times $s\leq\tau_{i}\leq s+\tau(\e)$:
\begin{align}
\E[\|\mathsf{E}^{\e,2}\|_{\tau}i(\langle\Phi^{\e}\rangle_{\tau})^{-1}] \ \lesssim \ &\int_{0}^{\mathfrak{t}-\tau(\e)}\d s \ \tau(\e)^{-1}\int_{0}^{\tau(\e)}\d r \E\|\Phi^{\e}(s+r,\cdot)-\Phi^{\e}(s,\cdot)\|_{\mathscr{C}^{10}} \nonumber \\
\lesssim \ &\int_{0}^{\mathfrak{t}-\tau(\e)}\d s \ \tau(\e)^{-1}\int_{0}^{\tau(\e)}\d r \E\Big\{\sum_{s\leq\tau_{i}\leq s+\tau(\e)}\e\|\mathscr{K}(\cdot,\mathbf{z}^{\e}(\tau_{i}))\|_{\mathscr{C}^{10}}\Big\} 
\nonumber \\
\lesssim \ &\int_{0}^{\mathfrak{t}-\tau(\e)}\d s \ \tau(\e)^{-1}\int_{0}^{\tau(\e)}\d r \ \e\times\E|\#\{s\leq\tau_{i}\leq s +\tau(\e)\}| \ \lesssim \ \tau(\e).
\end{align}
(The third line follows by regularity of $\mathscr{K}$ and by noting that the number of jump times $s\leq\tau_{i}\leq s+\tau(\e)$ is Poisson with speed $\e^{-1}\tau(\e)$.) Combining this with our $\mathsf{E}^{\e,1}$ estimate and the first display of this proof finishes the argument.
\end{proof}
\subsection{Coupling}
Let us explain the goal of this step. Take a time-interval $\mathbb{I}$ of length $\tau(\e)$. We want to replace the particle process $t\mapsto\mathbf{z}^{\e}(t)$ (for $t\in\mathbb{I}$) by a process $t\mapsto\wt{\mathbf{z}}^{\e}(\inf\mathbb{I},t)$ (also for $t\in\mathbb{I}$) that is defined in the same way as in Definition \ref{definition:interface} but upon replacing the time-dependent metric $\mathsf{g}^{\Phi^{\e}(t,\cdot)}$ by the frozen metric $\mathsf{g}^{\Phi^{\e}(\inf\mathbb{I},\cdot)}$. (More precisely, we want to build such a process $\wt{\mathbf{z}}^{\e}$ and couple it to $\mathbf{z}^{\e}$ in a way that, at least with high probability, stays within $\mathrm{o}(1)$ to $\mathbf{z}^{\e}$.)
\begin{prop}\label{prop:coupling}
{{}There exists $\tau(\e)>0$ which satisfies $\tau(\e)\to0$ and $\e^{-1}\delta(\e)\tau(\e)\to\infty$ as $\e\to0$, as well as a continuous and increasing function $i(\cdot)\geq1$ such that the following is true.

Take any stopping time $0\leq\tau\lesssim1$ and $(t,x)\in[0,\tau-\tau(\e)]\times\partial\bUz$. Take any $0\leq s\leq t-\tau(\e)$. There is a process $r\mapsto\wt{\mathbf{z}}^{\e}(s,r)$ that satisfies the following constraints. First, we have the estimate
\begin{align}
\sup_{0\leq r\leq\tau(\e)}\E[|\mathbf{z}^{\e}(s+r)-\wt{\mathbf{z}}^{\e}(s,r)|\times i(\langle\Phi^{\e}\rangle_{\tau})^{-1}] \ \leq \ \omega(\e), \label{eq:couplingI}
\end{align}
where $\omega(\e)\to0$ as $\e\to0$. Second, we require the process $r\mapsto\wt{\mathbf{z}}^{\e}(s,r)$ to be a continuous-time Markov process valued in $\partial\bUz$ whose infinitesimal generator is given by 
\begin{align}
\mathscr{L}^{\e,s}F(z) \ := \ \e^{-1}\int_{\partial\bUz}\mathbf{T}^{\Phi^{\e}(s,\cdot),\delta}_{z}(\d y)\left\{F(y)-F(z)\right\}.
\end{align}
(Above, $\Phi^{\e}(s,\cdot)$ is the interface process at the fixed time $s$.) This finishes the constraints for $\wt{\mathbf{z}}^{\e}$. Moreover, we have $\E[\|\wt{\mathsf{F}}^{\e,\mathrm{av}}-\mathsf{H}^{\e}\|_{\tau}i(\langle\Phi^{\e}\rangle_{\tau})^{-2}]\lesssim\zeta(\e)$, where $\zeta(\e)\to0$ and
\begin{align}
\mathsf{H}^{\e}(t,x) \ := \ \mathbf{1}_{t\geq\tau(\e)}\int_{0}^{t-\tau(\e)}\d s\ \tau(\e)^{-1}\int_{0}^{\tau(\e)}\d r \int_{\partial\bUz}\mathbf{T}^{\Phi^{\e}(s,\cdot),\delta}_{\wt{\mathbf{z}}^{\e}(s,r)}(\d y)\mathscr{K}(x,y)\mathsf{n}^{\Phi^{\e}(s,\cdot)}(y).
\end{align}
}
\end{prop}
Again, as noted at the beginning of this section, we will defer the proof of Proposition \ref{prop:coupling} to the next section. Let us comment briefly on it for now. As we already have used multiple times, on time scales $\lesssim\tau(\e)$, we expect the interface process to change by $\mathrm{o}(1)$. Thus, as long as the diffeomorphism class of the interface does not change (which is reflected in the $i(\langle\Phi^{\e}\rangle_{\tau})^{-1}$ factor used in the estimates in Proposition \ref{prop:coupling}), we expect that a reflecting Brownian motion with respect to the evolving metric and that with respect to the frozen metric to be close (under appropriate coupling). If the evolution of the metric is independent from that of the Brownian motion, this is an immediate consequence of parabolic regularity for the associated Kolmogorov equations. On the other hand, if there were no singular reflection/local time term in the \abbr{SDE}s for these reflecting Brownian motions, even in the case where the metric is adapted to the filtration generated by these Brownian motions, this would follow by Ito calculus. Hence, the key technical issue behind proving Proposition \ref{prop:coupling} ultimately boils down Ito calculus for \abbr{SDE}s with singular boundary local time coefficients. For this, we turn to a fairly geometric framework that is spelled out in the next section.
\subsection{Ergodic lemma}
So far, we have provided the necessary ingredients to replace $\mathsf{F}^{\e}$ (see before Proposition \ref{prop:hom}) by $\mathsf{H}^{\e}$ (see Proposition \ref{prop:coupling}). So, by Lemma \ref{lemma:timeaverage}, it is left to replace $\mathsf{H}^{\e}$ by $\mathsf{F}^{\e,\mathrm{eff},\mathrm{av}}$ (see Lemma \ref{lemma:timeaverage} for notation). For this, note the scale-$\tau(\e)$ integral in $\mathsf{H}^{\e}$ amounts to adding $\mathscr{K}$-shaped bumps at jump locations of the trace for a reflecting Brownian motion whose defining metric is independent of the integration-variable $r$ and dependent only on the variable $s$. Thus, we can establish homogenization for the $\d r$ integration therein and replace $\mathbf{T}^{\Phi^{\e}(s,\cdot),\delta}_{\wt{\mathbf{z}}^{\e}(s,r)}(\d y)\mapsto\mathbf{T}^{\Phi^{\e}(s,\cdot)}(\d y)$. This is the purpose of the next result.
\begin{lemma}\label{lemma:ergodic}
Recall $\mathsf{H}^{\e}$ from \emph{Proposition \ref{prop:coupling}}, and recall $\tau(\e)$ and $\mathsf{F}^{\e,\mathrm{eff},\mathrm{av}}$ from \emph{Lemma \ref{lemma:timeaverage}}. Fix a stopping time $0\leq\tau\lesssim1$. There exist $\ell(\e)\to0$ as $\e\to0$ and a continuous and increasing $i(\cdot)\geq1$ such that $\E[\|\mathsf{H}^{\e}-\mathsf{F}^{\e,\mathrm{eff},\mathrm{av}}\|_{\tau}i(\langle\Phi^{\e}\rangle_{\tau})^{-1}]\lesssim\ell(\e)$.
\end{lemma}
\begin{proof}
First take $0\leq t\leq\tau(\e)$. In this case, we have $\mathsf{H}^{\e}(t,x)=0$ by construction. Moreover, by the proof of \eqref{eq:timeaverage1}, we have the short-time bound $\|\mathsf{F}^{\e,\mathrm{eff},\mathrm{av}}(t,\cdot)\|_{\mathscr{C}^{10}}\lesssim\tau(\e)$. These are deterministic, for the rest of this proof, so we are left to consider times $\tau(\e)\leq t\leq\tau\lesssim1$. In this case, we have the following calculation (which follows basically by definition):
\begin{align}
&\mathsf{H}^{\e}(t,x)-\mathsf{F}^{\e,\mathrm{eff},\mathrm{av}}(t,x) \nonumber\\
&= \ \int_{0}^{t-\tau(\e)}\d s \ \tau(\e)^{-1}\int_{0}^{\tau(\e)}\d r \int_{\partial\bUz}\left\{\mathbf{T}^{\Phi^{\e}(s,\cdot),\delta}_{\wt{\mathbf{z}}^{\e}(s,r)}(\d y)-\mathbf{T}^{\Phi^{\e}(s,\cdot)}(\d y)\right\}\mathscr{K}(x,y)\mathsf{n}^{\Phi^{\e}(s,\cdot)}(y) \nonumber\\
&- \ \int_{t-\tau(\e)}^{t}\d s \int_{\partial\bUz}\mathbf{T}^{\Phi^{\e}(s,\cdot)}(\d y)\mathscr{K}(x,y)\mathsf{n}^{\Phi^{\e}(s,\cdot)}(y) \ =: \ \mathsf{E}^{\e,3}(t,x)+\mathsf{E}^{\e,4}(t,x). 
\end{align}
The proof of \eqref{eq:timeaverage1} shows that $\|\mathsf{E}^{\e,4}(t,\cdot)\|_{\mathscr{C}^{10}}\lesssim\tau(\e)$, so we turn to $\mathsf{E}^{\e,3}$. In particular, to complete the proof, it suffices to show the following, in which we extend $\mathsf{E}^{\e,3}(t,x)$ to $0\leq t\leq\tau(\e)$ by zero:
\begin{align}
\E[\|\mathsf{E}^{\e,3}\|_{\tau}i(\langle\Phi^{\e}\rangle_{\tau})^{-1}] \ \lesssim \ \ell(\e). \label{eq:ergodic0}
\end{align}
For this, we first make a couple of important remarks. 
\begin{enumerate}
\item First, we claim that the invariant measure of $r\mapsto\wt{\mathbf{z}}^{\e}(s,r)$ is the measure $\mathbf{T}^{\Phi^{\e}(s,\cdot)}(\d y)$ from \eqref{eq:limit}. To see this, we recall from Proposition \ref{prop:coupling} that $r\mapsto\wt{\mathbf{z}}^{\e}(s,r)$ is a continuous-time Poisson jump process valued in $\partial\bUz$ with explicit generator. So, it suffices to show that for any $F\in\mathscr{C}^{\infty}(\partial\bUz,\R)$, we have 
\begin{align}
&\int_{\partial\bUz}\mathbf{T}^{\Phi^{\e}(s,\cdot)}(\d z)\mathscr{L}^{\e,s}F(z) \nonumber\\
&= \ \e^{-1}\int_{\partial\bUz}\int_{\partial\bUz}\mathbf{T}^{\Phi^{\e}(s,\cdot)}(\d z)\mathbf{T}^{\Phi^{\e}(s,\cdot),\delta}_{z}(\d y)F(y) - \e^{-1}\int_{\partial\bUz}\mathbf{T}^{\Phi^{\e}(s,\cdot)}(\d z)F(z) \ = \ 0. \nonumber
\end{align}
The first identity follows by construction of $\mathscr{L}^{\e,s}$ from Proposition \ref{prop:coupling}. To show the second identity, we recall $\mathbf{T}^{\Phi^{\e}(s,\cdot),\delta}_{z}(\d y)$ is the distribution of $\mathbf{x}^{\Phi^{\e}(s,\cdot),z}$ stopped at the first time its boundary local time hits $\delta$, assuming its initial data is $z\in\partial\bUz$. (We recall that $\mathbf{x}^{\Phi^{\e}(s,\cdot),z}$ from Definition \ref{definition:interface} as a reflecting Brownian motion on $\bUz$ with metric $\mathsf{g}^{\Phi^{\e}(s,\cdot)}$.) Of course, this is the same as the distribution of the process $t\mapsto\mathbf{x}^{\Phi^{\e}(s,\cdot),z}(\tau^{\Phi^{\e}(s,\cdot),z,t})$ for $t=\delta$, where $\tau^{\Phi^{\e}(s,\cdot),z,t}$ from Definition \ref{definition:interface} is the first time the local time hits $t$. But this latter process has a generator given by the Dirichlet-to-Neumann operator, whose invariant measure is induced by the restriction $\mathsf{g}^{\Phi^{\e}(s,\cdot)}|_{\partial\bUz}$; see \cite{ARP,EO,Hsu0}. However, by change-of-variables, this is exactly $\mathbf{T}^{\Phi^{\e}(s,\cdot)}(\d y)$; see Definition \ref{definition:interface} and the paragraph after \eqref{eq:limit}. So, the law of $\mathbf{x}^{\Phi^{\e}(s,\cdot),z}$ stopped when its boundary local time hits $\delta$, assuming $z$ is distributed according to $\mathbf{T}^{\Phi^{\e}(s,\cdot)}(\d z)$, is given by $\mathbf{T}^{\Phi^{\e}(s,\cdot)}(\d y)$. Thus, the inner integral in the previous display (which is over the $z$-variable) equals $\mathbf{T}^{\Phi^{\e}(s,\cdot)}(\d y)$. The last identity in the previous display follows.
\item Not only can we compute the invariant measure of $r\mapsto\wt{\mathbf{z}}^{\e}(s,r)$ to be $\mathbf{T}^{\Phi^{\e}(s,\cdot)}(\d y)$, but we also know (the symmetric part) of its generator has a positive spectral gap of $\gtrsim\e^{-1}\delta(\e)$ that depends continuously only on $\langle\Phi^{\e}(s,\cdot)\rangle_{\mathscr{C}^{10}}$ (and on the geometry of $\partial\bUz$, though this point will not be important since $\partial\bUz$ is fixed). To see this, we note that $t\mapsto\mathbf{x}^{\Phi^{\e}(s,\cdot),z}(\tau^{\Phi^{\e}(s,\cdot),z,t})$ has a spectral gap that is uniformly positive (depending continuously on $\langle\Phi^{\e}(s,\cdot)\rangle_{\mathscr{C}^{10}}$). Indeed, this claim follows from \cite{ARP}. It then suffices to realize that $r\mapsto\wt{\mathbf{z}}^{\e}(s,r)$ is a continuous-time Poisson process whose jumps happen at speed $\e^{-1}$, where the jumps are sampled according to time-$\delta$ increments of $t\mapsto\mathbf{x}^{\Phi^{\e}(s,\cdot),z}(\tau^{\Phi^{\e}(s,\cdot),z,t})$.
\end{enumerate}
Let us proceed with the proof. First, by point (1) above, we can first write
\begin{align}
\int_{\partial\bUz}\mathbf{T}^{\Phi^{\e}(s,\cdot)}(\d y)\mathscr{K}(x,y)\mathsf{n}^{\Phi^{\e}(s,\cdot)}(y) \ = \ \E^{z,s,\mathrm{inv}}\int_{\partial\bUz}\mathbf{T}^{\Phi^{\e}(s,\cdot),\delta}_{z}(\d y)\mathscr{K}(x,y)\mathsf{n}^{\Phi^{\e}(s,\cdot)}(y),
\end{align}
where $\E^{z,s,\mathrm{inv}}$ is expectation with respect to $z\sim\mathbf{T}^{\Phi^{\e}(s,\cdot)}(\d y)$. (This identity follows from the fact that $\mathbf{T}^{\Phi^{\e}(s,\cdot)}(\d y)$ is invariant for the transition kernel $\mathbf{T}^{\Phi^{\e}(s,\cdot),\delta}_{z}(\d y)$.) This allows us to rewrite $\mathsf{E}^{\e,3}$ via something minus its invariant measure expectation, so
\begin{align}
\mathsf{E}^{\e,3}(t,x) \ = \ \int_{0}^{t-\tau(\e)}\d s \ \tau(\e)^{-1}\int_{0}^{\tau(\e)}\d r\left\{\psi_{x}(\wt{\mathbf{z}}^{\e}(s,r))-\E^{z,s,\mathrm{inv}}\psi_{x}\right\}, \label{eq:ergodic1}
\end{align}
where $\psi_{x}$ is given by the following function of $x,z\in\partial\bUz$:
\begin{align}
\psi_{x}(z) \ := \ \int_{\partial\bUz}\mathbf{T}^{\Phi^{\e}(s,\cdot),\delta}_{z}(\d y)\mathscr{K}(x,y)\mathsf{n}^{\Phi^{\e}(s,\cdot)}(y). \label{eq:ergodic1b}
\end{align}
At this point, bounding $\mathsf{E}^{\e,3}$ amounts to estimating the long-time average of a function that is centered with respect to the invariant measure. (It is long time because the speed of $\wt{\mathbf{z}}^{\e}$ is roughly $\e^{-1}\delta(\e)$ by point (2) above, and $\e^{-1}\delta(\e)\tau(\e)\to\infty$ by assumption.) The rest of this argument makes this precise. We claim that to prove \eqref{eq:ergodic1}, it suffices to instead prove the following pointwise bound for $\tau(\e)\leq t\lesssim1$ and $x\in\partial\bUz$:
\begin{align}
\E\left[\mathbf{1}_{t\leq\tau}|\mathsf{E}^{\e,3}(t,x)|\times i(\langle\Phi^{\e}\rangle_{\tau})^{-1}\right] \ \to_{\e\to0} \ 0. \label{eq:ergodic2}
\end{align}
To justify the sufficiency of \eqref{eq:ergodic2}, we point out that by regularity of the $\mathscr{K}$ kernel and boundedness of the normal vector $\mathsf{n}^{\Phi}$, we clearly have $\|\mathsf{E}^{\e,3}\|_{\mathfrak{t},N}\times i(\langle\Phi^{\e}\rangle_{\mathfrak{t}})^{-1}\leq\|\mathsf{E}^{\e,3}\|_{\mathfrak{t},N}\lesssim_{N,\mathfrak{t}}1$ uniformly in $\e$ for any $N$ and $\mathfrak{t}\lesssim1$. Therefore, by Arzela-Ascoli, we know that $\mathsf{E}^{\e,3}$ is tight in $\mathscr{C}([0,\mathfrak{t}],\mathscr{C}^{10}(\partial\bUz,\R^{\d}))$ for any $\mathfrak{t}\lesssim1$, and therefore so is the stopped process (given by replacing the time-variable $t$ by $t\wedge\tau$) since $\tau\lesssim1$ by assumption. But \eqref{eq:ergodic2} would show that every limit point of $i(\langle\Phi^{\e}\rangle_{\tau})^{-1}\mathsf{E}^{\e,3}(t\wedge\tau,x)$ is zero, and thus \eqref{eq:ergodic0} would hold. (Alternatively, the aforementioned a priori regularity of $\mathsf{E}^{\e,3}$ would let us approximate the the $\|\|_{\tau}$-norm on the \abbr{LHS} of \eqref{eq:ergodic0} by a supremum over $\eta^{-1}$-many points, where $\eta>0$ is fixed for now, up to an error that vanishes as $\eta\to0$. To estimate the supremum over $\eta^{-1}$-many space-time points, multiply \eqref{eq:ergodic2} by $\eta^{-1}$. This implies \eqref{eq:ergodic0} holds if we replace $\ell(\e)$ by something that vanishes as $\eta\to0$. But $\eta>0$ was arbitrary, so we can take it to vanish sufficiently slowly as we take $\e\to0$, for example.) We are therefore left to show \eqref{eq:ergodic2}. Because $i(\cdot)$ is increasing, we have the bound
\begin{align}
\text{\abbr{LHS}} \;\; \eqref{eq:ergodic2} \ \lesssim \ &\E\Big[\mathbf{1}_{t\leq\tau}\int_{0}^{t-\tau(\e)}\d s \ i(\langle\Phi^{\e}\rangle_{s})^{-1}\Big|\tau(\e)^{-1}\int_{0}^{\tau(\e)}\d r \left\{\psi_{x}(\wt{\mathbf{z}}^{\e}(s,r))-\E^{z,s,\mathrm{inv}}\psi_{x}\right\}\Big|\, \Big] \nonumber \\
\leq \ &\E\Big[\mathbf{1}_{t\lesssim1}\int_{0}^{t-\tau(\e)}\d s \ i(\langle\Phi^{\e}\rangle_{s})^{-1}\Big|\tau(\e)^{-1}\int_{0}^{\tau(\e)}\d r \left\{\psi_{x}(\wt{\mathbf{z}}^{\e}(s,r))-\E^{z,s,\mathrm{inv}}\psi_{x}\right\}\Big|\, \Big] \nonumber \\
\lesssim \ &\sup_{0\leq s\lesssim1}\E\Big[i(\langle\Phi^{\e}\rangle_{s})^{-1}\Big|\tau(\e)^{-1}\int_{0}^{\tau(\e)}\d r \left\{\psi_{x}(\wt{\mathbf{z}}^{\e}(s,r))-\E^{z,s,\mathrm{inv}}\psi_{x}\right\}\Big|\, \Big]. \label{eq:ergodic3}
\end{align}
(Indeed, the first bound follows by pulling the $i(\cdot)$-factor in the \abbr{LHS} of \eqref{eq:ergodic2} through the first integral in the \abbr{RHS} of \eqref{eq:ergodic1} and then replacing $\tau$ by $s$, which suffices for an upper bound since $i(\cdot)$ is increasing and since $s\leq t\leq\tau$ courtesy of the indicator function $\mathbf{1}_{t\leq\tau}$. The second bound follows because $\tau\lesssim1$, and the last follows by Fubini.) Let us now control \eqref{eq:ergodic3}. We first write
\begin{align}
&\tau(\e)^{-1}\int_{0}^{\tau(\e)}\d r \left\{\psi_{x}(\wt{\mathbf{z}}^{\e}(s,r))-\E^{z,s,\mathrm{inv}}\psi_{x}\right\} \nonumber\\
&= \ \tau(\e)^{-1}\left(\int_{0}^{\e\delta^{-1}}+\int_{\e\delta^{-1}}^{\tau(\e)}\right)\d r \left\{\psi_{x}(\wt{\mathbf{z}}^{\e}(s,r))-\E^{z,s,\mathrm{inv}}\psi_{x}\right\}. \nonumber
\end{align}
(We can do this since $\e\delta^{-1}\ll\tau(\e)$; see Lemma \ref{lemma:timeaverage} and Proposition \ref{prop:coupling}.) By $\e\delta^{-1}\ll\tau(\e)$ and uniform boundedness of $\psi_{x}(\cdot)$ (which comes from that of the kernel $\mathscr{K}$) we can easily bound the contribution on $[0,\e\delta^{-1}]$:
\begin{align}
\E\Big[i(\langle\Phi^{\e}\rangle_{s})^{-1}\Big|\tau(\e)^{-1}\int_{0}^{\e\delta^{-1}}\d r \left\{\psi_{x}(\wt{\mathbf{z}}^{\e}(s,r))-\E^{z,s,\mathrm{inv}}\psi_{x}\right\}\Big|\, \Big] \ \lesssim \ \tau(\e)^{-1}\e\delta^{-1} \ \to_{\e\to0} \ 0. \label{eq:ergodic4}
\end{align}
We move to the contribution on $[\e\delta^{-1},\tau(\e)]$. By conditioning, we clearly have the following preliminary bound:
\begin{align}
&\E\Big[i(\langle\Phi^{\e}\rangle_{s})^{-1}\Big|\tau(\e)^{-1}\int_{\e\delta^{-1}}^{\tau(\e)}\d r \left\{\psi_{x}(\wt{\mathbf{z}}^{\e}(s,r))-\E^{z,s,\mathrm{inv}}\psi_{x}\right\}\Big|\,\Big] \\
= \ &\E\Big[i(\langle\Phi^{\e}\rangle_{s})^{-1}\E_{s}\Big|\tau(\e)^{-1}\int_{\e\delta^{-1}}^{\tau(\e)}\d r \left\{\psi_{x}(\wt{\mathbf{z}}^{\e}(s,r))-\E^{z,s,\mathrm{inv}}\psi_{x}\right\}\Big|\,\Big],
\end{align}
where $\E_{s}$ means conditioning on all relevant filtrations at time $s$. Now, observe that in the second line, the integrand depends on $r\mapsto\wt{\mathbf{z}}^{\e}(s,r)$ only for $r\geq\e\delta^{-1}$, so the ``relevant initial data" therein is the law of $\wt{\mathbf{z}}^{\e}(s,r)$ at $r=\e\delta^{-1}$. But, by Theorem 1.1 in \cite{EO}, the law of $\wt{\mathbf{z}}^{\e}(s,r)$ at $r=\e\delta^{-1}$ has a Radon-Nikodym derivative with respect to the invariant measure $\mathbf{T}^{\Phi^{\e}(s,\cdot)}(\d y)$ which is bounded above (uniformly on $\partial\bUz$) by something that depends continuously on $\langle\Phi^{\e}(s,\cdot)\rangle_{\mathscr{C}^{10}}\leq\langle\Phi^{\e}\rangle_{s}$. (Put in words, the Radon-Nikodym derivative for the law of $\wt{\mathbf{z}}^{\e}(s,r)$ at $r=\e\delta^{-1}$ with respect to $\mathbf{T}^{\Phi^{\e}(s,\cdot)}(\d y)$ is, by standard Poisson concentration bounds, bounded from above by the Radon-Nikodym derivative for the law of the stopped process $t\mapsto\mathbf{x}^{\Phi^{\e}(s,\cdot),z}(\tau^{\Phi^{\e}(s,\cdot),z,t})$ for $t\gtrsim\e\delta^{-1}$. However, the Kolmogorov forward equation for the stopped process is a smoothing equation as shown in Theorem 1.1 in \cite{EO} since it comes from a diffusion, and the resulting heat kernel bounds depend only on the first, say seven, derivatives of the metric $\mathsf{g}^{\Phi^{\e}(s,\cdot)}|_{\partial\bUz}$ and its inverse, e.g. $\langle\Phi^{\e}(s,\cdot)\rangle_{\mathscr{C}^{10}}$.) Ultimately, if we include a change-of-measure factor that depends continuously on $\langle\Phi^{\e}\rangle_{s}$, in the previous display, we can replace the initial law at time $r=\e\delta^{-1}$ by the invariant measure:
\begin{align}
&\E_{s}\Big|\tau(\e)^{-1}\int_{\e\delta^{-1}}^{\tau(\e)}\d r \left\{\psi_{x}(\wt{\mathbf{z}}^{\e}(s,r))-\E^{z,s,\mathrm{inv}}\psi_{x}\right\}\Big| \nonumber\\
&\mapsto \E^{z,s,\mathrm{inv}}\Big|\tau(\e)^{-1}\int_{0}^{\tau(\e)-\e\delta^{-1}}\d r \left\{\psi_{x}(\wt{\mathbf{z}}^{\e}(s,r))-\E^{z,s,\mathrm{inv}}\psi_{x}\right\}\Big| \nonumber
\end{align}
where, to be clear, on the \abbr{RHS}, we are now sampling the initial law of $\wt{\mathbf{z}}^{\e}(s,r)$ at $r=0$ given by $\E^{z,s,\mathrm{inv}}$. It now suffices to show the following estimate, in which $i_{2}(\cdot)\geq1$ is a possibly different but still increasing and continuous function that we can choose:
\begin{align}
\sup_{0\leq s\lesssim1}\E\Big[i_{2}(\langle\Phi^{\e}\rangle_{s})^{-1}\E^{z,s,\mathrm{inv}}\Big|\tau(\e)^{-1}\int_{0}^{\tau(\e)-\e\delta^{-1}}\d r \left\{\psi_{x}(\wt{\mathbf{z}}^{\e}(s,r))-\E^{z,s,\mathrm{inv}}\psi_{x}\right\}\Big|\,\Big] \ \to_{\e\to0} \ 0. \label{eq:ergodic5}
\end{align}
To this end, recall from point (2) earlier that the process $\wt{\mathbf{z}}^{\e}(s,\cdot)$ has a uniformly positive spectral gap that is $\gtrsim\e^{-1}\delta(\e)$ (times something depending continuously on the underlying metric and thus on $\langle\Phi^{\e}\rangle_{s}$). Thus, we can bound the inner expectation in \eqref{eq:ergodic5} via the Kipnis-Varadhan inequality (see Appendix 1.6 of \cite{KL}). In words, this says that the time integrand in \eqref{eq:ergodic5} has zero mean with respect to the invariant measure of $\wt{\mathbf{z}}^{\e}(s,\cdot)$, so it is fluctuating. Hence, its time-average on any time-scale $\leq\tau(\e)$ can be controlled, in second moment, by $\tau(\e)^{-1}$ times a Brownian motion of speed $\e^{1/2}\delta(\e)^{-1/2}$ at time $\tau(\e)$; this is $\lesssim\e^{1/2}\delta(\e)^{-1/2}\tau(\e)^{-1/2}$. (The speed of this Brownian motion is controlled by the aforementioned spectral gap.) Precisely, Appendix 1.6 of \cite{KL} implies that the $\E^{z,s,\mathrm{inv}}$-term in \eqref{eq:ergodic5} is bounded above by $\tau(\e)^{-1/2}|\E^{z,s,\mathrm{inv}}[\wt{\psi}_{x}(\wt{\mathscr{L}}^{\e})^{-1}\wt{\psi}_{x}]|^{1/2}$, where $\wt{\psi}_{x}(\mathbf{z}):=\psi_{x}(\mathbf{z})-\E^{z,s,\mathrm{inv}}\psi_{x}$, and $\wt{\mathscr{L}}^{\e}$ is the generator for $\wt{\mathbf{z}}^{\e}(s,\cdot)$. (The factor of $\tau(\e)^{-1/2}$, as opposed to the factor of $\tau(\e)^{1/2}$ in Appendix 1.6 of \cite{KL}, is from dividing by $\tau(\e)$ in $\E^{z,s,\mathrm{inv}}$ in \eqref{eq:ergodic5}.) Because $\wt{\psi}_{x}$ is both uniformly bounded and orthogonal to the null-space of $\wt{\mathscr{L}}^{\e}$ by construction, and because $\wt{\mathscr{L}}^{\e}$ has a spectral gap of $\gtrsim\e^{-1}\delta(\e)$, to get an upper bound on $|\E^{z,s,\mathrm{inv}}[\wt{\psi}_{x}(\wt{\mathscr{L}}^{\e})^{-1}\wt{\psi}_{x}]|$, we can replace $\wt{\mathscr{L}}^{\e}$ with multiplication by $\e\delta(\e)^{-1}$. So, $\tau(\e)^{-1/2}|\E^{z,s,\mathrm{inv}}[\wt{\psi}_{x}(\wt{\mathscr{L}}^{\e})^{-1}\wt{\psi}_{x}]|^{1/2}\lesssim\e^{1/2}\delta(\e)^{-1/2}\tau(\e)^{-1/2}$. (The implied constants here depend on $\langle\Phi^{\e}\rangle_{s}$; they are dominated by $i_{2}(\langle\Phi^{\e}\rangle_{s})^{-1}$ on the LHS of \eqref{eq:ergodic5}. Let us also clarify that Appendix 1.6 of \cite{KL} requires $\wt{\mathbf{z}}^{\e}(s,\cdot)$ to be stationary, hence the step before \eqref{eq:ergodic5}.) As noted in Lemma \ref{lemma:timeaverage}, we have $\e^{1/2}\delta(\e)^{-1/2}\tau(\e)^{-1/2}\to0$. This gives \eqref{eq:ergodic5}. As noted before \eqref{eq:ergodic5}, this finishes the proof.
\end{proof}
\subsection{Proof of Proposition \ref{prop:hom}}
We combine Lemmas \ref{lemma:timeaverage}, \ref{lemma:freeze}, and \ref{lemma:ergodic} with Proposition \ref{prop:coupling} to get 
\begin{align}
\E\left[\|\mathsf{F}^{\e}-\mathsf{F}^{\e,\mathrm{eff}}\|_{\tau}\times i(\langle\Phi^{\e}\rangle_{\tau})^{-1}\right] \ \lesssim \ \tau(\e)+ \omega(\e)+ \ell(\e) \ \to_{\e\to0} \ 0.
\end{align}
Proposition \ref{prop:hom} then follows by the Markov inequality. \qed
%
%
%
\section{Proof of Proposition \ref{prop:coupling}}\label{section:couplingproof}
\subsection{Preliminary notation}
Throughout this section, we fix a time $s\geq0$. Next, by $\tau(\e)$, we mean a positive time that satisfies the constraints in Lemma \ref{lemma:timeaverage}. (It may change during our calculation, but it always depends only on $\delta(\e),\e$ to satisfy the constraints in Lemma \ref{lemma:timeaverage}.) Also, the function $i(\cdot)\geq1$ may change from line to line, but it is always increasing and continuous.

We start by introducing an \abbr{SDE} representation for the reflecting Brownian motion $\mathbf{x}^{\Phi,z}$ from Definition \ref{definition:interface} that defines the transition kernel for $\mathbf{z}^{\e}$. (First, recall notation from Definition \ref{definition:interface}.)
\begin{enumerate}
\item Given any $x\in\partial\bUz$, let $\mathbf{n}(x)$ be the unit inward normal vector at $x$. Next, for any $z\in\partial\bUz$, let $\mathbf{x}^{\Phi,z}$ be Brownian motion on $\bUz$ (with unit normal reflection on the boundary $\partial\bUz$) with respect to the metric $\mathsf{g}^{\Phi}$ and with initial data $\mathbf{x}^{\Phi,z}(0)=z\in\partial\bUz$. Thus, we have the \abbr{SDE}
\begin{align}
\d\mathbf{x}^{\Phi,z}(t) \ = \ {{}\d\mathbf{w}^{\Phi}}+\mathbf{n}(\mathbf{x}^{\Phi,z}(t))\d\mathbf{L}(t),
\end{align}
in which $\mathbf{L}(t)$ denotes the boundary local time of $\mathbf{x}^{\Phi,z}$, and $\mathbf{w}^{\Phi}$ is a Brownian motion on $\R^{\d}$ with respect to any fixed smooth extension of $\mathsf{g}^{\Phi}$ from $\bUz$ to $\R^{\d}$. Precisely, given a standard Euclidean Brownian motion $\mathbf{b}$ on $\R^{\d}$, we have the following \abbr{SDE} representation for $\mathbf{w}^{\Phi}$:
\begin{align}
{{}\d\mathbf{w}^{\Phi}} \ = \ \mathbf{A}^{\Phi}(\mathbf{x}^{\Phi,z}(t))\d\mathbf{b}(t)+\mathbf{m}^{\Phi}(\mathbf{x}^{\Phi,z}(t))\d t.
\end{align}
The matrix $\mathbf{A}^{\Phi}(\mathbf{x})$ is a symmetric $\d\times\d$-matrix satisfying $\mathbf{A}^{\Phi}(\mathbf{x})^{2}=(\mathsf{g}^{\Phi}(\mathbf{x}))^{-1}$ for all $\mathbf{x}\in\bUz$. (It is defined uniquely up to sign. Since Brownian motion $\mathbf{b}$ is invariant-in-law under changing sign, whatever choice we make for $\mathbf{A}^{\Phi}$ gives the same \abbr{SDE}.) The drift is given by
\begin{align}
\mathbf{m}^{\Phi}(\mathbf{x})_{j} \ = \ |\mathsf{g}^{\Phi}(\mathbf{x})|^{-\frac12}\sum_{i=1,\ldots,\d}\partial_{i}\left\{|\mathsf{g}^{\Phi}(\mathbf{x})|^{\frac12}(\mathsf{g}^{\Phi}(\mathbf{x}))^{-1}_{ij}\right\},
\end{align}
where $j\in\{1,\ldots,\d\}$, and $|\mathsf{g}^{\Phi}(\mathbf{x})|$ is the absolute value of the determinant of $\mathsf{g}^{\Phi}(\mathbf{x})$, and $\partial_{i}$ is differentiation with respect to the $i$-th Euclidean standard basis vector. (See Lemma \ref{lemma:sde} for well-posedness of the \abbr{SDE} for $\mathbf{x}^{\Phi,z}$ if $\Phi\in\mathscr{C}^{10}_{\simeq}(\partial\bUz,\R^{\d})$.) We clarify that $(\mathsf{g}^{\Phi}(\mathbf{x}))^{-1}_{ij}$ is the $(i,j)$-entry of the inverse matrix $(\mathsf{g}^{\Phi}(\mathbf{x}))^{-1}$. Now, take $\delta=\delta(\e)$ and recall the stopping time
\begin{align}
\tau^{\Phi,z,\delta} \ := \ \inf\{t\geq0:\mathbf{L}(t)=\delta\}.
\end{align}
\end{enumerate}
It turns out to be convenient to also introduce the following notation that we relate back to the processes $\mathbf{z}^{\e}$ and $\wt{\mathbf{z}}^{\e}$ in Proposition \ref{prop:coupling}. (In particular, this notation will help us construct the proposed process $\wt{\mathbf{z}}^{\e}$ therein.)
\begin{enumerate}
\item First, let us define $t\mapsto\wt{\mathbf{x}}^{\e}(s,t)$ to be a reflecting Brownian motion on $\bUz$ with respect to the metric $\mathsf{g}^{\Phi^{\e}(s,\cdot)}$ defined by the \emph{fixed} interface $\Phi^{\e}(s,\cdot)$. (By fixed, we also implicitly mean that we are conditioning on $\Phi^{\e}(s,\cdot)$ throughout this section.) In \abbr{SDE} terms, if $t\mapsto\mathbf{b}(t)$ is a standard Euclidean Brownian motion on $\R^{\d}$, we have (with notation explained after)
\begin{align}
\d\wt{\mathbf{x}}^{\e}(s,t) \ = \ \mathbf{A}^{\Phi^{\e}(s,\cdot)}(\wt{\mathbf{x}}^{\e}(s,t))\d\mathbf{b}(t) + \mathbf{m}^{\Phi^{\e}(s,\cdot)}(\wt{\mathbf{x}}^{\e}(s,t))\d t + \mathbf{n}(\wt{\mathbf{x}}^{\e}(s,t))\d\mathbf{L}^{s,\sim}(t).
\end{align}
The term $\mathbf{L}^{s,\sim}(t)$ denotes the boundary local time of $\wt{\mathbf{x}}^{\e}(s,r)$ accumulated between $r=0$ and $r=t$. We consider this \abbr{SDE} for times $t\geq0$, and we consider the initial data $\wt{\mathbf{x}}^{\e}(s,0)=\mathbf{z}^{\e}(s)$, where $\mathbf{z}^{\e}(s)$ is the location of the ``original particle" (from Definition \ref{definition:interface}) at time $s$ (which we condition on).
\item Let us now define a joint process $t\mapsto(\mathbf{x}^{\e}(s,t),\Psi_{s}^{\e}(t,\cdot))\in\Omega$ (until a possibly finite explosion time, which we omit for now and address shortly in Remark \ref{remark:coupling}). To define this joint process, with the same Brownian motion $\mathbf{b}$ as in point (1), we set the initial condition $\mathbf{x}^{\e}(s,0)=\wt{\mathbf{x}}^{\e}(s,0)=\mathbf{z}^{\e}(s)$ and otherwise let $\mathbf{x}^{\e}$ solve the \abbr{SDE}
\begin{align}
\d\mathbf{x}^{\e}(s,t) \ = \ \mathbf{A}^{\Psi^{\e}_{s}(t,\cdot)}(\mathbf{x}^{\e}(s,t))\d\mathbf{b}(t) + \mathbf{m}^{\Psi^{\e}_{s}(t,\cdot)}(\mathbf{x}^{\e}(s,t))\d t + \mathbf{n}(\mathbf{x}^{\e}(s,t))\d\mathbf{L}^{s}(t).
\end{align}
The term $\mathbf{L}^{s}(t)$ is the boundary local time but now for $r\mapsto\mathbf{x}^{\e}(s,r)$ between $r=0$ and $r=t$. (The superscript $s$ is kept to emphasize that it is the local time accumulated ``starting at time $s$", unlike the local time from Definition \ref{definition:interface}.) To specify the interface $t\mapsto\Psi^{\e}_{s}(t,\cdot)$, we give the following Poisson process representation:
\begin{align}
\Psi^{\e}_{s}(t,w) \ = \ \Phi^{\e}(s,w)+\sum_{0\leq\tau_{k}\leq t}\e\mathscr{K}(w,\mathbf{x}^{\e}(s,\tau_{k}))\mathsf{n}^{\Psi^{\e}_{s}(\tau_{k}^{-},\cdot)}(\mathbf{x}^{\e}(s,\tau_{k})),
\end{align}
where for any non-negative integer $k$, we define $\tau_{k}$ as the following stopping time for level sets of the local time:
\begin{align}
\tau_{k} \ := \ \inf\{t\geq0: \mathbf{L}^{s}(t)=k\delta(\e)\}.
\end{align}
In words, the interface $\Psi^{\e}_{s}(t,\cdot)$ grows according to $\e\mathscr{K}$-shaped bumps that are centered at points where the local time hits a new multiple of $\delta=\delta(\e)$. This is different than $\Phi^{\e}(s+t,\cdot)$ itself. Indeed, bumps in $\Psi^{\e}_{s}(t,\cdot)$ are added at the actual stopping times $\tau_{k}$ for the diffusion process $\mathbf{x}^{\e}(s,t)$ itself, whereas the bumps in $\Phi^{\e}(s+t,\cdot)$ are added at ringings of auxiliary Poisson clocks. (However, we do have the identity $\Phi^{\e}(s+t,\cdot)=\Psi^{\e}_{s}(\tau_{k(t)},\cdot)$, at least under an appropriate coupling, where $k(t)$ is a Poisson random variable of intensity $\e^{-1}t$. Because it would be inconvenient to always carry around an inverse function of $t\mapsto\tau_{k(t)}$, we introduced the auxiliary notation $\Psi^{\e}_{s}(t,\cdot)$.) Finally, it will be convenient to define versions of $\tau_{k}$ but for $\wt{\mathbf{x}}^{\e}$:
\begin{align}
\wt{\tau}_{k} \ := \ \inf\{t\geq0:\mathbf{L}^{s,\sim}(t)=k\delta(\e)\}.
\end{align}
\item We refer to Lemma \ref{lemma:sde} for global existence and uniqueness to the above \abbr{SDE}s and finiteness of the stopping times $\tau_{k}$ and $\wt{\tau}_{k}$ (for deterministic $k\geq0$), at least until explosion of $\langle\Psi^{\e}_{s}(t,\cdot)\rangle_{\mathscr{C}^{10}}$.
\end{enumerate}
Let us clarify briefly the previous constructions. As we alluded after the construction of $\tau_{k}$, the ``original" particle $t\mapsto\mathbf{z}^{\e}(s+t)$ is a continuous-time Poisson process with speed $\e^{-1}$ whose embedded Markov chain is $k\mapsto\mathbf{x}^{\e}(s,\tau_{k})$. \emph{By the same token, the process $t\mapsto\wt{\mathbf{z}}(s,t)$ in Proposition \ref{prop:coupling} may be realized as a continuous-time Poisson process of speed $\e^{-1}$ whose embedded chain is $k\mapsto\wt{\mathbf{x}}^{\e}(s,\wt{\tau}_{k})$.} In particular, by coupling the underlying Poisson clocks of $t\mapsto\mathbf{z}^{\e}(s+t),\wt{\mathbf{z}}(s,t)$, the proof of Proposition \ref{prop:coupling} will amount to comparing $t\mapsto\mathbf{x}^{\e}(s,t)$ and $t\mapsto\wt{\mathbf{x}}^{\e}(s,t)$ and $k\mapsto\tau_{k},\wt{\tau}_{k}$, both with high probability. We clarify that the speed of the Poisson clocks in the $\mathbf{z}^{\e},\wt{\mathbf{z}}^{\e}$ processes have speed $\e^{-1}$, so we must scale time by $\e^{-1}$ when analyzing $\mathbf{x}^{\e},\wt{\mathbf{x}}^{\e}$ processes below. (In particular, we will be interested in $\mathbf{x}^{\e},\wt{\mathbf{x}}^{\e}$ for time scales of order $\e^{-1}\tau(\e)\delta(\e)$ for the sake of Proposition \ref{prop:coupling}. Indeed, $\mathbf{z}^{\e},\wt{\mathbf{z}}^{\e}$ sample $\mathbf{x}^{\e},\wt{\mathbf{x}}^{\e}$ at time-$\delta(\e)$ increments with speed $\e^{-1}$, and we are interested in $\mathbf{z}^{\e},\wt{\mathbf{z}}^{\e}$ for times $\leq\tau(\e)$. Actually, for technical reasons, we will need to look at time scales just above $\e^{-1}\tau(\e)\delta(\e)$; see the $B(\e)$ constant in Lemma \ref{lemma:coupling1}, for example.)
\subsection{Comparing $\Phi^{\e}(s,\cdot)$ and $\Psi^{\e}_{s}(t,\cdot)$}
In order to make the comparisons that we have just explained, we will need to compare the underlying metrics defining $t\mapsto\mathbf{x}^{\e}(s,t)$ and $t\mapsto\wt{\mathbf{x}}^{\e}(s,t)$. Roughly speaking, we first know that $\Psi^{\e}_{s}(t,\cdot)-\Phi^{\e}(s,\cdot)$ is $\mathrm{O}(\e)$ times the number of $\tau_{k}$ we see before time $\tau(\e)\ll\e^{-1}\delta(\e)$ (this last asymptotic follows by constraints in Lemma \ref{lemma:timeaverage}). However, $\tau_{k}\leq\e^{-1}\delta(\e)$ if and only if $\mathbf{L}^{s}(\e^{-1}\delta(\e))\geq k\delta(\e)$, so we expect $k\ll\e^{-1}$ if $\tau_{k}\ll\e^{-1}\delta(\e)$. (Otherwise, the local time becomes super-linear.) Multiplying by $\mathrm{O}(\e)$ then bounds $\Psi^{\e}_{s}(t,\cdot)-\Phi^{\e}(s,\cdot)$ by $\mathrm{o}(1)$. To make this precise, we use the Ito formula for a particular $\mathsf{g}^{\Psi^{\e}_{s}(t,\cdot)}$-harmonic function on $\bUz$ and use of the corresponding elliptic theory, a common thread for this section. Before we state the result, recall $\langle,\rangle_{\mathscr{C}^{N}}$ from Definition \ref{definition:limit}.
\begin{lemma}\label{lemma:coupling1}
There exist $\alpha(\e),\gamma(\e)\to0$ and $B(\e)\to\infty$ as $\e\to0$, each depending only on $\e$, so that with probability at least $1-\gamma(\e)$, we have the following for increasing, continuous, and deterministic $i(\cdot)\geq1$:
\begin{align}
\sup_{0\leq t\leq\e^{-1}\tau(\e)\delta(\e)B(\e)}\langle\Phi^{\e}(s,\cdot),\Psi^{\e}_{s}(t,\cdot)\rangle_{\mathscr{C}^{10}} \ \lesssim \ \alpha(\e)\times i(\langle\Phi^{\e}(s,\cdot)\rangle_{\mathscr{C}^{10}}). \label{eq:coupling1I}
\end{align}
\end{lemma}
\begin{remark}\label{remark:coupling}
Recall the \abbr{SDE} for $t\mapsto\mathbf{x}^{\e}(s,t)$ only holds until explosion of $\Psi^{\e}_{s}(t,\cdot)$. Lemma \ref{lemma:coupling1} implies that explosion does not happen before time $\e^{-1}\delta(\e)\tau(\e)B(\e)$ with high probability (assuming the initial interface function $\Phi^{\e}(s,\cdot)$ is a diffeomorphism).
\end{remark}
\begin{proof}
We start by claiming the following preliminary estimate:
\begin{align}
\sup_{0\leq t\leq\e^{-1}\tau(\e)\delta(\e)B(\e)}
\langle\Phi^{\e}(s,\cdot),& \Psi^{\e}_{s}(t,\cdot)\rangle_{\mathscr{C}^{10}} \nonumber \\ 
& \lesssim 
  i(\langle\Phi^{\e}(s,\cdot)\rangle_{\mathscr{C}^{10}})\times\sup_{0\leq t\leq\e^{-1}\tau(\e)\delta(\e)B(\e)}\|\Phi^{\e}(s,\cdot)-\Psi^{\e}_{s}(t,\cdot)\|_{\mathscr{C}^{10}}. \label{eq:coupling1I1}
\end{align}
Indeed, \eqref{eq:coupling1I1} follows by the inverse function theorem and classical perturbation theory for matrices. For example, we have the identity 
\begin{align}
[\mathrm{Jac}\Phi^{\e}(s,\cdot)]^{-1}-[\mathrm{Jac}\Psi^{\e}_{s}(t,\cdot)]^{-1}=[\mathrm{Jac}\Phi^{\e}(s,\cdot)]^{-1}\{\mathrm{Jac}\Psi^{\e}_{s}(t,\cdot)-\mathrm{Jac}\Phi^{\e}(s,\cdot)\}[\mathrm{Jac}\Psi^{\e}_{s}(t,\cdot)]^{-1}.\nonumber
\end{align}
Now, differentiate and use the Leibniz rule to bound the first $9$ derivatives of $[\mathrm{Jac}\Phi^{\e}(s,\cdot)]^{-1}-[\mathrm{Jac}\Psi^{\e}_{s}(t,\cdot)]^{-1}$.

Let us bound the supremum on the \abbr{RHS} of \eqref{eq:coupling1I1}. By construction of $\Psi^{\e}_{s}(t,\cdot)$ from the previous subsection and regularity of $\mathscr{K}$, we obtain the following estimate for said supremum by summing all $\e\mathscr{K}$-bumps before time $\e^{-1}\tau(\e)\delta(\e)B(\e)$:
\begin{align}
&\sup_{0\leq t\leq\e^{-1}\tau(\e)\delta(\e)B(\e)}\|\Phi^{\e}(s,\cdot)-\Psi^{\e}_{s}(t,\cdot)\|_{\mathscr{C}^{10}} \label{eq:coupling1I2}\\
&\lesssim \ \sum_{0\leq\tau_{k}\leq\e^{-1}\tau(\e)\delta(\e)B(\e)}\e\|\mathscr{K}\|_{\mathscr{C}^{10}} \ \lesssim \ \e|\#\{\tau_{k}\leq\e^{-1}\tau(\e)\delta(\e)B(\e)\}|. \nonumber
\end{align}
We now claim that there exist $\alpha(\e),\gamma(\e)\to0$ as $\e\to0$ such that $\e|\#\{\tau_{k}\leq\e^{-1}\tau(\e)\delta(\e)B(\e)\}|\lesssim\alpha(\e)\times i(\langle\Phi^{\e}(s,\dot)\rangle_{\mathscr{C}^{10}})$ with probability at least $1-\gamma(\e)$ (assuming we take $B(\e)\to\infty$ sufficiently slow). Assuming this estimate, then with probability at least $1-\gamma(\e)$, the previous two displays would yield the proposed estimate \eqref{eq:coupling1I} immediately (for a different choice of $i(\cdot)$). Thus, it suffices to prove this claim. To this end, we claim that for any integer $m\geq0$, we have the logical implications:
\begin{align}
\#\{\tau_{k}\leq\e^{-1}\tau(\e)\delta(\e)B(\e)\}\geq m \ &\Rightarrow \ \tau_{m}\leq\{\e^{-1}\tau(\e)\delta(\e)B(\e)\}\wedge\tau_{m+1} \nonumber\\
&\Rightarrow \ \mathbf{L}^{s}(\{\e^{-1}\tau(\e)\delta(\e)B(\e)\}\wedge\tau_{m+1})\geq m\delta. \nonumber
\end{align}
(In words, if we see more than $m$-many jumps before time $\e^{-1}\tau(\e)\delta(\e)B(\e)$, then the local time at time $\e^{-1}\tau(\e)\delta(\e)B(\e)$ would exceed $m\delta$ by definition of the jumps. Taking minimum with $\tau_{m+1}$ is a technical convenience, which is allowed since $\tau_{m}\leq\tau_{m+1}$ with probability 1.) Now take $m=\e^{-1}\tau(\e)B(\e)^{2}\times i(\langle\Phi^{\e}(s,\cdot)\rangle_{\mathscr{C}^{10}})$ (with $B(\e)$ to be chosen later). By the above and Markov,
\begin{align}
&\mathbf{P}\left\{|\#\{\tau_{k}\leq\e^{-1}\tau(\e)\delta(\e)B(\e)\}|\geq m\right\} \nonumber\\
&\lesssim \ \mathbf{P}\left\{\mathbf{L}^{s}(\{\e^{-1}\tau(\e)\delta(\e)B(\e)\}\wedge\tau_{m+1})\geq m\delta\right\} \nonumber \\
&\lesssim \ m^{-1}\delta^{-1}\E\mathbf{L}^{s}(\{\e^{-1}\tau(\e)\delta(\e)B(\e)\}\wedge\tau_{m+1}) \nonumber\\
&\lesssim \ \e\delta(\e)^{-1}\tau(\e)^{-1}B(\e)^{-2}\times i(\langle\Phi^{\e}(s,\cdot)\rangle_{\mathscr{C}^{10}})^{-1}\E\mathbf{L}^{s}(\{\e^{-1}\tau(\e)\delta(\e)B(\e)\}\wedge\tau_{m+1}); \label{eq:coupling1I3}
\end{align}
we recall that the expectations $\E$ condition on the filtration at time $s$ (which is why $\langle\Phi^{\e}(s,\cdot)\rangle_{\mathscr{C}^{10}}$ can be taken outside $\E$). Note that $\tau(\e)B(\e)^{2}\to0$ if we choose $1\ll B(\e)\lesssim\tau(\e)^{-1/3}$, for example. Thus, if we can obtain that $\eqref{eq:coupling1I3}\leq\gamma(\e)$ for $\gamma(\e)\to0$, then by the above display, we would get $\e|\#\{\tau_{k}\leq\e^{-1}\tau(\e)\delta(\e)B(\e)\}|\lesssim\tau(\e)B(\e)^{2}\times i(\langle\Phi^{\e}(s,\dot)\rangle_{\mathscr{C}^{10}})$ with probability at least $1-\gamma(\e)$. The claim from the paragraph after \eqref{eq:coupling1I2} would therefore hold for $\alpha(\e)=\tau(\e)B(\e)^{2}$. Thus, it suffices to show that 
\begin{align}
\e\delta(\e)^{-1}\tau(\e)^{-1}B(\e)^{-2}\times i(\langle\Phi^{\e}(s,\cdot)\rangle_{\mathscr{C}^{10}})^{-1}\E\mathbf{L}^{s}(\{\e^{-1}\tau(\e)\delta(\e)B(\e)\}\wedge\tau_{m+1}) \ \leq \ \gamma(\e) \quad\mathrm{for}\quad \gamma(\e)\to0. \label{eq:coupling1I4}
\end{align}
Consider the solution $\mathscr{U}$ to the following Dirichlet problem, in which $\Delta_{\mathsf{g}}$ means Laplacian with respect to a metric $\mathsf{g}$ on $\bUz$:
\begin{align}
\Delta_{\mathsf{g}^{\Psi^{\e}_{s}(t,\cdot)}}\mathscr{U}(t,x) = 1 \quad\mathrm{and}\quad \mathscr{U}(t,x)|_{\partial\bUz}=0. \label{eq:coupling1I4b}
\end{align}
(The time-dependence of $\mathscr{U}$ comes only from the time-dependence of the underlying metric. In particular, the Laplacian is only in $x$.) The existence of a $\mathscr{C}^{3}_{x}$ solution to this \abbr{PDE} is guaranteed by classical elliptic regularity. (The $\mathscr{C}^{3}_{x}$-regularity of $\mathscr{U}$ holds for all times $t\geq0$ such that $\Psi^{\e}_{s}(t,\cdot)\in\mathscr{C}^{10}_{\simeq}(\partial\bUz,\R^{\d})$. In particular, we will only consider $\mathscr{U}$ before explosion of $\Psi^{\e}_{s}(t,\cdot)$, exactly like the \abbr{SDE} for $t\mapsto\mathbf{x}^{\e}(s,t)$. The following analysis still holds upon stopping everything before said explosion time, since this explosion time is a stopping time.) Note that $\mathscr{U}$ is adapted to the filtration of the joint process $t\mapsto(\mathbf{x}^{\e}(s,t),\Psi^{\e}_{s}(t,\cdot))$. Thus, we can apply the Ito formula to $\mathscr{U}$. Before doing so, for convenience only, let us set $\tau_{\mathrm{stop}}=\{\e^{-1}\tau(\e)\delta(\e)B(\e)\}\wedge\tau_{m+1}$. We claim
\begin{align}
\E\mathscr{U}(\tau_{\mathrm{stop}},\mathbf{x}^{\e}(s,\tau_{\mathrm{stop}}))&=\E\int_{0}^{\tau_{\mathrm{stop}}}\{\partial_{r}+\Delta_{\mathsf{g}^{\Psi_{s}^{\e}}(s,\cdot)}\}\mathscr{U}(r,\mathbf{x}^{\e}(s,r))\d r \nonumber\\
&+ \E\int_{0}^{\tau_{\mathrm{stop}}}\{\mathbf{n}\cdot\grad\mathscr{U}(r,\cdot)\}(\mathbf{x}^{\e}(s,r))\d\mathbf{L}^{s}(r). \nonumber
\end{align}
Technically, we have to include $\E\mathscr{U}(0),\mathbf{x}^{\e}(s,0))$ on the \abbr{RHS} of this previous display. However, $\mathbf{x}^{\e}(s,0)\in\partial\bUz$ by assumption in the previous subsection, and $\mathscr{U}$ vanishes on $\partial\bUz$, so this ``missing" contribution is zero. We also claim that we similarly have $\int \d r\partial_{r}\mathscr{U}(r,\mathbf{x}^{\e}(s,r))=0$ (where $\partial_{r}$ acts only on the first input of $\mathscr{U}$). Indeed, $\mathscr{U}$ is piecewise constant in $r$ since the metric $\mathsf{g}^{\Psi^{\e}_{s}(r,\cdot)}$ is, and therefore the $\d r$-integral of $\partial_{r}\mathscr{U}(r,\mathbf{x}^{\e}(s,r))$ is given by jumps of the type $\mathscr{U}(r,\mathbf{x}^{\e}(s,r))-\mathscr{U}(r^{-},\mathbf{x}^{\e}(s,r))$ that are supported at jump times of $\mathsf{g}^{\Psi_{s}^{\e}(r,\cdot)}$. But the jump times only happen when $\mathbf{x}^{\e}(s,r)\in\partial\bUz$, so the contribution of $\partial_{r}$ above is zero. Thus, by using the defining \abbr{PDE} for $\mathscr{U}$ and rearranging the previous display, we deduce the following estimate in which $\|\|_{\mathscr{L}^{\infty}_{t,x}}$ denotes the $\mathscr{L}^{\infty}$-norm for $0\leq t\leq\e^{-1}\tau(\e)\delta(\e)B(\e)\wedge\tau_{m+1}$ and $x\in\bUz$:
\begin{align}
-\E\int_{0}^{\tau_{\mathrm{stop}}}\{\mathbf{n}\cdot\grad\mathscr{U}(r,\cdot)\}(\mathbf{x}^{\e}(s,r))\d\mathbf{L}^{s}(r) \ \leq \ \E\|\mathscr{U}\|_{\mathscr{L}^{\infty}_{t,x}} + \E\tau_{\mathrm{stop}}. \label{eq:coupling1I5}
\end{align}
(The $\E$ outside the $\mathscr{U}$-norm is necessary because the $\mathscr{L}^{\infty}_{t,x}$-norm is only with respect to $(t,x)$-variables, and $\mathscr{U}$ is random outside these variables.) We now invoke \emph{deterministic} elliptic regularity estimates. First, we know $\|\mathscr{U}\|_{\mathscr{L}^{\infty}_{t,x}}\lesssim i_{1}(\|\langle\Psi^{\e}_{s}(t,\cdot)\rangle_{\mathscr{C}^{10}}\|_{\mathscr{L}^{\infty}_{t}})$ for some increasing and continuous $i_{1}(\cdot)\geq1$. (Indeed, the point here is that for any time $t$, the bounds on $\mathscr{U}(t,\cdot)$ are determined by $\Psi^{\e}_{s}(t,\cdot)$; now take a supremum over $t$.) Moreover, by the Hopf lemma, we know $-\mathbf{n}(z)\cdot\grad\mathscr{U}(r,z)\geq i_{2}(\langle\Psi^{\e}_{s}(r,\cdot)\rangle_{\mathscr{C}^{10}})^{-1}$ for all $z\in\partial\bUz$ for another increasing, continuous $i_{2}(\cdot)\geq1$. (In words, the normal derivative of $\mathscr{U}$ is strictly positive, and it depends continuously on at most 3 derivatives of the metric. For example, if $\mathscr{U}$ solves the Euclidean Dirichlet problem for $\bUz$ a compact interval, then $\mathscr{U}$ is a upwards-facing parabola, for which the Hopf lemma is clear.) Therefore, from \eqref{eq:coupling1I5}, we deduce
\begin{align}
\E\left\{\int_{0}^{\tau_{\mathrm{stop}}}i_{2}(\langle\Psi_{s}^{\e}(r,\cdot)\rangle_{\mathscr{C}^{10}})^{-1}\d\mathbf{L}^{s}(r)\right\} \ \lesssim \ \E\left\{i_{1}(\|\langle\Psi^{\e}_{s}(t,\cdot)\rangle_{\mathscr{C}^{10}}\|_{\mathscr{L}^{\infty}_{t}})\right\}+\E\tau_{\mathrm{stop}}. \label{eq:coupling1I6}
\end{align}
We would now like to replace $\Psi^{\e}_{s}$ in \eqref{eq:coupling1I6} by $\Phi^{\e}(s,\cdot)$, because the latter is deterministic (since we have conditioned on it). The mechanism for this amounts to the beginning of this argument, namely \eqref{eq:coupling1I1}-\eqref{eq:coupling1I2}. The benefit in the current situation, however, is that we are restrict to $r\leq\tau_{m+1}$ in \eqref{eq:coupling1I6}. In particular, we have a priori control on the number of $\e\mathscr{K}$-bumps we add. Let us make this precise. For any $0\leq r\leq\tau_{\mathrm{stop}}\leq\tau_{m+1}$, by the reasoning that gave \eqref{eq:coupling1I1}-\eqref{eq:coupling1I2}, we have 
\begin{align*}
&\langle\Psi_{s}^{\e}(r,\cdot)\rangle_{\mathscr{C}^{10}} \\
&\lesssim \ \langle\Phi^{\e}(s,\cdot)\rangle_{\mathscr{C}^{10}}+\langle\Psi_{s}^{\e}(r,\cdot),\Phi^{\e}(s,\cdot)\rangle_{\mathscr{C}^{10}} \ \lesssim \ \langle\Phi^{\e}(s,\cdot)\rangle_{\mathscr{C}^{10}}+i(\langle\Phi^{\e}(s,\cdot)\rangle_{\mathscr{C}^{10}})\|\Psi_{s}^{\e}(r,\cdot)-\Phi^{\e}(s,\cdot)\|_{\mathscr{C}^{10}} \\
&\lesssim \ \langle\Phi^{\e}(s,\cdot)\rangle_{\mathscr{C}^{10}}+i(\langle\Phi^{\e}(s,\cdot)\rangle_{\mathscr{C}^{10}})\times\e|\#\{\tau_{k}\leq\tau_{m+1}\}| \ \lesssim \ \langle\Phi^{\e}(s,\cdot)\rangle_{\mathscr{C}^{10}}+i(\langle\Phi^{\e}(s,\cdot)\rangle_{\mathscr{C}^{10}})\e[m+1] \\
&\lesssim \ \langle\Phi^{\e}(s,\cdot)\rangle_{\mathscr{C}^{10}}+i(\langle\Phi^{\e}(s,\cdot)\rangle_{\mathscr{C}^{10}})^{2}\times\tau(\e)B(\e)^{2} \ \lesssim \ \langle\Phi^{\e}(s,\cdot)\rangle_{\mathscr{C}^{10}}+i(\langle\Phi^{\e}(s,\cdot)\rangle_{\mathscr{C}^{10}})^{2}.
\end{align*}
(The last line follows by recalling $m=\e^{-1}\tau(\e)B(\e)^{2}$ from right before \eqref{eq:coupling1I3} and that $\tau(\e)B(\e)^{2}\ll1$ from right after \eqref{eq:coupling1I3}.) Therefore, upon adjusting $i_{1},i_{2}$ in \eqref{eq:coupling1I6} (but still continuous, increasing, deterministic, and independent of $\e$), we deduce
\begin{align}
\E\left\{\int_{0}^{\tau_{\mathrm{stop}}}i_{2}(\langle\Phi^{\e}(s,\cdot)\rangle_{\mathscr{C}^{10}})^{-1}\d\mathbf{L}^{s}(r)\right\} \ \lesssim \ \E\left\{i_{1}(\langle\Phi^{\e}(s,\cdot)\rangle_{\mathscr{C}^{10}})\right\}+\E\tau_{\mathrm{stop}}. \label{eq:coupling1I7}
\end{align}
Observe that on the \abbr{LHS} of \eqref{eq:coupling1I7}, the factor $i_{2}$ is independent of the $r$-integration variable. So, we can pull it out, and we are left with the local time itself. In particular, \eqref{eq:coupling1I7} gives
\begin{align}
\E\mathbf{L}^{s}(\tau_{\mathrm{stop}}) \ \lesssim \ i_{1}(\langle\Phi^{\e}(s,\cdot)\rangle_{\mathscr{C}^{10}})i_{2}(\langle\Phi^{\e}(s,\cdot)\rangle_{\mathscr{C}^{10}})+\E\tau_{\mathrm{stop}}. \label{eq:coupling1I8}
\end{align}
(We have also dropped the first expectation on the \abbr{RHS} because it depends only on the deterministic $\Phi^{\e}(s,\cdot)$.) Now, observe that $\tau_{\mathrm{stop}}\leq\e^{-1}\tau(\e)\delta(\e)B(\e)$ by construction (see the paragraph after \eqref{eq:coupling1I4b}). Thus, if we choose $i(\cdot)\geq1$ in the \abbr{LHS} of \eqref{eq:coupling1I4} to dominate $i_{1}\times i_{2}$ on the \abbr{RHS} of \eqref{eq:coupling1I8}, we deduce the following estimate from \eqref{eq:coupling1I8}
\begin{align}
\mathrm{LHS}\eqref{eq:coupling1I4} \ \lesssim \ \e\delta(\e)^{-1}\tau(\e)^{-1}B(\e)^{-2} + B(\e)^{-1}. \label{eq:copuling1I9}
\end{align}
By assumption, we know that $\e\delta(\e)^{-1}\tau(\e)^{-1}\to0$ (see Lemma \ref{lemma:timeaverage}). Taking $B(\e)\to\infty$ in any fashion therefore gives \eqref{eq:coupling1I4}. As noted immediately before \eqref{eq:coupling1I4}, this also completes the proof.
\end{proof}
\subsection{Estimating $t\mapsto|\mathbf{x}^{\e}(s,t)-\wt{\mathbf{x}}^{\e}(s,t)|$}
The proof of the following has two main ideas. For $\Phi\in\mathscr{C}^{10}_{\simeq}(\partial\bUz,\R^{\d})$, the metric $\mathsf{g}^{\Phi}$ is Euclidean after restricting to $\bUz\setminus\mathbf{C}_{1}$ by Definition \ref{definition:interface}. Therefore, $\mathbf{x}^{\e},\wt{\mathbf{x}}^{\e}$ are the same standard Brownian motion away from the collar $\mathbf{C}_{1}$, and the distance between them is constant as long as both live outside of said collar. Second, under the collar coordinates $\mathbf{C}_{1}\simeq\partial\bUz\times[0,1]$, their $[0,1]$-components evolve as the same one-dimensional reflecting Euclidean Brownian motion; see Definition \ref{definition:interface}. This makes estimating their difference easy. On the other hand, the ``horizontal" $\partial\bUz$-components evolve according to \abbr{SDE}s on a \emph{boundary-free domain} defined by metrics that are close with high probability due to Lemma \ref{lemma:coupling1}, which can be analyzed by the usual localization trick for \abbr{SDE}s on manifolds. Of course, there are obstacles to be wary of, such as making sure that both $\mathbf{x}^{\e},\wt{\mathbf{x}}^{\e}$ can be put in the same coordinate chart with high probability.
\begin{lemma}\label{lemma:coupling2}
There exist $\alpha(\e),\gamma(\e)\to0$ and $B(\e)\to\infty$ as $\e\to0$ and $\tau(\e)$ satisfying constraints in \emph{Lemma \ref{lemma:timeaverage}}, all of which depend only on $\e$, so that with probability at least $1-\gamma(\e)$, we have the following estimate:
\begin{align}
\sup_{0\leq t\leq\e^{-1}\delta(\e)\tau(\e)B(\e)}|\mathbf{x}^{\e}(s,t)-\wt{\mathbf{x}}^{\e}(s,t)| \ \lesssim \ \alpha(\e)\times i(\langle\Phi^{\e}(s,\cdot)\rangle_{\mathscr{C}^{10}}). \label{eq:coupling2I}
\end{align}
\end{lemma}
\begin{proof}
We start with some preliminary considerations. In this proof, we will take $i_{1}(\cdot)$ to be continuous, increasing, deterministic, and satisfy $i_{1}(\mathrm{a})\geq1+|\mathrm{a}|$. We will also take $\mu(\e)\to0$ sufficiently slow (to be determined by the end of this argument). First, we claim it suffices to assume the following estimate (in that we condition on time $s$ and restrict to the following time-$s$-measurable event), where $i(\cdot)$ is the function from Lemma \ref{lemma:coupling1}:
\begin{align}
i_{1}(\langle\Phi^{\e}(s,\cdot)\rangle_{\mathscr{C}^{10}})+i(\langle\Phi^{\e}(s,\cdot)\rangle_{\mathscr{C}^{10}}) \ \leq \ \mu(\e)^{-1}. \label{eq:coupling2I1}
\end{align}
Indeed, if not, then because $\mathbf{x}^{\e}(s,\cdot),\wt{\mathbf{x}}^{\e}(s,\cdot)$ belong to a fixed compact set $\bUz$, we have
\begin{align}
\sup_{0\leq t\leq\e^{-1}\delta(\e)\tau(\e)B(\e)}|\mathbf{x}^{\e}(s,t)-\wt{\mathbf{x}}^{\e}(s,t)| \ \lesssim \ 1 \ = \ \mu(\e)\mu(\e)^{-1} \ \lesssim \ \mu(\e) i_{2}(\langle\Phi^{\e}(s,\cdot)\rangle_{\mathscr{C}^{10}}),
\end{align}
from which \eqref{eq:coupling2I} clearly follows. (The point of this bound on $\langle\Phi^{\e}(s,\cdot)\rangle_{\mathscr{C}^{10}}$ is to provide a priori estimates for the underlying geometry. We will not use it until near the end of the proof.) Let us now define
\begin{align}
\tau^{\mathrm{sep}} \ &:= \ \inf\{t\geq0: |\mathbf{x}^{\e}(s,t)-\wt{\mathbf{x}}^{\e}(s,t)|\geq\tfrac{1}{999}\}\wedge\e^{-1}\tau(\e)\delta(\e)B(\e)\\
\tau^{\mathrm{geo}} \ &:= \ \inf\{t\geq0: \langle\Phi^{\e}(s,\cdot),\Psi^{\e}_{s}(t,\cdot)\rangle_{\mathscr{C}^{10}}\geq\alpha(\e)\times i(\langle\Phi^{\e}(s,\cdot)\rangle_{\mathscr{C}^{10}})\}\wedge\e^{-1}\tau(\e)\delta(\e)B(\e).
\end{align}
We clarify that $1/999$ is just a choice of small but $\e$-independent constant. Both $\alpha(\e)$ and $i(\cdot)$ in $\tau^{\mathrm{geo}}$ are from Lemma \ref{lemma:coupling1}. Lastly, $B(\e)\to\infty$ sufficiently slow (to be determined by the end of this argument). Set the stopping time $\tau^{\mathrm{pre}}:=\tau^{\mathrm{sep}}\wedge\tau^{\mathrm{geo}}$. We claim the stopped processes $t\mapsto\mathbf{x}^{\e}(s,t\wedge\tau^{\mathrm{pre}}),\wt{\mathbf{x}}^{\e}(s,t\wedge\tau^{\mathrm{pre}})$ admit the following localized description. (In what follows, by ``horizontal", we always mean the $\partial\bUz$-coordinate.)
\begin{enumerate}
\item First, we introduce the necessary geometry and local coordinate charts. Let $\mathbf{C}_{\ell}$ (for $0\leq\ell\leq3$) denote the collar of $\bUz$ of length $\ell$. (See right before Theorem \ref{theorem:thm2} for what this means. Also, by assumption in Theorem \ref{theorem:thm2}, since we restrict to $\ell\leq3$, such a collar exists.) Second, let $\{\mathbf{O}^{\partial,i}\}_{i}$ be a finite collection (whose size depends only on $\bUz$) of open subsets of $\partial\bUz$ (with respect to the subspace topology on $\partial\bUz$) such that given any $\mathbf{x}_{1},\mathbf{x}_{2}\in\mathbf{C}_{3}$ satisfying $|\mathbf{x}_{1}-\mathbf{x}_{2}|\leq1/999$, we can find some $\mathbf{O}^{\partial,i}$ so $\mathbf{x}_{1},\mathbf{x}_{2}\in\mathbf{O}^{i}:=\mathbf{O}^{\partial,i}\times[0,3]\subseteq\mathbf{C}_{3}$, and the distance between horizontal projections of $\mathbf{x}_{1},\mathbf{x}_{2}$ in $\partial\bUz$ from the boundary of $\mathbf{O}^{\partial,i}$ is at least $1/999$. (In words, if one of $\mathbf{x}^{\e},\wt{\mathbf{x}}^{\e}$ is in the collar, before time $\tau^{\mathrm{pre}}\leq\tau^{\mathrm{sep}}$, we can put them in the same coordinate chart, and their escaped speed from said chart is bounded uniformly in $\e$.) It will be convenient for forthcoming analysis to additionally assume, upon possibly ``shrinking" $\mathbf{O}^{\partial,i}$, that $\mathbf{O}^{\partial,i}$ is diffeomorphic to a subset of $\R^{\d-1}$ (so that it admits Euclidean coordinates).
\item We now give the aforementioned local description. We set $\sigma_{0}=0$, and define $\sigma_{k}$ inductively (for integer $k\geq1$). Fix $k\geq0$. Suppose, first, that $\mathbf{x}^{\e}(s,\sigma_{k}\wedge\tau^{\mathrm{pre}})\in\mathbf{C}_{3/2}$. By definition of $\tau^{\mathrm{pre}}$, we know that $\mathbf{x}^{\e}(s,\sigma_{k}\wedge\tau^{\mathrm{pre}}),\wt{\mathbf{x}}^{\e}(s,\sigma_{k}\wedge\tau^{\mathrm{pre}})\in\mathbf{C}_{2}$. Now, take $\mathbf{O}^{\partial,i}$ such that the distance between the horizontal projections of $\mathbf{x}^{\e}(s,\sigma_{k}\wedge\tau^{\mathrm{pre}}),\wt{\mathbf{x}}^{\e}(s,\sigma_{k}\wedge\tau^{\mathrm{pre}})$ onto $\partial\bUz$ and the boundary of $\mathbf{O}^{\partial,i}$ are each $\geq1/999$. In this case, we define $\sigma_{k+1}$ to be the first time that $\mathbf{x}^{\e}(s,\sigma_{k+1}\wedge\tau^{\mathrm{pre}})\not\in\mathbf{O}^{i}$ or $\wt{\mathbf{x}}^{\e}(s,\sigma_{k+1}\wedge\tau^{\mathrm{pre}})\not\in\mathbf{O}^{i}$. (Thus, $\sigma_{k+1}$ is the first time one of the particles leaves the coordinate chart, which by construction, forces one of the particles to travel distance $1/999$. We clarify that the choice of $\mathbf{O}^{\partial,i}$ is \emph{not} canonical; this is not an issue.)
\item Suppose instead that $\mathbf{x}^{\e}(s,\sigma_{k}\wedge\tau^{\mathrm{pre}})\not\in\mathbf{C}_{3/2}$. Since we work before time $\tau^{\mathrm{pre}}$, we know that $\mathbf{x}^{\e}(s,\sigma_{k}\wedge\tau^{\mathrm{pre}}),\wt{\mathbf{x}}^{\e}(s,\sigma_{k}\wedge\tau^{\mathrm{pre}})\not\in\mathbf{C}_{7/6}$. In this case, let $\sigma_{k+1}$ be the first time that $\mathbf{x}^{\e}(s,\sigma_{k}\wedge\tau^{\mathrm{pre}})\in\mathbf{C}_{9/8}$ or $\wt{\mathbf{x}}^{\e}(s,\sigma_{k}\wedge\tau^{\mathrm{pre}})\in\mathbf{C}_{9/8}$. We emphasize that this forces at least one of the particles to travel distance $1/999$.
\end{enumerate}
Now, let $N(\e)\to\infty$ be a term that diverges sufficiently slowly. (We specify it by the end of this argument. It depends only on $\e^{-1}\delta(\e)\tau(\e)B(\e)$, the time-scale on which we study $\mathbf{x}^{\e},\wt{\mathbf{x}}^{\e}$.) Also define the stopping time $\tau^{\mathrm{switch}}:=\sigma_{N(\e)}\wedge\e^{-1}\tau(\e)\delta(\e)B(\e)$. (This gives a priori control on the number of times we change coordinate charts, which we eventually bound via the fact that each coordinate chart requires a Brownian motion to travel distance $\geq1/999$. This is where dependence of $N(\e)$ on $\e^{-1}\delta(\e)\tau(\e)B(\e)$ is manifest.) Finally, define the stopping time $\tau^{\mathrm{st}}:=\tau^{\mathrm{switch}}\wedge\tau^{\mathrm{pre}}$. Now, some local calculations.
\begin{enumerate}
\item Fix $k\geq0$. Suppose $\mathbf{x}^{\e}(s,\sigma_{k}\wedge\tau^{\mathrm{st}})\not\in\mathbf{C}_{3/2}$. By construction, for $t\in[\sigma_{k}\wedge\tau^{\mathrm{st}},\sigma_{k+1}\wedge\tau^{\mathrm{st}}]$, we have $\mathbf{x}^{\e}(s,t),\wt{\mathbf{x}}^{\e}(s,t)\not\in\mathbf{C}_{1}$. By definition (see the beginning of this section and Definition \ref{definition:interface}), the defining metrics for $\mathbf{x}^{\e}(s,\cdot),\wt{\mathbf{x}}^{\e}(s,\cdot)$ are both the standard Euclidean metric. Thus, for $t\in[\sigma_{k}\wedge\tau^{\mathrm{st}},\sigma_{k+1}\wedge\tau^{\mathrm{st}}]$, we know $\mathbf{x}^{\e}(s,t),\wt{\mathbf{x}}^{\e}(s,t)$ are the same standard Brownian motion. It follows immediately that the distance between these particles is constant in $t\in[\sigma_{k}\wedge\tau^{\mathrm{st}},\sigma_{k+1}\wedge\tau^{\mathrm{st}}]$:
\begin{align}
|\mathbf{x}^{\e}(s,t)-\wt{\mathbf{x}}^{\e}(s,t)| \ = \ |\mathbf{x}^{\e}(s,\sigma_{k}\wedge\tau^{\mathrm{st}})-\wt{\mathbf{x}}^{\e}(s,\sigma_{k}\wedge\tau^{\mathrm{st}})|. \label{eq:coupling2I2}
\end{align}
\item Now, suppose that $\mathbf{x}^{\e}(s,\sigma_{k}\wedge\tau^{\mathrm{st}})\in\mathbf{C}_{3/2}$. For $t\in[\sigma_{k}\wedge\tau^{\mathrm{st}},\sigma_{k+1}\wedge\tau^{\mathrm{st}}]$, write $\mathbf{x}^{\e}(s,t)$ in its coordinates $(\mathbf{y}^{\mathrm{hor}}(t),\mathbf{y}^{\perp}(t))$ under the collar $\mathbf{O}^{i}\simeq\mathbf{O}^{\partial,i}\times[0,1]$. We claim that for $t\in[\sigma_{k}\wedge\tau^{\mathrm{st}},\sigma_{k+1}\wedge\tau^{\mathrm{st}}]$, the coordinate $\mathbf{y}^{\perp}(t)$ is a one-dimensional reflecting Brownian motion on $[0,\infty)$, and the other solves the following \abbr{SDE} (written in Euclidean coordinates for $\mathbf{O}^{\partial,i}$):
\begin{align}
\d\mathbf{y}^{\mathrm{hor}}(t) \ = \ \mathbf{A}^{\mathrm{hor},\Psi^{\e}_{s}(t,\cdot)}(\mathbf{y}^{\perp}(t),\mathbf{y}^{\mathrm{hor}}(t))\d\mathbf{b}^{\mathrm{hor}}(t)+\mathbf{m}^{\mathrm{hor},\Psi^{\e}_{s}(t,\cdot)}(\mathbf{y}^{\perp}(t),\mathbf{y}^{\mathrm{hor}}(t))\d t.
\end{align}
Let us clarify the notation in the previous display. The term $\mathbf{b}^{\mathrm{hor}}(t)$ is a standard $\R^{\d-1}$-valued Brownian motion. The term $\mathbf{A}^{\mathrm{hor},\Psi^{\e}_{s}(t,\cdot)}(\cdot,\cdot)$ is a $(\d-1)\times(\d-1)$-dimensional matrix. It is smooth in all of its entries (and in $\Psi^{\e}_{s}(t,\cdot)$ with respect to the $\langle\rangle_{\mathscr{C}^{10}}$-metric) with derivatives depending only on $\bUz$ (or more precisely, the transition maps for our Euclidean coordinates of $\mathbf{O}^{\partial,i}$ sets, and thus independent of $\e$). Lastly, the term $\mathbf{m}^{\mathrm{hor},\Psi^{\e}_{s}(t,\cdot)}$ is valued in $\R^{\d-1}$, and it enjoys the same smoothness properties as $\mathbf{A}^{\mathrm{hor},\Psi^{\e}_{s}(t,\cdot)}$. Let us now clarify where the above \abbr{SDE}s come from. Under the collar $\mathbf{C}_{1}\simeq\partial\bUz\times[0,1]$, we recall that the defining metric for $\mathbf{x}^{\e}(s,\cdot)$ is Euclidean in the $[0,1]$-factor; see Definition \ref{definition:interface}. Therefore, the coordinate $\mathbf{y}^{\perp}(t)$ evolves according to a Euclidean Brownian motion in one spatial dimension. Moreover, the point $0\in[0,1]$ corresponds to $\mathbf{x}^{\e}(s,\cdot)$ hitting $\partial\bUz$. Since the $[0,1]$-factor corresponds to the direction of reflection, the $\mathbf{y}^{\perp}(t)$ coordinate has a standard reflection rule for Brownian motions at $0\in[0,1]$. Thus, $\mathbf{y}^{\perp}(t)$ is, indeed, a standard one-dimensional reflecting Brownian motion on $[0,\infty)$ (that we stop if it hits $3$ by construction of $\sigma_{k+1}$). To get the $\mathbf{y}^{\mathrm{hor}}(t)$ \abbr{SDE}, recall from Definition \ref{definition:interface} that the metric for the $\partial\bUz$-factor is a smooth function of $\mathsf{g}^{\Psi^{\e}_{s}(t,\cdot)}$ and the $[0,1]$-coordinate (since this coordinate is the interpolation parameter between $\mathsf{g}^{\Psi^{\e}_{s}(t,\cdot)}$ and Euclidean metric on $\partial\bUz$). So, the $\mathbf{y}^{\mathrm{hor}}(t)$ \abbr{SDE} follows by the fact that the coefficients for Brownian motion on a manifold are smooth with respect to the metric, its inverse, and its first two derivatives (see the beginning of Section \ref{section:couplingproof} for explicit formulas). We clarify the $\mathbf{y}^{\mathrm{hor}}(t)$ \abbr{SDE} has no reflection term because the reflection of $\mathbf{x}^{\e}(s,\cdot)$ is orthogonal to $\partial\bUz$. By the same token, we know that for $t\in[\sigma_{k}\wedge\tau^{\mathrm{st}},\sigma_{k+1}\wedge\tau^{\mathrm{st}}]$, the collar coordinates of $\wt{\mathbf{x}}^{\e}(s,t)$ are given by $(\wt{\mathbf{y}}^{\mathrm{hor}}(t),\wt{\mathbf{y}}^{\perp}(t))$, which satisfy the following evolution. Because the metric for $\wt{\mathbf{y}}^{\perp}(t)$ is still Euclidean, we know $\wt{\mathbf{y}}^{\perp}(t)$ is also a standard reflecting one-dimensional Brownian motion on $[0,\infty)$. Similarly, we have 
\begin{align}
\d\wt{\mathbf{y}}^{\mathrm{hor}}(t) \ = \ \mathbf{A}^{\mathrm{hor},\Phi^{\e}(s,\cdot)}(\wt{\mathbf{y}}^{\perp}(t),\wt{\mathbf{y}}^{\mathrm{hor}}(t))\d\mathbf{b}^{\mathrm{hor}}(t)+\mathbf{m}^{\mathrm{hor},\Phi^{\e}(s,\cdot)}(\wt{\mathbf{y}}^{\perp}(t),\wt{\mathbf{y}}^{\mathrm{hor}}(t))\d t,
\end{align}
where $\mathbf{A},\mathrm{m}$ are the same as in the $\mathbf{y}^{\mathrm{hor}}$ \abbr{SDE}, just evaluated at different interface function and points. (The Brownian motions in these horizontal \abbr{SDE}s are the same because the Brownian motions driving $\mathbf{x}^{\e}(s,\cdot),\wt{\mathbf{x}}^{\e}(s,\cdot)$ are the same.)

We now bound $(\mathbf{y}^{\mathrm{hor}}(t),\mathbf{y}^{\perp}(t))-(\wt{\mathbf{y}}^{\mathrm{hor}}(t),\wt{\mathbf{y}}^{\perp}(t))$. Write $\mathbf{y}^{\perp}(t)=|\mathbf{b}^{\perp}(t-\sigma_{k}\wedge\tau^{\mathrm{st}})+\mathbf{y}^{\perp}(\sigma_{k}\wedge\tau^{\mathrm{st}})|$ and $\wt{\mathbf{y}}^{\perp}(t)=|\mathbf{b}^{\perp}(t-\sigma_{k}\wedge\tau^{\mathrm{st}})+\wt{\mathbf{y}}^{\perp}(\sigma_{k}\wedge\tau^{\mathrm{st}})|$, where $\mathbf{b}^{\perp}$ is a standard one-dimensional Brownian motion on $\R$. (We are implicitly using that $\mathbf{y}^{\perp},\wt{\mathbf{y}}^{\perp}$ are driven by the same Brownian motions; indeed, the Brownian motions for $\mathbf{x}^{\e}(s,\cdot),\wt{\mathbf{x}}^{\e}(s,\cdot)$ are the same by construction at the beginning of this section. We also using crucially that a one-dimensional reflecting Brownian motion on $[0,\infty)$ is the absolute value of standard Brownian motion on $\R$.) By this and the triangle inequality, we get
\begin{align}
|\mathbf{y}^{\perp}(t)-\wt{\mathbf{y}}^{\perp}(t)| \ \leq \ |\mathbf{y}^{\perp}(\sigma_{k}\wedge\tau^{\mathrm{st}})-\wt{\mathbf{y}}^{\perp}(\sigma_{k}\wedge\tau^{\mathrm{st}})|. \label{eq:coupling2I3}
\end{align}
(This is true for all $t\in[\sigma_{k}\wedge\tau^{\mathrm{st}},\sigma_{k+1}\wedge\tau^{\mathrm{st}}]$, since it is for these times that we have the $\mathbf{b}^{\perp}$ representation.) On the other hand, we have the following \abbr{SDE} for $\mathbf{y}^{\mathrm{hor}}(t)-\wt{\mathbf{y}}^{\mathrm{hor}}(t)$ for all $t\in[\sigma_{k}\wedge\tau^{\mathrm{st}},\sigma_{k+1}\wedge\tau^{\mathrm{st}}]$:
\begin{align}
\d\{\mathbf{y}^{\mathrm{hor}}(t)-\wt{\mathbf{y}}^{\mathrm{hor}}(t)\} \ = \ &\left\{\mathbf{A}^{\mathrm{hor},\Phi^{\e}(s,\cdot)}(\mathbf{y}^{\perp}(t),\mathbf{y}^{\mathrm{hor}}(t))-\mathbf{A}^{\mathrm{hor},\Phi^{\e}(s,\cdot)}(\wt{\mathbf{y}}^{\perp}(t),\wt{\mathbf{y}}^{\mathrm{hor}}(t))\right\}\d\mathbf{b}^{\mathrm{hor}}(t) \label{eq:coupling2I4a}\\
+ \ &\left\{\mathbf{A}^{\mathrm{hor},\Psi^{\e}_{s}(t,\cdot)}(\mathbf{y}^{\perp}(t),\mathbf{y}^{\mathrm{hor}}(t))-\mathbf{A}^{\mathrm{hor},\Phi^{\e}(s,\cdot)}(\mathbf{y}^{\perp}(t),\mathbf{y}^{\mathrm{hor}}(t))\right\}\d\mathbf{b}^{\mathrm{hor}}(t)\label{eq:coupling2I4b}\\
+ \  &\left\{\mathbf{m}^{\mathrm{hor},\Phi^{\e}(s,\cdot)}(\mathbf{y}^{\perp}(t),\mathbf{y}^{\mathrm{hor}}(t))-\mathbf{m}^{\mathrm{hor},\Phi^{\e}(s,\cdot)}(\wt{\mathbf{y}}^{\perp}(t),\wt{\mathbf{y}}^{\mathrm{hor}}(t))\right\}\d t \label{eq:coupling2I4c}\\
+ \ &\left\{\mathbf{m}^{\mathrm{hor},\Psi^{\e}_{s}(t,\cdot)}(\mathbf{y}^{\perp}(t),\mathbf{y}^{\mathrm{hor}}(t))-\mathbf{m}^{\mathrm{hor},\Phi^{\e}(s,\cdot)}(\mathbf{y}^{\perp}(t),\mathbf{y}^{\mathrm{hor}}(t))\right\}\d t.\label{eq:coupling2I4d}
\end{align}
As $\sigma_{k},\sigma_{k+1},\tau^{\mathrm{st}}$ are all stopping times, we can control the martingale terms in \eqref{eq:coupling2I4a}-\eqref{eq:coupling2I4b} via Ito calculus. By smoothness of $\mathbf{A},\mathbf{m}$, we know the coefficients in \eqref{eq:coupling2I4a} and \eqref{eq:coupling2I4c} are deterministically bounded (in absolute value) by $|\mathbf{y}^{\perp}(t)-\wt{\mathbf{y}}^{\perp}(t)|+|\mathbf{y}^{\mathrm{hor}}(t)-\wt{\mathbf{y}}^{\mathrm{hor}}(t)|$ (up to a factor depending only on $\langle\Phi^{\e}(s,\cdot)\rangle$). Similarly, the coefficients in \eqref{eq:coupling2I4b} and \eqref{eq:coupling2I4d} are bounded (up to $\mathrm{O}(1)$ factor) by $\langle\Phi^{\e}(s,\cdot),\Psi^{\e}_{s}(t,\cdot)\rangle_{\mathscr{C}^{10}}$, which itself is $\lesssim\alpha(\e)\times i(\langle\Phi^{\e}(s,\cdot)\rangle_{\mathscr{C}^{10}})$ since we work before $\tau^{\mathrm{st}}$. (See the beginning of this proof for the definition of $\tau^{\mathrm{geo}}$.) To state the resulting estimate via Ito calculus, let us first introduce some notation. First, let $\E_{k}$ denote expectation conditioning on the filtration at time $\sigma_{k}\wedge\tau^{\mathrm{st}}$ (which lets us still use Ito calculus by the strong Markov property for Ito \abbr{SDE}s). Also, let us define the following maximal processes for $t\geq\sigma_{k}\wedge\tau^{\mathrm{st}}$:
\begin{align}
\mathfrak{m}_{\mathrm{hor},k}(t) \ &:= \ \sup_{\sigma_{k}\wedge\tau^{\mathrm{st}}\leq r\leq t\wedge\sigma_{k+1}\wedge\tau^{\mathrm{st}}}|\mathbf{y}^{\mathrm{hor}}(r)-\wt{\mathbf{y}}^{\mathrm{hor}}(r)| \\
\mathfrak{m}_{k}(t) \ &:= \ \sup_{\sigma_{k}\wedge\tau^{\mathrm{st}}\leq r\leq t\wedge\sigma_{k+1}\wedge\tau^{\mathrm{st}}}\left\{|\mathbf{y}^{\mathrm{hor}}(r)-\wt{\mathbf{y}}^{\mathrm{hor}}(r)|+|\mathbf{y}^{\perp}(r)-\wt{\mathbf{y}}^{\perp}(r)|\right\} \\
\mathfrak{m}_{\mathbf{x},k}(t) \ &:= \ \sup_{\sigma_{k}\wedge\tau^{\mathrm{st}}\leq r\leq t\wedge\sigma_{k+1}\wedge\tau^{\mathrm{st}}}|\mathbf{x}^{\e}(s,r)-\wt{\mathbf{x}}^{\e}(s,r)|.
\end{align}
Now, by the Ito isometry, \eqref{eq:coupling2I4a}-\eqref{eq:coupling2I4d}, and the paragraph immediately after, we have the following for all $t\geq\sigma_{k}\wedge\tau^{\mathrm{st}}$:
\begin{align}
\E_{k}|\mathfrak{m}_{\mathrm{hor},k}(t)|^{2} \ \lesssim \ &\E_{k}\Big\{\left[1+|t-\sigma_{k}\wedge\tau^{\mathrm{st}}|\right]\int_{\sigma_{k}\wedge\tau^{\mathrm{st}}}^{t}|\mathfrak{m}_{\mathrm{hor},k}(s)|^{2}\d s\Big\}\label{eq:coupling2I5a}\\
+ \ &\E_{k}\Big\{\left[1+|t-\sigma_{k}\wedge\tau^{\mathrm{st}}|\right]\int_{\sigma_{k}\wedge\tau^{\mathrm{st}}}^{t}|\mathbf{y}^{\perp}(s)-\wt{\mathbf{y}}^{\perp}(s)|^{2}\d s\Big\}\label{eq:coupling2I5b}\\
+ \ &\E_{k}\Big\{\left[1+|t-\sigma_{k}\wedge\tau^{\mathrm{st}}|\right]^{2}\alpha(\e)^{2}i(\langle\Phi^{\e}(s,\cdot)\rangle_{\mathscr{C}^{10}})^{2}\Big\}. \label{eq:coupling2I5c}
\end{align}
By combining \eqref{eq:coupling2I3} and \eqref{eq:coupling2I5a}-\eqref{eq:coupling2I5c}, we get, again for all $t\geq\sigma_{k}\wedge\tau^{\mathrm{st}}$:
\begin{align}
\E_{k}|\mathfrak{m}_{k}(t)|^{2} \ \lesssim \ &\E_{k}\Big\{\left[1+|t-\sigma_{k}\wedge\tau^{\mathrm{st}}|\right]\int_{\sigma_{k}\wedge\tau^{\mathrm{st}}}^{t}|\mathfrak{m}_{k}(s)|^{2}\d s\Big\}\label{eq:coupling2I6a}\\
+ \ &\E_{k}\Big\{\left[1+|t-\sigma_{k}\wedge\tau^{\mathrm{st}}|\right]^{2}\alpha(\e)^{2}i(\langle\Phi^{\e}(s,\cdot)\rangle_{\mathscr{C}^{10}})^{2}\Big\}. \label{eq:coupling2I6b}
\end{align}
Note $\sigma_{k}\wedge\tau^{\mathrm{st}}\leq\tau^{\mathrm{st}}\leq\e^{-1}\tau(\e)\delta(\e)B(\e)$ 
\abbr{wp1}.
So, for all $\sigma_{k}\wedge\tau^{\mathrm{st}}\leq t\leq\e^{-1}\delta(\e)\tau(\e)B(\e)$, we get
\begin{align}
 \E_{k}|& \mathfrak{m}_{k}(t)|^{2} \  \lesssim \ \nonumber \\
& \e^{-1}\tau(\e)\delta(\e)B(\e)\int_{\sigma_{k}\wedge\tau^{\mathrm{st}}}^{t}\E_{k}|\mathfrak{m}_{k}(s)|^{2}\d s + \e^{-2}\tau(\e)^{2}\delta(\e)^{2}B(\e)^{2}\alpha(\e)^{2}i(\langle\Phi^{\e}(s,\cdot)\rangle_{\mathscr{C}^{10}})^{2}. \label{eq:coupling2I7}
\end{align}
By taking $t=\e^{-1}\delta(\e)\tau(\e)B(\e)$ and using the Gronwall inequality, we deduce the estimate
\begin{align}
&\E_{k}|\mathfrak{m}_{k}(\sigma_{k+1}\wedge\tau^{\mathrm{st}})|^{2} \ = \ \E_{k}|\mathfrak{m}_{k}(\e^{-1}\delta(\e)\tau(\e)B(\e))|^{2} \nonumber \\
&\lesssim \ \exp\left\{\e^{-2}\tau(\e)^{2}\delta(\e)^{2}B(\e)^{2}\right\}\left\{|\mathfrak{m}_{k}(\sigma_{k}\wedge\tau^{\mathrm{st}})|^{2}+\e^{-2}\tau(\e)^{2}\delta(\e)^{2}B(\e)^{2}\alpha(\e)^{2}i(\langle\Phi^{\e}(s,\cdot)\rangle_{\mathscr{C}^{10}})^{2}\right\}, \label{eq:coupling2I8}
\end{align}
where the first identity above follows because $\mathfrak{m}_{k}(t)$ stops at time $\sigma_{k+1}\wedge\tau^{\mathrm{st}}$, so $\mathfrak{m}_{k}(t)=\mathfrak{m}_{k}(t\wedge\sigma_{k+1}\wedge\tau^{\mathrm{st}})$ with probability 1 for any $t\geq\tau^{\mathrm{st}}$, and in \eqref{eq:coupling2I8}, there is no $\E_{k}$ because \eqref{eq:coupling2I8} is measurable with respect to the time-$\sigma_{k}\wedge\tau^{\mathrm{st}}$ filtration. For the last step in this bullet point, let us make two observations. First note $\mathfrak{m}_{\mathbf{x},k}(t)\lesssim\mathfrak{m}_{k}(t)\lesssim\mathfrak{m}_{\mathbf{x},k}(t)$ for all $t$ with deterministic implied constants. (Indeed, the $\ell^{2}$ distance between vectors $(\mathbf{y}^{\mathrm{hor}}(t),\mathbf{y}^{\perp}(t))$ and $(\wt{\mathbf{y}}^{\mathrm{hor}}(t),\wt{\mathbf{y}}^{\perp}(t))$ and the distance between $\mathbf{x}^{\e}(s,t),\wt{\mathbf{x}}^{\e}(s,t)$ are not the same, but they differ up to multiplicative factors that are bounded above and below away from 0 and that depend only on our coordinate charts for $\mathbf{O}^{\partial,i}$.) Second, by the law of iterated expectation, we can put $\E$ in front of $|\mathfrak{m}_{k}(\sigma_{k+1}\wedge\tau^{\mathrm{st}})|^{2}$ and $|\mathfrak{m}_{k}(\sigma_{k}\wedge\tau^{\mathrm{st}})|^{2}$ in the previous display (where $\E$ still conditions on $\Phi^{\e}(s,\cdot)$). Third, we can replace $\mathfrak{m}_{\mathbf{x},k}(\sigma_{k}\wedge\tau^{\mathrm{st}})$ by $\mathfrak{m}_{\mathbf{x},k-1}(\sigma_{k}\wedge\tau^{\mathrm{st}})$ for the sake of an upper bound, assuming $k\geq1$. (Indeed, the latter $\mathfrak{m}_{\mathbf{x},k-1}(\sigma_{k}\wedge\tau^{\mathrm{st}})$ is a supremum over a set of values that includes $\mathfrak{m}_{\mathbf{x},k}(\sigma_{k}\wedge\tau^{\mathrm{st}})$.) Therefore, we ultimately deduce the following for any $k\geq1$:
\begin{align}
\E|\mathfrak{m}_{\mathbf{x},k}(\sigma_{k+1}\wedge & \tau^{\mathrm{st}})|^{2} \ \lesssim \ 
\exp\left\{\e^{-2}\tau(\e)^{2}\delta(\e)^{2}B(\e)^{2}\right\}\E|\mathfrak{m}_{\mathbf{x},k-1}(\sigma_{k}\wedge\tau^{\mathrm{st}})|^{2} \label{eq:coupling2I9a}\\
+ \ &\exp\left\{\e^{-2}\tau(\e)^{2}\delta(\e)^{2}B(\e)^{2}\right\}\e^{-2}\tau(\e)^{2}\delta(\e)^{2}B(\e)^{2}\alpha(\e)^{2}i(\langle\Phi^{\e}(s,\cdot)\rangle_{\mathscr{C}^{10}})^{2}. \label{eq:coupling2I9b}
\end{align}
\item By combining the previous two bullet points, we deduce that for any $k\geq1$, the estimate \eqref{eq:coupling2I9a}-\eqref{eq:coupling2I9b} is true regardless of where $\mathbf{x}^{\e}(s,\sigma_{k}\wedge\tau^{\mathrm{st}})$ is. (Indeed, it holds in the setting of point (2), and a sharper version \eqref{eq:coupling2I2} holds in the setting of point (1).) Now, by iterating the resulting bound, we deduce that for any $k\geq1$, we have the following for some constant $\eta>0$:
\begin{align}
&\E|\mathfrak{m}_{\mathbf{x},k}(\sigma_{k+1}\wedge\tau^{\mathrm{st}})|^{2} \nonumber\\
&\lesssim \ \eta^{k}\exp\left\{k\e^{-2}\tau(\e)^{2}\delta(\e)^{2}B(\e)^{2}\right\}\E|\mathfrak{m}_{\mathbf{x},0}(\sigma_{1}\wedge\tau^{\mathrm{st}})|^{2}\nonumber\\
&+ \ k\eta^{k}\exp\left\{k\e^{-2}\tau(\e)^{2}\delta(\e)^{2}B(\e)^{2}\right\}\e^{-2}\tau(\e)^{2}\delta(\e)^{2}B(\e)^{2}\alpha(\e)^{2}i(\langle\Phi^{\e}(s,\cdot)\rangle_{\mathscr{C}^{10}})^{2}. \nonumber
\end{align}
We then use \eqref{eq:coupling2I2} or \eqref{eq:coupling2I8} (depending on the location of $\mathbf{x}^{\e}(s,\sigma_{0}\wedge\tau^{\mathrm{st}})$) to bound $\E|\mathfrak{m}_{\mathbf{x},0}(\sigma_{1}\wedge\tau^{\mathrm{st}})|^{2}$. In the case that we use \eqref{eq:coupling2I2}, we get $\mathfrak{m}_{\mathbf{x},0}(\sigma_{1}\wedge\tau^{\mathrm{st}})=\mathfrak{m}_{\mathbf{x},0}(\sigma_{0}\wedge\tau^{\mathrm{st}})=\mathfrak{m}_{\mathbf{x},0}(0)=0$, since $\sigma_{0}=0$ and $\mathbf{x}^{\e}(s,0)=\wt{\mathbf{x}}^{\e}(s,0)$ in the setting where \eqref{eq:coupling2I2} would apply. In the case that we use \eqref{eq:coupling2I8}, we would get the following estimate:
\begin{align}
&\E|\mathfrak{m}_{\mathbf{x},0}(\sigma_{1}\wedge\tau^{\mathrm{st}})|^{2} \nonumber\\
&\lesssim \ \exp\left\{\e^{-2}\tau(\e)^{2}\delta(\e)^{2}B(\e)^{2}\right\}\left\{|\mathfrak{m}_{\mathbf{x},0}(\sigma_{0}\wedge\tau^{\mathrm{st}})|^{2}+\e^{-2}\tau(\e)^{2}\delta(\e)^{2}B(\e)^{2}\alpha(\e)^{2}i(\langle\Phi^{\e}(s,\cdot)\rangle_{\mathscr{C}^{10}})^{2}\right\} \nonumber\\
&= \ \exp\left\{\e^{-2}\tau(\e)^{2}\delta(\e)^{2}B(\e)^{2}\right\}\e^{-2}\tau(\e)^{2}\delta(\e)^{2}B(\e)^{2}\alpha(\e)^{2}i(\langle\Phi^{\e}(s,\cdot)\rangle_{\mathscr{C}^{10}})^{2}.
\end{align}
Note that the previous estimate is true in the case that we apply \eqref{eq:coupling2I2} instead of \eqref{eq:coupling2I8} as well (even though it is much less sharp). Therefore, in either case, from the previous estimate, we deduce the following:
\begin{align}
&\E|\mathfrak{m}_{\mathbf{x},k}(\sigma_{k+1}\wedge\tau^{\mathrm{st}})|^{2} \nonumber\\
&\lesssim \ k\eta^{k+1}\exp\left\{\{k+1\}\e^{-2}\tau(\e)^{2}\delta(\e)^{2}B(\e)^{2}\right\}\e^{-2}\tau(\e)^{2}\delta(\e)^{2}B(\e)^{2}\alpha(\e)^{2}i(\langle\Phi^{\e}(s,\cdot)\rangle_{\mathscr{C}^{10}})^{2}. \label{eq:coupling2I10}
\end{align}
We now claim that \eqref{eq:coupling2I10} implies the following for a possibly different but fixed constant $\eta>0$:
\begin{align}
&\E\Big\{\sup_{0\leq t\leq\e^{-1}\delta(\e)\tau(\e)B(\e)}|\mathbf{x}^{\e}(s,t\wedge\tau^{\mathrm{st}})-\wt{\mathbf{x}}^{\e}(s,t\wedge\tau^{\mathrm{st}})|^{2}\Big\} \label{eq:coupling2I11a}\\
&\lesssim \ N(\e)^{2}\eta^{N(\e)}\exp\left\{\e^{-2}\tau(\e)^{2}\delta(\e)^{2}B(\e)^{2}[N(\e)+1]\right\}\e^{-2}\tau(\e)^{2}\delta(\e)^{2}B(\e)^{2}\alpha(\e)^{2}i(\langle\Phi^{\e}(s,\cdot)\rangle_{\mathscr{C}^{10}})^{2}. \label{eq:coupling2I11b}
\end{align}
Indeed, note \eqref{eq:coupling2I11a} is unchanged if we replace $\e^{-1}\delta(\e)\tau(\e)B(\e)$ in the supremum therein by $\sigma_{N(\e)}$. This follows from the fact $\tau^{\mathrm{st}}\leq\sigma_{N(\e)}\wedge\e^{-1}\delta(\e)\tau(\e)B(\e)$ that holds by construction (see before \eqref{eq:coupling2I2}). So, \eqref{eq:coupling2I11a} is deterministically bounded by the sum over $0\leq k\leq N(\e)$ of $\mathrm{LHS}\eqref{eq:coupling2I10}$, at which point \eqref{eq:coupling2I11a}-\eqref{eq:coupling2I11b} follows by summing \eqref{eq:coupling2I10} over $0\leq k\leq N(\e)$. Now, let us choose $\tau(\e)\to0$ sufficiently fast so that $\e^{-2}\tau(\e)^{2}\delta(\e)^{2}\to\infty$ as slow as we want. Choose $B(\e),N(\e)\to\infty$ as slow as we want so that, recalling $\alpha(\e)\to0$, we have $\eqref{eq:coupling2I11b}\times i(\langle\Phi^{\e}(s,\cdot)\rangle_{\mathscr{C}^{10}})^{-2}\to0$ as $\e\to0$. (We clarify that the speed at which $\alpha(\e)\to0$ from Lemma \ref{lemma:coupling1} does not change if all we do is possibly lower $\tau(\e),B(\e)$. Also, all we require from $\tau(\e)$ is that it vanishes as $\e\to0$ and $\e^{_1}\tau(\e)\delta(\e)\to\infty$; this constraint is still satisfied; see Lemma \ref{lemma:timeaverage} and Proposition \ref{prop:coupling}.) Thus, by the Markov inequality and \eqref{eq:coupling2I11a}-\eqref{eq:coupling2I11b}, we deduce that with probability at least $1-\gamma_{1}(\e)$, for some $\gamma_{1}(\e)\to0$,
\begin{align}
\sup_{0\leq t\leq\e^{-1}\delta(\e)\tau(\e)B(\e)}|\mathbf{x}^{\e}(s,t\wedge\tau^{\mathrm{st}})-\wt{\mathbf{x}}^{\e}(s,t\wedge\tau^{\mathrm{st}})| \ \lesssim \ \alpha_{1}(\e)\times i(\langle\Phi^{\e}(s,\cdot)\rangle_{\mathscr{C}^{10}}), \label{eq:coupling2I12}
\end{align}
where $\alpha_{1}(\e)\to0$ as $\e\to0$.
\end{enumerate}
By \eqref{eq:coupling2I12}, to complete this proof, it now suffices to show that for some choice of $N(\e),B(\e)\to\infty$, we have
\begin{align}
\mathbf{P}\left\{\tau^{\mathrm{st}}<\e^{-1}\delta(\e)\tau(\e)B(\e)\right\} \ \leq \ \gamma_{2}(\e), \label{eq:coupling2I13}
\end{align}
for some $\gamma_{2}(\e)\to0$. To this end, recall $\tau^{\mathrm{st}}:=\tau^{\mathrm{sep}}\wedge\tau^{\mathrm{geo}}\wedge\sigma_{N(\e)}\wedge\e^{-1}\tau(\e)\delta(\e)B(\e)$. By union bound, we get
\begin{align}
\mathbf{P}\big\{\tau^{\mathrm{st}}  < \e^{-1}\delta(\e) & \tau(\e)B(\e)\big\} \ \leq \nonumber \\
& \mathbf{P}\left\{\tau^{\mathrm{geo}}<\e^{-1}\delta(\e)\tau(\e)B(\e)\right\} + \mathbf{P}\left\{\tau^{\mathrm{st}}=\tau^{\mathrm{sep}}<\e^{-1}\delta(\e)\tau(\e)B(\e)\right\}\label{eq:coupling2I14a}\\
+ \ &\mathbf{P}\left\{\tau^{\mathrm{st}}=\sigma_{N(\e)}<\e^{-1}\delta(\e)\tau(\e)B(\e)\right\}. \label{eq:coupling2I14b}
\end{align}
%
\begin{enumerate}
\item First, by Lemma \ref{lemma:coupling1}, the first term on the \abbr{RHS} of \eqref{eq:coupling2I14a} vanishes as $\e\to0$. To control the second term on the \abbr{RHS} of \eqref{eq:coupling2I14a}, note that on the event $\tau^{\mathrm{st}}=\tau^{\mathrm{sep}}<\e^{-1}\delta(\e)\tau(\e)B(\e)$, by the almost sure continuity of $\mathbf{x}^{\e}(s,\cdot),\wt{\mathbf{x}}^{\e}(s,\cdot)$ (see Lemma \ref{lemma:sde}), with probability 1 we have $|\mathbf{x}^{\e}(s,\tau^{\mathrm{st}})-\wt{\mathbf{x}}^{\e}(s,\tau^{\mathrm{st}})|=1/999$. But \eqref{eq:coupling2I12} says that this can only happen with probability $\mathrm{o}(1)$. (Indeed, as assumed at the beginning of this proof, we know $i(\langle\Phi^{\e}(s,\cdot)\rangle_{\mathscr{C}^{10}})\leq\mu(\e)^{-1}$; now take $\mu(\e)\to0$ sufficiently slowly depending on $\alpha_{1}(\e)$ so that the \abbr{RHS} of \eqref{eq:coupling2I12} vanishes and is thus $\ll1/999$. To clarify this, we recall that $\mu(\e)$ can be chosen arbitrarily, so said choice of $\mu(\e)$ is allowed; see the beginning of this argument.) In particular, we deduce that the \abbr{RHS} of \eqref{eq:coupling2I14a} vanishes as $\e\to0$. 
\item It remains to control \eqref{eq:coupling2I14b}, for some choice of $\tau(\e)$ that satisfies constraints of Lemma \ref{lemma:timeaverage} and some choice of $B(\e)\to\infty$ (sufficiently slowly). In words, we want the number of times we switch coordinate charts to be less than $N(\e)$, for $N(\e)\to\infty$ also sufficiently slowly. We first claim that if $\tau^{\mathrm{st}}=\sigma_{N(\e)}<\e^{-1}\delta(\e)\tau(\e)B(\e)$, the $\mathscr{C}^{1/3}_{t}$-Holder norm of either $t\mapsto\mathbf{x}^{\e}(s,t)$ or $t\mapsto\wt{\mathbf{x}}^{\e}(s,t)$ (for $0\leq t\leq\e^{-1}\delta(\e)\tau(\e)B(\e)$) must be $\gtrsim N(\e)^{1/3}\e^{1/3}\delta(\e)^{-1/3}\tau(\e)^{-1/3}B(\e)^{-1/3}$. Indeed, each time we switch coordinate charts, one of $\mathbf{x}^{\e},\wt{\mathbf{x}}^{\e}$ must travel distance of $1/999$. Doing so $N(\e)$-many times by time $\e^{-1}\delta(\e)\tau(\e)B(\e)$ means that at least once, we must have switched coordinate charts in time less than $N(\e)^{-1}\e^{-1}\delta(\e)\tau(\e)B(\e)$, so at least one of $\mathbf{x}^{\e},\wt{\mathbf{x}}^{\e}$ must have traveled distance $\gtrsim1$ in time at most $N(\e)^{-1}\e^{-1}\delta(\e)\tau(\e)B(\e)$. The claim now follows.

Now, recall we can take $\tau(\e)\to0$ in whatever fashion, as long as $\e^{-1}\delta(\e)\tau(\e)\to\infty$. So, we can take $\e^{-1}\delta(\e)\tau(\e)\to\infty$ as slow as we want. We can also take $B(\e)\to\infty$ as slow as we want. This means we can take $N(\e)\to\infty$ as slow as we want while still having the divergence $N(\e)^{1/3}\e^{1/3}\delta(\e)^{-1/3}\tau(\e)^{-1/3}B(\e)^{-1/3}\to\infty$ (if we choose $N(\e)$ to depend only on $\e,\delta(\e),\tau(\e),B(\e)$). In particular, there exists $H(\e)\to\infty$ such that if $\mathscr{E}$ denotes the event where the $\mathscr{C}^{1/3}_{t}$-Holder norm of either $t\mapsto\mathbf{x}^{\e}(s,t)$ or $t\mapsto\wt{\mathbf{x}}^{\e}(s,t)$ (for $0\leq t\leq\e^{-1}\delta(\e)\tau(\e)B(\e)$) exceeds $H(\e)$, then
\begin{align}
\mathbf{P}\left\{\tau^{\mathrm{st}}=\sigma_{N(\e)}<\e^{-1}\delta(\e)\tau(\e)B(\e)\right\} \ \leq \ \mathbf{P}\{\mathscr{E}\}. \label{eq:coupling2I15}
\end{align}
However, by Lemma \ref{lemma:sde}, we can take $\xi(\e)\to0$ such that with probability at least $1-\xi(\e)$, we have 
\begin{align}
 \sup_{\substack{r,t\in[0,\e^{-1}\delta(\e)\tau(\e)B(\e)]\\ r\neq t}} \frac{|\mathbf{x}^{\e}(s,r)-\mathbf{x}^{\e}(s,t)|}{|t-r|^{\frac13}} & \lesssim \nonumber \\
 \ h(\xi(\e)^{-1})\times & \sup_{0\leq t\leq\e^{-1}\delta(\e)\tau(\e)B(\e)} 
 j(\e^{-1}\delta(\e)\tau(\e)B(\e),\langle\Psi^{\e}_{s}(t,\cdot)\rangle_{\mathscr{C}^{10}}) \label{eq:coupling2I16a}\\
\sup_{\substack{r,t\in[0,\e^{-1}\delta(\e)\tau(\e)B(\e)]\\ r\neq t}}\frac{|\wt{\mathbf{x}}^{\e}(s,r)-\wt{\mathbf{x}}^{\e}(s,t)|}{|t-r|^{\frac13}} \ &\lesssim \ h(\xi(\e)^{-1})\times j(\e^{-1}\delta(\e)\tau(\e)B(\e),\langle\Phi^{\e}(s,\cdot)\rangle_{\mathscr{C}^{10}}).\label{eq:coupling2I16b}
\end{align}
Above, $h(\cdot)$ and $j(\cdot,\cdot)$ are increasing and continuous. Recall that we can take $\e^{-1}\delta(\e)\tau(\e)B(\e)\to\infty$ as slow as we want. Also, recall that $\langle\Phi^{\e}(s,\cdot)\rangle_{\mathscr{C}^{10}}$ is allowed to diverge, but as slowly as we want; see the beginning of this proof. Moreover, by Lemma \ref{lemma:timeaverage}, on an event of probability $1-\mathrm{o}(1)$, we can replace $\Psi^{\e}_{s}(t,\cdot)$ by $\Phi^{\e}(s,\cdot)$. Thus, if we choose $\xi(\e)\to0$ sufficiently slow, we can choose the \abbr{RHS} of the previous two displays to grow as slow as we want. But if the previous two displays hold, then we cannot be on the event $\mathscr{E}$. (Let us be clear about how $\e$-dependent constants relate to each other. On the event $\mathscr{E}$, the \abbr{LHS} of at least one of \eqref{eq:coupling2I16a}, \eqref{eq:coupling2I16b} must be $\geq H(\e)$, where $H(\e)=N(\e)^{1/3}\e^{1/3}\delta(\e)^{-1/3}\tau(\e)^{-1/3}B(\e)^{-1/3}$. But we have just argued that \eqref{eq:coupling2I16a}-\eqref{eq:coupling2I16b} hold with probability $1-\mathrm{o}(1)$. We can now choose the \abbr{RHS} of \eqref{eq:coupling2I16a}-\eqref{eq:coupling2I16b} to grow as slow as we want while still being much smaller than $H(\e)$, because $H(\e)$ has the extra factor $N(\e)^{1/3}$ at our disposal that does not show up in \eqref{eq:coupling2I16a}-\eqref{eq:coupling2I16b}. Of course, we can still ultimately choose $N(\e)\to\infty$ as slow as we want if we choose appropriate $\delta(\e),\tau(\e),B(\e),\xi(\e)$.) Thus, we know $\mathbf{P}\{\mathscr{E}\}\to0$ as $\e\to0$. Combining this with \eqref{eq:coupling2I15} shows that $\eqref{eq:coupling2I14b}\to0$ as $\e\to0$. 
\end{enumerate}
The previous two points and \eqref{eq:coupling2I14a}-\eqref{eq:coupling2I14b} give \eqref{eq:coupling2I13} for some $\gamma_{2}(\e)\to0$. As noted before \eqref{eq:coupling2I13}, this finishes the proof.
\end{proof}
\subsection{Estimating $k\mapsto|\tau_{k}-\wt{\tau}_{k}|$}
Thus far, we have compared both the metrics $\mathsf{g}^{\Phi}$ (for $\Phi=\Phi^{\e}(s,\cdot),\Psi^{\e}_{s}(t,\cdot)$) and the processes $t\mapsto\mathbf{x}^{\e}(s,t),\wt{\mathbf{x}}^{\e}(s,t)$. As mentioned before Lemma \ref{lemma:coupling1}, we are left to compare the stopping times $k\mapsto\tau_{k},\wt{\tau}_{k}$. Intuitively, these stopping times depend on the boundary local times of $t\mapsto\mathbf{x}^{\e}(s,t),\wt{\mathbf{x}}^{\e}(s,t)$, which should depend continuously (in some sense) on $t\mapsto\mathbf{x}^{\e}(s,t),\wt{\mathbf{x}}^{\e}(s,t)$ themselves (as well as the difference $\Psi^{\e}_{s}(t,\cdot)-\Phi^{\e}(s,\cdot)$, which is controlled by Lemma \ref{lemma:coupling1}). To make this precise, we use another application of Ito to a $\mathsf{g}^{\Phi}$-harmonic function (for $\Phi=\Phi^{\e}(s,\cdot),\Psi^{\e}_{s}(t,\cdot)$) combined with elliptic regularity estimates for said function, similar to the proof of Lemma \ref{lemma:coupling1}. (Indeed, the proof of Lemma \ref{lemma:coupling1} uses this to derive an equation for local times whose regularity properties we can then study.)
\begin{lemma}\label{lemma:coupling3}
Recall $B(\e)$ from \emph{Lemmas \ref{lemma:coupling1}, \ref{lemma:coupling2}}. For any integer $k\geq0$, define the following stopping times:
\begin{align}
\tau_{k,\wedge} \ := \ \tau_{k}\wedge\e^{-1}\tau(\e)\delta(\e)B(\e) \quad\mathrm{and}\quad \wt{\tau}_{k,\wedge} \ := \ \wt{\tau}_{k}\wedge\e^{-1}\tau(\e)\delta(\e)B(\e).
\end{align}
Fix any $k\geq0$ and any $t(\e)\to\infty$ as $\e\to0$. By possibly decreasing $B(\e)$ from \emph{Lemmas \ref{lemma:coupling1}, \ref{lemma:coupling2}}, though still guaranteeing that $B(\e)\to\infty$ and that $B(\e)$ depends only on $\e$, there exist $\alpha(\e),\gamma(\e)\to0$ as $\e\to0$ such that with probability at least $1-\gamma(\e)$,
\begin{align}
\tau_{k,\wedge}+\wt{\tau}_{k,\wedge} \ &\lesssim \ t(\e) k\delta(\e)\times i(\langle\Phi^{\e}(s,\cdot)\rangle_{\mathscr{C}^{10}}) \label{eq:coupling3I} \\
|\tau_{k,\wedge}-\wt{\tau}_{k,\wedge}| \ &\lesssim \ \alpha(\e)\times i(\langle\Phi^{\e}(s,\cdot)\rangle_{\mathscr{C}^{10}}). \label{eq:coupling3II}
\end{align}
\end{lemma}
\begin{remark}
The relevance of the time-scale $\e^{-1}\tau(\e)\delta(\e)$ for the $\mathbf{x}^{\e},\wt{\mathbf{x}}^{\e}$ processes can be seen by \eqref{eq:coupling3I} and recalling that we are sampling $k$ therein according to a Poisson process of speed $\e^{-1}\tau(\e)$ in Proposition \ref{prop:coupling}.
\end{remark}
\begin{proof}
We first prove \eqref{eq:coupling3I}. For this, we bound $\tau_{k,\wedge}$. The bound for $\wt{\tau}_{k,\wedge}$ follows by an identical argument but replacing $\Psi^{\e}_{s}(t,\cdot)$ and $\mathbf{x}^{\e}(s,\cdot)$ by $\Phi^{\e}(s,\cdot)$ and $\wt{\mathbf{x}}^{\e}(s,\cdot)$, respectively. First, as in the proof of Lemma \ref{lemma:coupling1}, let $\mathscr{U}$ solve
\begin{align}
\Delta_{\mathsf{g}^{\Psi^{\e}_{s}(t,\cdot)}}\mathscr{U}(t,x) = 1 \quad\mathrm{and}\quad \mathscr{U}(t,\cdot)|_{\partial\bUz} = 0. 
\end{align}
(As in the proof of Lemma \ref{lemma:coupling1}, we will stop $\Psi^{\e}_{s}(t,\cdot)$ and $\mathbf{x}^{\e}(s,t)$, and thus $\mathscr{U}$, before the explosion of $\Psi^{\e}_{s}(t,\cdot)$ in $\mathscr{C}^{10}_{\simeq}(\partial\bUz,\R^{\d})$.) We will want to use the Ito formula. However, we require first a technical cutoff in the form of an auxiliary stopping time. Define $\tau_{\mathrm{cut}}$ as the first non-negative time such that \eqref{eq:coupling1I} fails. Now, define $\tau_{k,\wedge,\mathrm{cut}}:=\tau_{k,\wedge}\wedge\tau_{\mathrm{cut}}$. We claim that by Ito,
\begin{align}
\E\mathscr{U}(\tau_{k,\wedge,\mathrm{cut}},\mathbf{x}^{\e}(s,\tau_{k,\wedge})) \ = \ \E\tau_{k,\wedge,\mathrm{cut}} + \E\int_{0}^{\tau_{k,\wedge,\mathrm{cut}}}\left\{\mathbf{n}\cdot\grad\mathscr{U}(r,\cdot)\right\}(\mathbf{x}^{\e}(s,r))\d\mathbf{L}^{s}(r). \label{eq:coupling3I1}
\end{align}
Indeed, there is no initial data term on the \abbr{RHS} because $\mathbf{x}^{\e}(s,0)\in\partial\bUz$ and $\mathscr{U}$ vanishes on $\partial\bUz$. The first term on the \abbr{RHS} of \eqref{eq:coupling3I1} comes from $\{\partial_{r}+\Delta_{\mathsf{g}^{\Psi^{\e}_{s}(r,\cdot)}}\}\mathscr{U}(r,\mathbf{x}^{\e}(s,r))$ as explained before \eqref{eq:coupling1I5}.

Now, we detour from the proof of Lemma \ref{lemma:coupling1}. First, we use elliptic regularity to get that $|\grad\mathscr{U}(t,x)|\lesssim i_{1}(\langle\Psi^{\e}_{s}(t,\cdot)\rangle_{\mathscr{C}^{10}})$ for continuous, increasing, and deterministic $i_{1}(\cdot)\geq1$. (The implied constant is independent of $t,x$.) But time $\tau_{k,\wedge,\mathrm{cut}}\leq\tau_{\mathrm{cut}}$, so $i_{1}(\langle\Psi^{\e}_{s}(t,\cdot)\rangle_{\mathscr{C}^{10}})\lesssim i(\langle\Phi^{\e}(s,\cdot)\rangle_{\mathscr{C}^{10}})$ for another continuous, increasing, and deterministic $i(\cdot)\geq1$. This bounds
\begin{align}
\E\int_{0}^{\tau_{k,\wedge,\mathrm{cut}}}\left\{\mathbf{n}\cdot\grad\mathscr{U}(r,\cdot)\right\}(\mathbf{x}^{\e}(s,r))\d\mathbf{L}^{s}(r) \ &\lesssim \ i(\langle\Phi^{\e}(s,\cdot)\rangle_{\mathscr{C}^{10}})\times \E\mathbf{L}^{s}(\tau_{k,\wedge,\mathrm{cut}}) \label{eq:coupling3I1b}\\
&\lesssim \ k\delta(\e)\times i(\langle\Phi^{\e}(s,\cdot)\rangle_{\mathscr{C}^{10}}), \nonumber
\end{align}
where the last bound follows because the local time is monotone non-decreasing and $\tau_{k,\wedge,\mathrm{cut}}\leq\tau_{k}$ and because $\mathbf{L}^{s}(\tau_{k})=k\delta(\e)$ just by definition of $\tau_{k}$. (Recall that $\E$ conditions on $\Phi^{\e}(s,\cdot)$, which is why there is no expectation on the far \abbr{RHS} of the above display.) We claim the \abbr{LHS} of \eqref{eq:coupling3I1} is $\leq0$. Indeed, this follows by the maximum principle applied to the defining \abbr{PDE} for $\mathscr{U}$. (For example, if the Laplacian therein were the Euclidean Laplacian and if $\bUz$ were a one-dimensional interval, then $\mathscr{U}$ would be a upward-facing parabola with zero boundary conditions.) From this and the previous two displays, we deduce
\begin{align}
\E\tau_{k,\wedge,\mathrm{cut}} \ \lesssim \ k\delta(\e)\times i(\langle\Phi^{\e}(s,\cdot)\rangle_{\mathscr{C}^{10}}). \label{eq:coupling3I2}
\end{align}
Therefore, by the Markov inequality, with probability at least $1-\gamma(\e)$, we have $\tau_{k,\wedge,\mathrm{cut}}\lesssim t(\e)k\delta(\e)\times i(\langle\Phi^{\e}(s,\cdot)\rangle_{\mathscr{C}^{10}})$ for any $t(\e)\to\infty$ (for $\gamma(\e)\to0$ depending only on $t(\e)$). It remains to obtain $\tau_{k,\wedge,\mathrm{cut}}=\tau_{k,\wedge}$ with probability at least $1-\gamma'(\e)$ (for some $\gamma'(\e)\to0$). But $\tau_{k,\wedge,\mathrm{cut}}=\tau_{k,\wedge}\wedge\tau_{\mathrm{cut}}$, where $\tau_{\mathrm{cut}}$ is the first time \eqref{eq:coupling1I} fails; see before \eqref{eq:coupling3I1}. Now use Lemma \ref{lemma:coupling1}.

We are left to prove \eqref{eq:coupling3II}. Define $\tau_{\mathrm{ap}}:=\tau_{\mathrm{ap},1}\wedge\tau_{\mathrm{ap},2}\wedge\tau_{\mathrm{cut}}$, where $\tau_{\mathrm{cut}}$ is the same as before, and
\begin{align}
\tau_{\mathrm{ap},1} \ &:= \ \inf\{t\geq0: |\mathbf{x}^{\e}(s,t)-\wt{\mathbf{x}}^{\e}(s,t)| \gtrsim \alpha(\e)\times i(\langle\Phi^{\e}(s,\cdot)\rangle_{\mathscr{C}^{10}})\}\wedge\e^{-1}\tau(\e)\delta(\e)B(\e)\\
\tau_{\mathrm{ap},2} \ &:= \ \inf\{t\geq0: |\mathbf{L}^{s}(t)-\mathbf{L}^{s,\sim}(t)| \gtrsim \beta(\e)\times i(\langle\Phi^{\e}(s,\cdot)\rangle_{\mathscr{C}^{10}})\}\wedge\e^{-1}\tau(\e)\delta(\e)B(\e).
\end{align}
In words, $\tau_{\mathrm{ap},1}$ gives an a priori estimate for the reflecting Brownian motions, and $\tau_{\mathrm{ap},2}$ gives one for their respective local times. Above, $\beta(\e)\to0$ will be determined at the end of this proof. All we need is that it vanishes as $\e\to0$ (as slowly as we want). Let  us also recall that $\alpha(\e)$ is from Lemma \ref{lemma:coupling3}, and $B(\e)$ is a constant to be chosen at the end of this argument that satisfies the constraints of Lemmas \ref{lemma:coupling1} and \ref{lemma:coupling2} (i.e. $B(\e)\to\infty$). Lastly, define $\tau_{+}:=(\tau_{k,\wedge}\vee\wt{\tau}_{k,\wedge})\wedge\tau_{\mathrm{ap}}$ and $\tau_{-}:=(\tau_{k,\wedge}\wedge\wt{\tau}_{k,\wedge})\wedge\tau_{\mathrm{ap}}$.

We now apply the Ito formula from time $r=\tau_{-}$ to time $r=\tau_{+}$ as in the proof of \eqref{eq:coupling3I}. For this, we are allowed to condition on the filtration at time $\tau_{-}$ by the strong Markov property. Thus, for now, assume that $\tau_{-}=\wt{\tau}_{k,\wedge}\wedge\tau_{\mathrm{ap}}$. (The only other option is $\tau_{-}=\tau_{k,\wedge}\wedge\tau_{\mathrm{ap}}$. In this case, just repeat the following analysis except replacing $\Psi^{\e}_{s}(t,\cdot)$ and $\mathbf{x}^{\e}(s,\cdot)$ by $\Phi^{\e}(s,\cdot)$ and $\wt{\mathbf{x}}^{\e}(s,\cdot)$, respectively.) Let $\mathscr{U}$ be the same \abbr{PDE} solution from the beginning of this proof. We have 
\begin{align}
\E\mathscr{U}(\tau_{+}, & \mathbf{x}^{\e}(s,\tau_{+})) \nonumber \\
& = \ \E\mathscr{U}(\tau_{-},\mathbf{x}^{\e}(s,\tau_{-})) + \E\{\tau_{+}-\tau_{-}\} + \E\int_{\tau_{-}}^{\tau_{+}}\{\mathbf{n}\cdot\grad\mathscr{U}(r,\cdot)\}(\mathbf{x}^{\e}(s,r))\d\mathbf{L}^{s}(r). \label{eq:coupling3II1}
\end{align}
\eqref{eq:coupling3II1} follows by the same reasoning as \eqref{eq:coupling3I1}; we have just included the initial data term at time $\tau_{-}$ on the \abbr{RHS} of \eqref{eq:coupling3II1}. Using the same reasoning that gives us \eqref{eq:coupling3I1b}, we can control the last term on the \abbr{RHS} of \eqref{eq:coupling3II1} via
\begin{align}
\E\int_{\tau_{-}}^{\tau_{+}}\{\mathbf{n}\cdot\grad\mathscr{U}(r,\cdot)\}(\mathbf{x}^{\e}(s,r))\d\mathbf{L}^{s}(r) \ \lesssim \ i(\langle\Phi^{\e}(s,\cdot)\rangle_{\mathscr{C}^{10}})\times\E|\mathbf{L}^{s}(\tau_{+})-\mathbf{L}^{s}(\tau_{-})|. \label{eq:coupling3II2}
\end{align}
(In words, because of elliptic regularity, we have control of $\grad\mathscr{U}(r,\cdot)$ in terms of the $\Psi^{\e}_{s}(r,\cdot)$ metric, which itself is controlled in terms of the $\Phi^{\e}(s,\cdot)$ metric before time $\tau_{+}\leq\tau_{\mathrm{cut}}$ by definition of $\tau_{\mathrm{cut}}$. Then, we are just integrating differentiated local time on $[\tau_{-},\tau_{+}]$.) On the other hand, we claim that for $i_{1}(\cdot)\geq1$ continuous, increasing, and deterministic,
\begin{align}
\E\mathscr{U}(\tau_{+},\mathbf{x}^{\e}(s,\tau_{+}))-\E\mathscr{U}(\tau_{-},\mathbf{x}^{\e}(s,\tau_{-})) \ \lesssim \ \alpha(\e)\times i_{1}(\langle\Phi^{\e}(s,\cdot)\rangle_{\mathscr{C}^{10}}). \label{eq:coupling3II3}
\end{align}
To see this, recall $\tau_{-}=\tau_{k,\wedge}\wedge\wt{\tau}_{k,\wedge}\wedge\tau_{\mathrm{ap}}$ and $\tau_{+}=(\tau_{k,\wedge}\vee\wt{\tau}_{k,\wedge})\wedge\tau_{\mathrm{ap}}$. If $\tau_{-}=\tau_{\mathrm{ap}}$, then $\tau_{\mathrm{ap}}\leq\tau_{k,\wedge},\wt{\tau}_{k,\wedge}$, so $\tau_{+}=\tau_{\mathrm{ap}}=\tau_{-}$, meaning the \abbr{LHS} of \eqref{eq:coupling3II3} is zero on this event. So, condition on $\tau_{-}=\tau_{k,\wedge}\wedge\wt{\tau}_{k,\wedge}=\tau_{k}\wedge\wt{\tau}_{k}\wedge\e^{-1}\delta(\e)\tau(\e)B(\e)$. By the same reasoning we have just given, if $\tau_{-}=\e^{-1}\delta(\e)\tau(\e)B(\e)$, then $\tau_{+}=\tau_{-}$ (since $\tau_{+}\leq\e^{-1}\delta(\e)\tau(\e)B(\e)$ by construction). Thus, let us condition on the event $\tau_{-}=\tau_{k}\wedge\wt{\tau}_{k}$. If $\tau_{-}=\tau_{k}$, then the second term on the \abbr{LHS} of \eqref{eq:coupling3II3} is zero, since $\mathbf{x}^{\e}(s,\tau_{k})\in\partial\bUz$ and $\mathscr{U}$ vanishes on $\partial\bUz$, both by construction. Moreover, recall from the proof of \eqref{eq:coupling3I} that $\mathscr{U}\leq0$ with probability 1. This shows that on the event $\tau_{-}=\tau_{k}$, the expectation gives us $\leq0$. It remains to condition on $\tau_{-}=\wt{\tau}_{k}$. In this case, we still know the first term on the \abbr{LHS} of \eqref{eq:coupling3II3} is $\leq0$. As for the second term on the \abbr{LHS} of \eqref{eq:coupling3II3}, we first replace $\mathbf{x}^{\e}(s,\wt{\tau}_{k})$ by $\wt{\mathbf{x}}^{\e}(s,\wt{\tau}_{k})$. Since we are on the event where $\wt{\tau}_{k}\leq\tau_{\mathrm{ap}}$ (as we have conditioned on $\wt{\tau}_{k}=\tau_{-}\leq{\tau}_{\mathrm{ap}}$, the last of which follows by construction of $\tau_{-}$), we know that $|\mathbf{x}^{\e}(s,\wt{\tau}_{k})-\wt{\mathbf{x}}^{\e}(s,\wt{\tau}_{k})|\lesssim\alpha(\e)\times i(\langle\Phi^{\e}(s,\cdot)\rangle_{\mathscr{C}^{10}})$ (see the definition of $\tau_{\mathrm{ap}}\leq\tau_{\mathrm{ap},1}$). So, because the gradients of $\mathscr{U}(t,\cdot)$ are controlled by $i(\langle\Psi^{\e}_{s}(t,\cdot)\rangle_{\mathscr{C}^{10}})\lesssim i_{2}(\langle\Phi^{\e}(t,\cdot)\rangle_{\mathscr{C}^{10}})$, which holds since we work before $\tau_{\mathrm{ap}}\leq\tau_{\mathrm{cut}}$, the expectation on the current event is $\lesssim\mathrm{\abbr{RHS}}\eqref{eq:coupling3II3}$ just by the mean-value theorem. Thus, \eqref{eq:coupling3II3} follows. Now, the previous three displays give
\begin{align}
\E\{\tau_{+}-\tau_{-}\} \ \lesssim \ i(\langle\Phi^{\e}(s,\cdot)\rangle_{\mathscr{C}^{10}})\times\E|\mathbf{L}^{s}(\tau_{+})-\mathbf{L}^{s}(\tau_{-})| + \alpha(\e)\times i_{1}(\langle\Phi^{\e}(s,\cdot)\rangle_{\mathscr{C}^{10}}). \label{eq:coupling3II3b}
\end{align}
First, note that we can drop the absolute value bars on the \abbr{RHS} of \eqref{eq:coupling3II3b}, since the local time is non-decreasing and $\tau_{-}\leq\tau_{+}$. So, for the first term on the \abbr{RHS} of \eqref{eq:coupling3II3b}, we can write
\begin{align}
\mathbf{L}^{s}(\tau_{+})-\mathbf{L}^{s}(\tau_{-}) \ = \ &\{\mathbf{L}^{s,\sim}(\tau_{-})-\mathbf{L}^{s}(\tau_{-})\}+\{\mathbf{L}^{s}(\tau_{+})-\mathbf{L}^{s,\sim}(\tau_{-})\}. \label{eq:coupling3II4}
\end{align}
Since $\tau_{+},\tau_{-}\leq\tau_{\mathrm{ap}}$, the first braced term on the \abbr{RHS} of \eqref{eq:coupling3II4} is $\lesssim\beta(\e)\times i_{3}(\langle\Phi^{\e}(s,\cdot)\rangle_{\mathscr{C}^{10}})$. For the second braced term, recall that $\tau_{-}=\tau_{k,\wedge}\wedge\wt{\tau}_{k,\wedge}\wedge\tau_{\mathrm{ap}}$ and $\tau_{+}=(\tau_{k,\wedge}\vee\wt{\tau}_{k,\wedge})\wedge\tau_{\mathrm{ap}}$. If $\tau_{-}\neq\wt{\tau}_{k}$, then $\tau_{-}=\tau_{\mathrm{ap}}$. (Indeed, we assumed after conditioning that $\tau_{-}=\wt{\tau}_{k,\wedge}\wedge\tau_{\mathrm{ap}}$. Now, if $\tau_{-}=\wt{\tau}_{k,\wedge}$ and $\tau_{-}\neq\wt{\tau}_{k}$, then $\tau_{-}=\e^{-1}\delta(\e)\tau(\e)B(\e)$; see the statement of Lemma \ref{lemma:coupling3}. Since $\tau_{-}\leq\tau_{\mathrm{ap}}\leq\e^{-1}\delta(\e)\tau(\e)B(\e)$ by assumption and construction, we get $\tau_{-}=\tau_{\mathrm{ap}}$.) But $\tau_{-}\leq\tau_{+}\leq\tau_{\mathrm{ap}}$, so on this event, we know $\tau_{-}=\tau_{+}$, and the second term on the \abbr{RHS} of \eqref{eq:coupling3II4} is zero on this event. Thus, let us condition on the event $\tau_{-}=\wt{\tau}_{k}$. In this case, we know $\mathbf{L}^{s,\sim}(\tau_{-})=k\delta(\e)$. On the other hand, since local time is non-decreasing and $\tau_{+}=(\tau_{k,\wedge}\vee\wt{\tau}_{k,\wedge})\wedge\tau_{\mathrm{ap}}\leq\tau_{k}$, we know $\mathbf{L}^{s}(\tau_{+})\leq\mathbf{L}^{s}(\tau_{k})=k\delta(\e)$. Thus, the second braced term on the \abbr{RHS} of \eqref{eq:coupling3II4} is $\leq0$. In particular, we deduce from \eqref{eq:coupling3II3b}, \eqref{eq:coupling3II4}, and this paragraph that 
\begin{align}
\E\{\tau_{+}-\tau_{-}\} \ \lesssim \ \beta(\e)\times i_{4}(\langle\Phi^{\e}(s,\cdot)\rangle_{\mathscr{C}^{10}}) + \alpha(\e)\times i_{1}(\langle\Phi^{\e}(s,\cdot)\rangle_{\mathscr{C}^{10}}).
\end{align}
Note $\tau_{+}-\tau_{-}=|\tau_{+}-\tau_{-}|$, since $\tau_{+}\geq\tau_{-}$. Thus, by \eqref{eq:coupling3II4} and the Markov inequality, we deduce that \eqref{eq:coupling3II} is true if we replace $\tau_{k,\wedge}-\wt{\tau}_{k,\wedge}$ by $\tau_{+}-\tau_{-}$. Since $\tau_{-}=\tau_{k,\wedge}\wedge\wt{\tau}_{k,\wedge}\wedge\tau_{\mathrm{ap}}$ and $\tau_{+}=(\tau_{k,\wedge}\vee\wt{\tau}_{k,\wedge})\wedge\tau_{\mathrm{ap}}$, in order to show \eqref{eq:coupling3II}, we are left to show
\begin{align}
\mathbf{P}\{\tau_{\mathrm{ap}}=\e^{-1}\tau(\e)\delta(\e)B(\e)\} \ \geq \ 1-\xi(\e), \label{eq:coupling3II5}
\end{align}
where $\xi(\e)\to0$ as $\e\to0$. (Indeed, on this event, taking minimum with $\tau_{\mathrm{ap}}$ does nothing, since $\tau_{k,\wedge},\wt{\tau}_{k,\wedge}\leq\e^{-1}\tau(\e)\delta(\e)B(\e)$ by construction; see the statement of Lemma \ref{lemma:coupling3}. So, we would deduce from \eqref{eq:coupling3II5} that $\tau_{k,\wedge}-\wt{\tau}_{k,\wedge}=\tau_{+}-\tau_{-}$ with probability $1-\mathrm{o}(1)$, meaning that if \eqref{eq:coupling3II} holds for $\tau_{+}-\tau_{-}$ in place of $\tau_{k,\wedge}-\wt{\tau}_{k,\wedge}$, then it holds as written \eqref{eq:coupling3II} on an event of slightly smaller but still $1-\mathrm{o}(1)$ probability, which is enough.) To show \eqref{eq:coupling3II5}, by a union bound and $\tau_{\mathrm{ap}}=\tau_{\mathrm{ap},1}\wedge\tau_{\mathrm{ap},2}\wedge\tau_{\mathrm{cut}}$, we have
\begin{align}
\mathbf{P}\{\tau_{\mathrm{ap}}<\e^{-1}\tau(\e)\delta(\e)B(\e)\} \ \leq \ &\mathbf{P}\{\tau_{\mathrm{ap},2}<\e^{-1}\tau(\e)\delta(\e)B(\e), \ \tau_{\mathrm{ap},2}\leq\tau_{\mathrm{ap},1}\wedge\tau_{\mathrm{cut}}\}\label{eq:coupling3II6}\\
+ \ &\mathbf{P}\{\tau_{\mathrm{ap},1}<\e^{-1}\tau(\e)\delta(\e)B(\e)\}+\mathbf{P}\{\tau_{\mathrm{cut}}<\e^{-1}\tau(\e)\delta(\e)B(\e)\}. 
\end{align}
As explained earlier in the proof of \eqref{eq:coupling3I}, Lemma \ref{lemma:coupling1} implies the last term in the second line is $\mathrm{o}(1)$ as $\e\to0$. The same is true by Lemma \ref{lemma:coupling2} for the first term in the second line. Thus, it suffices to control the \abbr{RHS} of \eqref{eq:coupling3II6}. First, we approximate the local times $\mathbf{L}^{s}$ and $\mathbf{L}^{s,\sim}$ by smoother versions. In particular, by Theorem 2.6 of \cite{BCS} (with modifications explained shortly), for any $\mathfrak{t}\geq0$ and $\beta_{1},\beta_{2},\beta_{3}>0$, there is a choice of $\beta_{0}=\beta_{0}(\beta_{1},\beta_{2},\beta_{3})$ such that with probability at least $1-\beta_{3}$,
\begin{align}
\sup_{0\leq t\leq\mathfrak{t}}|\mathbf{L}^{s}(t)-\mathbf{L}^{s,\beta_{0}}(t)|+\sup_{0\leq t\leq\mathfrak{t}}|\mathbf{L}^{s,\sim}(t)-\mathbf{L}^{s,\sim,\beta_{0}}(t)| \ \lesssim \ \beta_{1}+\beta_{2}\times\sup_{0\leq t\leq\mathfrak{t}} i(\mathfrak{t},\langle\Psi^{\e}_{s}(t,\cdot)\rangle_{\mathscr{C}^{10}}), \label{eq:coupling3II7}
\end{align}
where $i(\cdot,\cdot)$ is jointly continuous, increasing, and deterministic, and in what follows, $\Upsilon_{\beta_{0}}:\R^{\d}\to\R$ is smooth with derivatives depending continuous only on the parameter $\beta_{0}>0$:
\begin{align}
\mathbf{L}^{s,\beta_{0}}(t) \ := \ \int_{0}^{t}\Upsilon_{\beta_{0}}(\mathbf{x}^{\e}(s,r))\d r \quad\mathrm{and}\quad \mathbf{L}^{s,\sim,\beta_{0}}(t) \ := \ \int_{0}^{t}\Upsilon_{\beta_{0}}(\wt{\mathbf{x}}^{\e}(s,r))\d r.
\end{align}
In words, the local time $\mathbf{L}^{s}(t)$ is, at least morally, $\mathbf{L}^{s,\beta_{0}}(t)$ for $\beta_{0}=0$ (in which case $\Upsilon_{\beta_{0}=0}$ is a delta function on the boundary $\partial\bUz$). Thus, we are just smoothing the delta function a little. More precisely, for any scale of accuracy measured by $\beta_{1},\beta_{2}>0$, we can approximate the local time by a smoother version with derivatives depending continuously on $\beta_{1},\beta_{2}$. The error in this smoothing depends on the underlying geometry $\Psi^{\e}_{s}(t,\cdot)$ and deteriorates on the time-scale $\mathfrak{t}$, both in continuous fashion; this is where the function $i(\cdot,\cdot)$ comes from. (Technically, Theorem 2.6 in \cite{BCS} gives the previous claim but letting $\Upsilon_{\beta_{0}}$ be the normalized indicator of a small neighborhood of $\partial\bUz$ of width-scale $\beta_{0}$. This, of course, is not smooth. However, the proof of Theorem 2.6 in \cite{BCS} still gives the above claim with smooth $\Upsilon_{\beta_{0}}$ if we instead solve the equation $\Delta_{\mathsf{g}^{\Phi}}\mathscr{V}\approx\Upsilon_{\beta_{0}}$ with smooth $\Upsilon_{\beta_{0}}$ therein, where $\Phi=\Phi^{\e}(s,\cdot),\Psi^{\e}_{s}(t,\cdot)$ in this context. Indeed, \cite{BCS} does exactly this but for the aforementioned non-smooth $\Upsilon_{\beta_{0}}$.)

Let us fix $\beta_{0},\beta_{1},\beta_{2},\beta_{3}$ in the previous setting for now. We will decide on their values (and how they depend on $\e$ as $\e\to0$ at the end of this argument). Because $\tau_{\mathrm{ap},2}\leq\tau_{\mathrm{ap},1}$ by assumption of the event we are working on, via Taylor expansion, we have 
\begin{align}
\sup_{0\leq t\leq\mathfrak{t}}|\mathbf{L}^{s,\beta_{0}}(t)-\mathbf{L}^{s,\sim,\beta_{0}}(t)| \ &\lesssim \ i(\beta_{0}^{-1}) \int_{0}^{\mathfrak{t}}|\mathbf{x}^{\e}(s,r)-\wt{\mathbf{x}}^{\e}(s,r)|\d r \label{eq:coupling3II8}\\
&\lesssim \ \mathfrak{t}\times i(\beta_{0}^{-1})\times \alpha(\e) \times i(\langle\Phi^{\e}(s,\cdot)\rangle_{\mathscr{C}^{10}}). \nonumber
\end{align}
Combining \eqref{eq:coupling3II7} and \eqref{eq:coupling3II8} with the triangle inequality, we deduce the following with probability at least $1-\beta_{3}$:
\begin{align}
\sup_{0\leq t\leq\mathfrak{t}}|\mathbf{L}^{s}(t)-\mathbf{L}^{s,\sim}(t)| \ \lesssim \ \mathfrak{t}\times i(\beta_{0}^{-1})\times \alpha(\e) \times i(\langle\Phi^{\e}(s,\cdot)\rangle_{\mathscr{C}^{10}})+\beta_{1}+\beta_{2}\times\sup_{0\leq t\leq\mathfrak{t}} i(\mathfrak{t},\langle\Psi^{\e}_{s}(t,\cdot)\rangle_{\mathscr{C}^{10}}). \label{eq:coupling3II9}
\end{align}
Because $\tau_{\mathrm{ap},2}\leq\tau_{\mathrm{cut}}$ by assumption, we can replace $\Psi^{\e}_{s}(t,\cdot)$ on the \abbr{RHS} of \eqref{eq:coupling3II9} by $\Phi^{\e}(s,\cdot)$. We now choose the time-scale $\mathfrak{t}=\e^{-1}\delta(\e)\tau(\e)B(\e)$. We also let $\beta_{1},\beta_{2},\beta_{3}\to0$ as $\e\to0$ sufficiently slowly so that if we choose $B(\e)\to\infty$ sufficiently slowly as well, then $\mathfrak{t}\to\infty$ and $i(\beta_{0}^{-1})\to\infty$ sufficiently slowly and $\mathfrak{t}\times i(\beta_{0}^{-1})\times \alpha(\e)\to0$. (One can check immediately that the constant $\alpha(\e)$, which is from Lemma \ref{lemma:coupling2}, would only possibly depend on $B(\e)$ and is allowed to stay the same if all we do is make $B(\e)$ smaller. Thus, there exist such choices of $\beta_{1},\beta_{2},\beta_{3}\to0$ and $B(\e)\to\infty$.) We also choose $B(\e)\to\infty$ possibly even more slowly depending on $\beta_{2}$ such that the last term on the \abbr{RHS} of \eqref{eq:coupling3II9} is $=\mathrm{o}(1)\times i_{5}(\langle\Phi^{\e}(s,\cdot)\rangle_{\mathscr{C}^{10}})$. So, we deduce
\begin{align}
i_{6}(\langle\Phi^{\e}(s,\cdot)\rangle_{\mathscr{C}^{10}})^{-1}\times\sup_{0\leq t\leq\mathfrak{t}}|\mathbf{L}^{s}(t)-\mathbf{L}^{s,\sim}(t)| \ \lesssim \ \mathrm{o}(1)
\end{align}
with probability $1-\mathrm{o}(1)$. Choosing $\beta(\e)\to0$ to be the $\mathrm{o}(1)$ term on the \abbr{RHS} of this display shows that the lower bound defining $\tau_{\mathrm{ap},2}$ (see before \eqref{eq:coupling3II1}) is never realized except for on an event of probability $\mathrm{o}(1)$. Equivalently, the \abbr{RHS} of \eqref{eq:coupling3II6} vanishes as $\e\to0$. Therefore, \eqref{eq:coupling3II5} follows. As noted right before \eqref{eq:coupling3II5}, this completes the proof.
\end{proof}
\subsection{Proof of Proposition \ref{prop:coupling}}
Recall the stopping times $k\mapsto\tau_{k,\wedge}\wt{\tau}_{k,\wedge}$ from Lemma \ref{lemma:coupling3}. We introduce the following modification of $t\mapsto\mathbf{z}^{\e}(s+t),\wt{\mathbf{z}}^{\e}(s,t)$ (which are defined via $t\mapsto\mathbf{x}^{\e}(s,t),\wt{\mathbf{x}}^{\e}(s,t)$, respectively, at the beginning of this section). We let $t\mapsto\mathbf{w}(s,t)$ be the continuous-time Poisson process such that $\mathbf{w}(s,0)=\mathbf{x}^{\e}(s,0)=\mathbf{z}^{\e}(s)$, whose Poisson clocks have speed $\e^{-1}$, and whose embedded Markov chain is $k\mapsto\mathbf{x}^{\e}(s,\tau_{k,\wedge})$. Let $t\mapsto\wt{\mathbf{w}}(s,t)$ be the continuous-time Poisson process such that $\wt{\mathbf{w}}^{\e}(s,0)=\wt{\mathbf{x}}^{\e}(s,0)=\mathbf{x}^{\e}(s,0)=\mathbf{z}^{\e}(s,0)=\wt{\mathbf{z}}^{\e}(s,0)$, whose Poisson clocks have speed of $\e^{-1}$, and whose embedded Markov chain is given by $k\mapsto\wt{\mathbf{x}}^{\e}(s,\wt{\tau}_{k,\wedge})$. The clocks in all Poisson processes are the same (so $t\mapsto\mathbf{z}^{\e}(s+t),\wt{\mathbf{z}}^{\e}(s,t),\mathbf{w}(s,t),\wt{\mathbf{w}}(s,t)$ all jump at the same times).

We start with the following preliminary claim. There exists $\beta(\e)\to0$ as $\e\to0$ such that 
\begin{align}
\mathbf{P}\Big\{\sup_{0\leq t\leq\tau(\e)}|\mathbf{w}(s,t)-\mathbf{z}^{\e}(s+t)|+|\wt{\mathbf{w}}(s,t)-\wt{\mathbf{z}}^{\e}(s,t)| = 0\Big\} \ \geq \ 1-\beta(\e)\times i(\langle\Phi^{\e}(s,\cdot)\rangle_{\mathscr{C}^{10}}). \label{eq:coupling1}
\end{align}
In words, $t\mapsto\mathbf{w}(s,t),\mathbf{z}^{\e}(s+t)$ agree for all $0\leq t\leq\tau(\e)$ outside an event with probability at most $\beta(\e)\times i(\langle\Phi^{\e}(s,\cdot)\rangle_{\mathscr{C}^{10}})$. (And the same for $t\mapsto\wt{\mathbf{w}}(s,t),\wt{\mathbf{z}}^{\e}(s,t)$.) To see \eqref{eq:coupling1}, note that the embedded Markov chains for both $t\mapsto\mathbf{w}(s,t),\mathbf{z}^{\e}(s+t)$ are the same until we hit a number of steps $k$ such that $\tau_{k,\wedge}\neq\tau_{k}$. Equivalently (see Lemma \ref{lemma:coupling3}), we hit a number of steps $k$ in the embedded chain so that $\tau_{k,\wedge}=\e^{-1}\tau(\e)\delta(\e)B(\e)$, where $B(\e)\to\infty$ as $\e\to0$. First, according to the upper bound \eqref{eq:coupling3I}, we know that $\tau_{k,\wedge}\lesssim t(\e)k\delta(\e)\times i(\langle\Phi^{\e}(s,\cdot)\rangle_{\mathscr{C}^{10}})$ with probability at least $1-\gamma(\e)$ (where $\gamma(\e)\to0$ and $t(\e)\to\infty$ as slowly as we want as $\e\to0$). By the same reasoning for $t\mapsto\wt{\mathbf{w}}(s,t),\wt{\mathbf{z}}^{\e}(s,t)$, we ultimately deduce that the \abbr{LHS} of \eqref{eq:coupling1} is bounded below by $1-\mathbf{P}-\gamma(\e)$, where $\mathbf{P}$ is the probability that a Poisson random variable $\mathfrak{p}$ of intensity $\e^{-1}\tau(\e)$ (which counts ringings before time $\tau(\e)$) satisfies $t(\e)\mathfrak{p}\delta(\e)\times i(\langle\Phi^{\e}(s,\cdot)\rangle_{\mathscr{C}^{10}})\geq\e^{-1}\tau(\e)\delta(\e)B(\e)$, or equivalently, $\mathfrak{p}\geq\e^{-1}\tau(\e)B(\e)t(\e)^{-1}i(\langle\Phi^{\e}(s,\cdot)\rangle_{\mathscr{C}^{10}})^{-1}$. But $B(\e)\to\infty$, so if we take $t(\e)=B(\e)^{1/2}$, for example, we deduce that $\mathbf{P}\leq\beta(\e)\times i(\langle\Phi^{\e}(s,\cdot)\rangle_{\mathscr{C}^{10}})$ (since $\mathfrak{p}$ has intensity $\e^{-1}\tau(\e)$), where $\beta(\e)\to0$ as $\e\to0$. This completes the proof of \eqref{eq:coupling1}. Thus, we are left to compare $t\mapsto\mathbf{w}(s,t),\wt{\mathbf{w}}(s,t)$. Put more precisely, because all ``particle processes" live in the fixed compact set $\bUz$, we have the following in which $\tau\geq s+\tau(\e)$ is a stopping time (see the statement of Proposition \ref{prop:coupling}) and $\mathscr{E}$ is the event on the \abbr{LHS} of \eqref{eq:coupling1}:
\begin{align}
&\sup_{0\leq t\leq\tau(\e)}\E\left\{|\mathbf{z}^{\e}(s+t)-\wt{\mathbf{z}}^{\e}(s,t)|\times i(\langle\Phi^{\e}\rangle_{\tau})^{-1}\right\} \nonumber\\
\leq &\sup_{0\leq t\leq\tau(\e)}\E\left\{(\mathbf{1}_{\mathscr{E}}+\mathbf{1}_{\mathscr{E}^{C}})|\mathbf{z}^{\e}(s+t)-\wt{\mathbf{z}}^{\e}(s,t)|\times i(\langle\Phi^{\e}\rangle_{\tau})^{-1}\right\} \label{eq:coupling2a}\\
\lesssim &\sup_{0\leq t\leq\tau(\e)}\E\left\{\mathbf{1}_{\mathscr{E}}|\mathbf{z}^{\e}(s+t)-\wt{\mathbf{z}}^{\e}(s,t)|\times i(\langle\Phi^{\e}\rangle_{\tau})^{-1}\right\} + \beta(\e) \label{eq:coupling2b}\\
\leq &\sup_{0\leq t\leq\tau(\e)}\E\left\{|\mathbf{w}^{\e}(s,t)-\wt{\mathbf{w}}^{\e}(s,t)|\times i(\langle\Phi^{\e}\rangle_{\tau})^{-1}\right\}+\beta(\e). \label{eq:coupling2c}
\end{align}
To estimate the first term in \eqref{eq:coupling2c}, we first let $\mathfrak{p}_{t}$ be a Poisson random variable of intensity $\e^{-1}t$ for $0\leq t\leq\tau(\e)$. By what it means to be a Poisson process with embedded chain and because $i(\cdot)$ is increasing, we can write
\begin{align}
&\E\left\{|\mathbf{w}^{\e}(s,t)-\wt{\mathbf{w}}^{\e}(s,t)|\times i(\langle\Phi^{\e}\rangle_{\tau})^{-1}\right\} \nonumber\\
&\leq \sum_{k\geq0}\mathbf{P}\{\mathfrak{p}_{t}=k\} \E\left\{|\mathbf{x}^{\e}(s,\tau_{k,\wedge})-\wt{\mathbf{x}}^{\e}(s,\wt{\tau}_{k,\wedge})|\times i(\langle\Phi^{\e}(s,\cdot)\rangle_{\mathscr{C}^{10}})^{-1}\right\}. \label{eq:coupling3}
\end{align}
Next, we use the triangle inequality. For any $k\geq0$, we deduce that 
\begin{align}
&\E\left\{|\mathbf{x}^{\e}(s,\tau_{k,\wedge})-\wt{\mathbf{x}}^{\e}(s,\wt{\tau}_{k,\wedge})|\times i(\langle\Phi^{\e}(s,\cdot)\rangle_{\mathscr{C}^{10}})^{-1}\right\} \nonumber\\
&\leq \ \E\left\{|\mathbf{x}^{\e}(s,\tau_{k,\wedge})-\wt{\mathbf{x}}^{\e}(s,\tau_{k,\wedge})|\times i(\langle\Phi^{\e}(s,\cdot)\rangle_{\mathscr{C}^{10}})^{-1}\right\}\label{eq:coupling4a}\\
&+ \ \E\left\{|\wt{\mathbf{x}}^{\e}(s,\tau_{k,\wedge})-\wt{\mathbf{x}}^{\e}(s,\wt{\tau}_{k,\wedge})|\times i(\langle\Phi^{\e}(s,\cdot)\rangle_{\mathscr{C}^{10}})^{-1}\right\}. \label{eq:coupling4b}
\end{align}
Let $\mathscr{E}_{1}$ denote the event where $|\mathbf{x}^{\e}(s,\tau_{k,\wedge})-\wt{\mathbf{x}}^{\e}(s,\tau_{k,\wedge})|\lesssim\alpha(\e)\times i(\langle\Phi^{\e}(s,\cdot)\rangle_{\mathscr{C}^{10}})$. By construction in Lemma \ref{lemma:coupling3}, we know $\tau_{k,\wedge}\leq\e^{-1}\tau(\e)\delta(\e)B(\e)$ with probability 1. Thus, by Lemma \ref{lemma:coupling2}, we know $\mathbf{P}[\mathscr{E}_{1}]\geq1-\gamma(\e)$, where $\gamma(\e)\to0$ as $\e\to0$. Since $\mathbf{x}^{\e},\wt{\mathbf{x}}^{\e}$ are uniformly bounded (they belong to the compact set $\bUz$), a union bound argument gives
\begin{align}
\mathrm{\abbr{RHS}}\eqref{eq:coupling4a} \ \leq \ \E\left\{(\mathbf{1}_{\mathscr{E}_{1}}+\mathbf{1}_{\mathscr{E}_{1}^{C}})|\mathbf{x}^{\e}(s,\tau_{k,\wedge})-\wt{\mathbf{x}}^{\e}(s,\tau_{k,\wedge})|\times i(\langle\Phi^{\e}(s,\cdot)\rangle_{\mathscr{C}^{10}})^{-1}\right\} \ \lesssim \ \alpha(\e)+\gamma(\e). \label{eq:coupling5}
\end{align}
Let $\mathscr{E}_{2}$ denote the event where the $\mathscr{C}^{1/3}$-norm of $t\mapsto\wt{\mathbf{x}}^{\e}(s,t)$ is bounded above by $h(\e)\times\wt{i}(\langle\Phi^{\e}(s,\cdot)\rangle_{\mathscr{C}^{10}})$, where $h(\e)\to\infty$ as $\e\to0$ slowly (to be specified shortly), and $\wt{i}(\cdot)\geq1$ is an increasing, continuous function (specified shortly as well), and $0\leq t\leq\e^{-1}\tau(\e)\delta(\e)B(\e)$. By Lemma \ref{lemma:sde}, if we take $B(\e)\to\infty$ sufficiently slowly (depending on $h(\e)$), we know $\mathbf{P}[\mathscr{E}_{2}]\geq1-\upsilon(\e)$, where $\upsilon(\e)\to0$ as $\e\to0$. Moreover, by definition of $\mathscr{E}_{2}$, on the event $\mathscr{E}_{2}$, we know $|\wt{\mathbf{x}}^{\e}(s,\tau_{k,\wedge})-\wt{\mathbf{x}}^{\e}(s,\wt{\tau}_{k,\wedge})|\lesssim h(\e)|\tau_{k,\wedge}-\wt{\tau}_{k,\wedge}|^{1/3}$. Now, let $\mathscr{E}_{3}$ denote the event $|\tau_{k,\wedge}-\wt{\tau}_{k,\wedge}|\lesssim\alpha(\e)\times i(\langle\Phi^{\e}(s,\cdot)\rangle_{\mathscr{C}^{10}})$. By Lemma \ref{lemma:coupling3}, we know $\mathbf{P}[\mathscr{E}_{3}]\geq1-\gamma(\e)$, where $\gamma(\e)\to0$ as $\e\to0$. Using all of this, as well as the bound $|\wt{\mathbf{x}}^{\e}|\lesssim1$ (since $\wt{\mathbf{x}}^{\e}\in\bUz$), we ultimately deduce the following sequence of calculations if we choose $\wt{i}(\cdot)\geq1$ depending only on $i(\cdot)$:
\begin{align*}
\eqref{eq:coupling4b} \ \leq \ &\E\left\{(\mathbf{1}_{\mathscr{E}_{2}}+\mathbf{1}_{\mathscr{E}_{2}^{C}})|\wt{\mathbf{x}}^{\e}(s,\tau_{k,\wedge})-\wt{\mathbf{x}}^{\e}(s,\wt{\tau}_{k,\wedge})|\times i(\langle\Phi^{\e}(s,\cdot)\rangle_{\mathscr{C}^{10}})^{-1}\right\} \\
\lesssim \ &\E\left\{\mathbf{1}_{\mathscr{E}_{2}}|\wt{\mathbf{x}}^{\e}(s,\tau_{k,\wedge})-\wt{\mathbf{x}}^{\e}(s,\wt{\tau}_{k,\wedge})|\times i(\langle\Phi^{\e}(s,\cdot)\rangle_{\mathscr{C}^{10}})^{-1}\right\} + \upsilon(\e) \\
\lesssim \ &\E\left\{\mathbf{1}_{\mathscr{E}_{2}}(\mathbf{1}_{\mathscr{E}_{3}}+\mathbf{1}_{\mathscr{E}_{3}^{C}})|\wt{\mathbf{x}}^{\e}(s,\tau_{k,\wedge})-\wt{\mathbf{x}}^{\e}(s,\wt{\tau}_{k,\wedge})|\times i(\langle\Phi^{\e}(s,\cdot)\rangle_{\mathscr{C}^{10}})^{-1}\right\} + \upsilon(\e) \\
\lesssim \ &\E\left\{\mathbf{1}_{\mathscr{E}_{2}}\mathbf{1}_{\mathscr{E}_{3}}|\wt{\mathbf{x}}^{\e}(s,\tau_{k,\wedge})-\wt{\mathbf{x}}^{\e}(s,\wt{\tau}_{k,\wedge})|\times i(\langle\Phi^{\e}(s,\cdot)\rangle_{\mathscr{C}^{10}})^{-1}\right\} + \gamma(\e) + \upsilon(\e) \\
\lesssim \ &h(\e)\alpha(\e)^{1/3}+\gamma(\e)+\upsilon(\e).
\end{align*}
If we choose $h(\e)=\alpha(\e)^{-1/6}$ (recall from the previous paragraph that $h(\e)\to\infty$ is for our choosing), we deduce $\eqref{eq:coupling4b}\to0$ as $\e\to0$. (Again, whatever choice we make for $h(\e)$ may force us to choose $B(\e)\to\infty$ in Lemma \ref{lemma:coupling1} to happen more slowly, but this is not an issue.) Combining this with \eqref{eq:coupling2a}-\eqref{eq:coupling2c}, \eqref{eq:coupling3}, \eqref{eq:coupling4a}-\eqref{eq:coupling4b} and \eqref{eq:coupling5} yields \eqref{eq:couplingI}. 

We move to proving $\E[\|\wt{\mathsf{F}}^{\e,\mathrm{av}}-\mathsf{H}^{\e}\|_{\tau}i(\langle\Phi^{\e}\rangle_{\tau})^{-2}]\lesssim\zeta(\e)$ for $\zeta(\e)\to0$ as $\e\to0$. To this end, we start with a deterministic estimate. If $0\leq t\leq\tau(\e)$, then by construction, we have $\wt{\mathsf{F}}^{\e,\mathrm{av}}(t,\cdot)=\mathsf{H}^{\e}(t,\cdot)=0$. For $\tau(\e)\leq t\leq\tau$, we know $\wt{\mathsf{F}}^{\e,\mathrm{av}}(t,x)-\mathsf{H}^{\e}(t,x)$ is equal to
\begin{align}
\int_{0}^{t-\tau(\e)}\d s \ \tau(\e)^{-1}\int_{0}^{\tau(\e)}\d r \int_{\partial\bUz}\left\{\mathbf{T}^{\Phi^{\e}(s,\cdot),\delta}_{\mathbf{z}^{\e}(s+r)}(\d y)-\mathbf{T}^{\Phi^{\e}(s,\cdot),\delta}_{\wt{\mathbf{z}}^{\e}(s,r)}(\d y)\right\}\mathscr{K}(x,y)\mathsf{n}^{\Phi^{\e}(s,\cdot)}(y). \label{eq:coupling6}
\end{align}
Lemma \ref{lemma:freezeprelim} now implies the following estimate given \eqref{eq:coupling6}, where $\alpha(\e)\to0$ as $\e\to0$:
\begin{align*}
\|\wt{\mathsf{F}}^{\e,\mathrm{av}}(t,\cdot)-\mathsf{H}^{\e}(t,\cdot)& \|_{\mathscr{C}^{10}}&\\
\lesssim \ &\int_{0}^{t-\tau(\e)}\d s \ \tau(\e)^{-1}\int_{0}^{\tau(\e)}\d r \left\{|\mathbf{z}^{\e}(s+r)-\wt{\mathbf{z}}^{\e}(s,r)|+\alpha(\e)\right\}\times i(\langle\Phi^{\e}(s,\cdot)\rangle_{\mathscr{C}^{10}}) \\
\lesssim \ &i(\langle\Phi^{\e}\rangle_{\tau})\int_{0}^{t-\tau(\e)}\d s \ \tau(\e)^{-1}\int_{0}^{\tau(\e)}\d r \left\{|\mathbf{z}^{\e}(s+r)-\wt{\mathbf{z}}^{\e}(s,r)|+\alpha(\e)\right\}.
\end{align*}
We can extend the outer integration domain from $[0,t-\tau(\e)]$ to $[0,\mathfrak{t}]$ for some deterministic time $\mathfrak{t}\lesssim1$ (because $t\leq\tau\lesssim1$ by assumption). The resulting bound is independent of the $t$-variable in the previous display, so this would give an upper bound for $\|\wt{\mathsf{F}}^{\e,\mathrm{av}}-\mathsf{H}^{\e}\|_{\tau}$. Ultimately, by applying the triangle inequality and \eqref{eq:couplingI}, we deduce the following estimate (since $\mathfrak{t}\lesssim1$):
\begin{align}
\E[\|\wt{\mathsf{F}}^{\e,\mathrm{av}}-\mathsf{H}^{\e}\|_{\tau}i(\langle\Phi^{\e}\rangle_{\tau})^{-2}] \ &\lesssim \ \alpha(\e)+\sup_{0\leq t\leq\tau(\e)}\E\left\{|\mathbf{z}^{\e}(s+t)-\wt{\mathbf{z}}^{\e}(s,t)|\times i(\langle\Phi^{\e}\rangle_{\tau})^{-1}\right\} \nonumber  \\
&\lesssim \ \alpha(\e)+\omega(\e).
\end{align}
By choosing $\zeta(\e)=\alpha(\e)+\omega(\e)$, the proposed estimate $\E[\|\wt{\mathsf{F}}^{\e,\mathrm{av}}-\mathsf{H}^{\e}\|_{\tau}i(\langle\Phi^{\e}\rangle_{\tau})^{-2}]\lesssim\zeta(\e)$ follows. \qed
\appendix
\section{Technical points about reflecting \abbr{SDE}s}\label{section:appendix}
We record elementary properties of reflecting Brownian motions on compact (smooth) Riemannian manifolds. To set it up, consider a process $t\mapsto(\mathbf{x}(t),\Phi(t))$ valued in $\bUz\times\mathscr{C}^{10}_{\simeq}(\partial\bUz,\R^{\d})$ as follows. First, following the notation of the beginning of Section \ref{section:couplingproof}, suppose $\mathbf{x}(\cdot)$ solves the following \abbr{SDE}, where $\mathbf{L}$ is the boundary local time (on $\partial\bUz$) of $\mathbf{x}$ and $\mathbf{b}(\cdot)$ is a standard $\R^{\d}$-valued Brownian motion:
\begin{align}
\d\mathbf{x}(t) \ = \ \mathbf{A}^{\Phi(t)}(\mathbf{x}(t))\d\mathbf{b}(t) + \mathbf{m}^{\Phi(t)}(\mathbf{x}(t))\d t + \mathbf{n}(\mathbf{x}(t))\d\mathbf{L}(t).
\end{align}
Suppose also that $\Phi(\cdot)$ is an adapted Poisson process (with clocks of a constant, finite speed).
\begin{lemma}\label{lemma:sde}
The process $t\mapsto(\mathbf{x}(t),\Phi(t))$ has a unique strong solution for all $t\geq0$ (upon fixing initial condition $(\mathbf{x}(0),\Phi(0))$). 

Now, fix $\beta\in[0,1/2)$. There exists a function $j(\cdot,\cdot)$ that is jointly continuous and increasing and a function $h(\cdot)$ that is also continuous and increasing such that given any $\tau>0$ and $\alpha>0$, with probability at least $1-\alpha$, we have
\begin{align}
\sup_{s\neq t\in[0,\tau]}\frac{|\mathbf{x}(t)-\mathbf{x}(s)|}{|t-s|^{\beta}} \ \lesssim_{\beta} \ h(\alpha^{-1})\times \sup_{0\leq t\leq\tau}j(\tau,\langle\Phi(t)\rangle_{\mathscr{C}^{10}}).
\end{align}
Now, suppose $\Phi(t)=\Phi(0)$ and $\mathbf{x}(0)\in\partial\bUz$. Fix $\xi\geq0$, and set $\tau^{\xi}=\inf\{t\geq0: \mathbf{L}(t)=\xi\}$. Then $\tau^{\xi}<\infty$ almost surely.
\end{lemma}
\begin{remark}\label{remark:sde}
The assumption $\mathbf{x}(0)\in\partial\bUz$ in the last claim can be removed. Indeed, if not, then wait for $\mathbf{x}$ to hit the boundary $\partial\bUz$ and start counting the boundary local time at that time.
\end{remark}
\begin{proof}
Suppose $\Phi(t)=\Phi(0)$ (so $\Phi(0)$ is constant and deterministic once we condition on the initial data). In this case, existence and uniqueness follows by the classical localization-plus-gluing algorithm below.
\begin{itemize}
\item Strong existence and uniqueness of solutions in any local coordinate chart (until the exit time for this chart) follows by standard Ito theory. Indeed, the coefficients $\mathbf{A}^{\Phi(t)}(\mathbf{x})$ and $\mathbf{m}^{\Phi(t)}(\mathbf{x})$ are uniformly smooth in $\mathbf{x}$ because $\bUz$ is compact, so standard Ito estimates apply. We only need to be careful about the local time $\mathbf{L}$. Note this term is supported at the boundary. In the collar $\mathbf{C}_{2}$, under the coordinates $\mathbf{C}_{2}\simeq\partial\bUz\times[0,2]$, by definition of the metric $\mathsf{g}^{\Phi}$, the \abbr{SDE} for $\mathbf{x}(t)$ pushes forward to standard reflecting Brownian motion on $[0,\infty)$ (which is just the absolute value of standard $\R$-valued Brownian motion) and Brownian motion on $\partial\bUz$ (with respect to a Riemannian metric that depends smoothly on the $[0,2]$-coordinate). As $\partial\bUz$ has no boundary, we can use the same local coordinate chart method but for Brownian motion on $\partial\bUz$ without the issue of a local time.
\item As a consequence, we have local existence and uniqueness with respect to a sequence of stopping times given by exit times of local coordinate charts. It suffices to show that for an appropriate choice of local coordinate charts (on which we solve classical Euclidean \abbr{SDE}s), these stopping times diverge with probability 1. We can choose local coordinate charts so that to leave any new coordinate chart, the process $\mathbf{x}$ must travel Euclidean distance $\mathrm{r}(\bUz)$ (for $\mathrm{r}(\bUz)>0$ depending only on $\bUz$ since $\bUz$ is compact). So, it suffices to show that the sequence of escape times from balls of radius $\geq\mathrm{r}(\bUz)>0$ for a standard Ito \abbr{SDE} (with uniformly bounded, smooth coefficients) diverges with probability 1. This is a standard result proved by escape probability estimates and Borel-Cantelli.
\end{itemize}
Thus, we have global existence and uniqueness if $\Phi(t)=\Phi(0)$. For general $\Phi$, we note that the argument for global existence and uniqueness requires only $\langle\Phi(0)\rangle_{\mathscr{C}^{10}}<\infty$. Since we assume $\langle\Phi(t)\rangle_{\mathscr{C}^{10}}<\infty$ for all $t\geq0$, we can further localize. Precisely, we can find a unique strong solution for $\mathbf{x}$ between any two jump times of $\Phi$ (note that the jump times are stopping times). We then glue these solutions and recall the Poisson clocks for $\Phi(\cdot)$ have constant speed, so the sequence of its jump times diverge (this gives global solutions).

Let us now show the Holder bound. We first write
\begin{align}
\mathbf{x}(t) \ = \ \mathbf{x}(0) + \int_{0}^{t}\mathbf{A}^{\Phi(s)}(\mathbf{x}(s))\d\mathbf{b}(s) + \int_{0}^{t}\mathbf{m}^{\Phi(s)}(\mathbf{x}(s))\d s  + \int_{0}^{t}\mathbf{n}(\mathbf{x}(s))\d\mathbf{L}(s). \label{eq:sdesde}
\end{align}
Holder regularity of the first two integrals on the \abbr{RHS} is classical (by Ito estimates and by uniform boundedness of $\mathbf{A}^{\Phi(s)}$ and $\mathbf{m}^{\Phi(s)}$; again, the jumps in $\Phi(s)$ are stopping times, so the Brownian integral is an honest Ito integral). For the local time integral, note that for any $t_{1}\leq t_{2}$, we have the following since $\mathbf{n}$ is unit-length and the measure $\d\mathbf{L}$ is non-negative:
\begin{align}
|\int_{t_{1}}^{t_{2}}\mathbf{n}(\mathbf{x}(s))\d\mathbf{L}(s)| \ \leq \ \int_{t_{1}}^{t_{2}}\d\mathbf{L}(s) \ = \ \mathbf{L}(t_{2})-\mathbf{L}(t_{1}).
\end{align}
Thus, controlling Holder regularity of the last term in \eqref{eq:sdesde} amounts to controlling Holder regularity of local time itself. Again, under collar coordinates $\mathbf{C}_{2}\simeq\partial\bUz\times[0,2]$, the process $\mathbf{x}$ pushes forward to a standard one-dimensional reflecting Brownian motion and a $\partial\bUz$-\abbr{SDE} (without reflection). Thus, the local time $\mathbf{L}$ is the local time of a standard one-dimensional Brownian motion. The Holder-$1/2^{-}$ property of this local time is classical. (In a nutshell, said local time can be written as the sum of an Ito integral and the absolute value of a Brownian motion, so its Holder regularity follows immediately.) This gives Holder-$1/2^{-}$ regularity of $\mathbf{x}(\cdot)$ with probability 1.

It remains to show the almost sure finiteness of $\tau^{\xi}$ for any $\xi\geq0$. It suffices to show that for any $m\geq0$, we have $\E_{0}(\tau^{\xi}\wedge m)\lesssim_{\Phi(0),\bUz}1+\xi$, where  $\E_{0}$ means expectation conditioning on $(\Phi(0),\mathbf{x}(0))$. (It is crucial that this bound does not depend on $m$.) Let $\mathscr{U}$ solve the following Dirichlet problem:
\begin{align}
\Delta_{\mathsf{g}^{\Phi(0)}}\mathscr{U}(x) \ = \ 1 \quad\mathrm{and}\quad \mathscr{U}(x)|_{\partial\bUz} \ = \ 0.
\end{align}
Existence and regularity of $\mathscr{U}$ follows by elliptic theory. By Ito, the \abbr{PDE} for $\mathscr{U}$, and the assumption $\mathbf{x}(0)\in\partial\bUz$, exactly as in \eqref{eq:coupling3I1}, we have the following identity (for which we recall $\Phi(t)=\Phi(0)$ for all $t$, so the generator acting on $\mathscr{U}$ is always $\equiv1$):
\begin{align}
\E_{0}\mathscr{U}(\mathbf{x}(\tau^{\xi}\wedge m)) \ = \ \E_{0}\{\tau^{\xi}\wedge m\}+\E_{0}\int_{0}^{\tau^{\xi}\wedge m}\{\mathbf{n}\cdot\grad\mathscr{U}\}(\mathbf{x}(r))\d\mathbf{L}(r).
\end{align}
By elliptic regularity, we know that $\mathscr{U}$ is smooth with derivatives bounded in terms of $\Phi(0),\bUz$. In particular, by this and by rearranging the previous display, we get
\begin{align}
\E_{0}\{\tau^{\xi}\wedge m\} \ \lesssim_{\Phi(0),\bUz} 1 + \E_{0}\int_{0}^{\tau^{\xi}\wedge m}\d\mathbf{L}(r) \ = \ 1 + \E_{0}\mathbf{L}(\tau^{\xi}\wedge m) \ \leq \ 1+\E_{0}\mathbf{L}(\tau^{\xi}) \ = \ 1+\xi.
\end{align}
Above, we again used that the measure $\d\mathbf{L}$ is non-negative (or equivalently, $\mathbf{L}$ is non-decreasing). The last identity follows by definition of $\tau^{\xi}$ and almost sure continuity of local time. This completes the proof.
\end{proof}



\end{document}